\documentclass[11pt,reqno]{amsart}

\usepackage{geometry}
\usepackage[T1]{fontenc}
\usepackage{newtxtext}   
\usepackage[hidelinks]{hyperref} 
\hypersetup{colorlinks=true,linkcolor=red,citecolor=blue}

\usepackage{breakcites}

\usepackage{algorithm}
\usepackage{algpseudocode}
\usepackage{graphicx}
\usepackage{dsfont}
\usepackage{setspace}
\usepackage{enumitem}
\usepackage[dvipsnames]{xcolor}
\usepackage{graphicx}
\usepackage{epstopdf}
\usepackage{csquotes}
\usepackage{mathtools}
\usepackage{amssymb, amsmath,amsthm, amsfonts}
\usepackage{ifthen}

\usepackage{caption}

\usepackage{verbatim}
\usepackage{pifont}
\usepackage{bm}  
\usepackage{stackengine}
\usepackage{bbm}

\allowdisplaybreaks
\usepackage{accents}
\usepackage[capitalise]{cleveref}
\usepackage{xpatch}
\makeatletter
\xpatchcmd{\@thm}{\thm@headpunct{.}}{\thm@headpunct{}}{}{}

\makeatother

\makeatother

\newtheorem{definition}{Definition}[section]
\newtheorem{assumption}{Assumption}

\newtheorem{theorem}{Theorem}[section]
\newtheorem{lemma}{Lemma}[section]
\newtheorem{remark}{Remark}[section]
\newtheorem{corollary}{Corollary}[section]
\newtheorem{proposition}{Proposition}[section]
\newtheorem{example}{Example}

\usepackage{graphicx}
\usepackage{subcaption}
\usepackage{enumitem}
\usepackage{hyperref}
\usepackage{url}
\usepackage{wrapfig,framed}
\usepackage{setspace}
\usepackage{graphicx}
\usepackage{epstopdf}
\usepackage{csquotes}
\usepackage{mathtools}
\usepackage{amssymb, amsmath,amsthm, amsfonts}
\usepackage{ifthen}
\usepackage{verbatim}
\usepackage{pifont}
\usepackage{bm}  
\usepackage{stackengine}
\usepackage{bbm}

\allowdisplaybreaks
\usepackage{accents}

\usepackage{xpatch}
\makeatletter
\xpatchcmd{\@thm}{\thm@headpunct{.}}{\thm@headpunct{}}{}{}
\makeatother

\usepackage{tikz}
\usetikzlibrary{positioning}

\newcommand{\bsigma}{\mathbf{\Sigma}}

\newcommand{\blambda}{\mathbf{\Lambda}}

\DeclareMathOperator{\id}{id}
\DeclareMathOperator{\vectorize}{vec}

\DeclareMathOperator{\argmin}{argmin}

\DeclareMathOperator{\supp}{spt}

\newcommand{\curve}[1]{(#1)_{t\in[0,1]}}
\newcommand{\R}{\mathbb{R}}

\newcommand\numberthis{\addtocounter{equation}{1}\tag{\theequation}}

\newcommand{\Wp}{\mathsf{W}_p}

\newcommand{\bA}{\mathbf{A}}
\newcommand{\bB}{\mathbf{B}}

\newcommand{\bH}{\mathbf{H}}
\newcommand{\bI}{\mathbf{I}}

\newcommand{\bK}{\mathbf{K}}
\newcommand{\bL}{\mathbf{L}}
\newcommand{\bM}{\mathbf{M}}

\newcommand{\bO}{\mathbf{O}}
\newcommand{\bP}{\mathbf{P}}
\newcommand{\bQ}{\mathbf{Q}}

\newcommand{\bS}{\mathbf{S}}

\newcommand{\bU}{\mathbf{U}}

\newcommand{\bX}{\mathbf{X}}

\newcommand{\cB}{\mathcal{B}}

\newcommand{\cD}{\mathcal{D}}

\newcommand{\cH}{\mathcal{H}}
\newcommand{\cI}{\mathcal{I}}
\newcommand{\cJ}{\mathcal{J}}

\newcommand{\cL}{\mathcal{L}}

\newcommand{\cN}{\mathcal{N}}
\newcommand{\cO}{\mathcal{O}}
\newcommand{\cP}{\mathcal{P}}

\newcommand{\cX}{\mathcal{X}}
\newcommand{\cY}{\mathcal{Y}}

\newcommand{\sC}{\mathsf{C}}
\newcommand{\sD}{\mathsf{D}}
\newcommand{\sd}{\mathsf{d}}

\newcommand{\sF}{\mathsf{F}}

\newcommand{\sH}{\mathsf{H}}

\newcommand{\sI}{\mathsf{I}}

\newcommand{\sR}{\mathsf{R}}

\newcommand{\sV}{\mathsf{V}}
\newcommand{\sW}{\mathsf{W}}

\newcommand{\NN}{\mathbb{N}}

\newcommand{\RR}{\mathbb{R}}

\newcommand{\op}{\mathrm{op}}
\newcommand{\F}{\mathrm{F}}
\newcommand{\lip}{\mathrm{Lip}}

\newcommand{\llangle}{\left\langle}
\newcommand{\rrangle}{\right\rangle}

\newcommand{\Od}{\mathrm{O}(d)}
\newcommand{\SOd}{\mathrm{SO}(d)}

\newcommand{\tr}{\mathsf{Tr}}

\newcommand{\dom}{\mathsf{Dom}}
\newcommand{\spec}{\mathsf{spec}}

\newcommand{\ac}{\mathrm{ac}}
\newcommand{\GW}{\mathsf{GW}}
\newcommand{\IGW}{\mathsf{IGW}}
\newcommand{\grad}{\mathsf{grad}}

\newcommand{\W}{\mathsf{W}}

\newcommand{\MMD}{\mathsf{MMD}}

\definecolor{darkblue}{rgb}{0.0,0.0,0.66}  
\definecolor{darkred}{rgb}{100,0.0,0.0} 
\hypersetup{colorlinks=true,linkcolor=darkred,citecolor=darkblue}


\begin{document}

\title[Gradient Flows and Riemannian Structure in the Gromov-Wasserstein Geometry]{Gradient Flows and Riemannian Structure \\in the Gromov-Wasserstein Geometry}

\author[Z. Zhang]{Zhengxin Zhang}
\address[Z. Zhang]{Center for Applied Mathematics, Cornell University}
\email{zz658@cornell.edu}

\author[Z. Goldfeld]{Ziv Goldfeld}
\address[Z. Goldfeld]{School of Electrical and Computer Engineering, Cornell University}
\email{goldfeld@cornell.edu}

\author[K. Greenewald]{Kristjan Greenewald}
\address[K. Greenewald]{MIT-IBM Watson AI Lab; IBM Research}
\email{kristjan.h.greenewald@ibm.com}

\author[Y. Mroueh]{Youssef Mroueh}
\address[Y. Mroueh]{IBM Research}
\email{mroueh@us.ibm.com}

\author[B. K. Sriperumbudur]{Bharath K. Sriperumbudur}
\address[B. K. Sriperumbudur]{Department of Statistics, Pennsylvania State University}
\email{bks18@psu.edu}


\maketitle

\begin{center}
\small Communicated by L\'ena\"ic Chizat
\end{center}

\begin{abstract}
The Wasserstein space of probability measures is known for its intricate Riemannian structure, which underpins the Wasserstein geometry and enables gradient flow algorithms. However, the Wasserstein geometry may not be suitable for certain tasks or data modalities. Motivated by scenarios where the global structure of the data needs to be preserved, this work initiates the study of gradient flows and Riemannian structure in the Gromov-Wasserstein (GW) geometry, which is particularly suited for such purposes. We focus on the inner product GW (IGW) distance between distributions on $\RR^d$, which preserves the angles within the data and serves as a convenient initial setting due to its analytic tractability. Given a functional $\sF:\cP_2(\RR^d)\to\RR$ to optimize and an initial distribution $\rho_0\in\cP_2(\RR^d)$, we present an implicit IGW minimizing movement scheme that generates a sequence of distributions $\{\rho_i\}_{i=0}^n$, which are close in IGW and aligned in the 2-Wasserstein sense. Taking the time step to zero, we prove that the (piecewise constant interpolation of the) discrete solution converges to an IGW generalized minimizing movement (GMM) $(\rho_t)_t$ that follows the continuity equation with a velocity field $v_t\in L^2(\rho_t;\RR^d)$, specified by a global transformation of the Wasserstein gradient of $\sF$ (viz., the gradient of its first variation). The transformation is given by a mobility operator that modifies the Wasserstein gradient to encode not only local information, but also global structure, as expected for the IGW gradient flow.
Our gradient flow analysis leads us to identify the Riemannian structure that gives rise to the intrinsic IGW geometry, using which we establish a Benamou-Brenier-like formula for IGW. We conclude with a formal derivation, akin to the Otto calculus, of the IGW gradient as the inverse mobility acting on the Wasserstein gradient. Numerical experiments demonstrating the global nature of IGW interpolations are provided to complement the~theory.
 \end{abstract}

\section{Introduction}

The Wasserstein gradient flow describes the evolution of probability measures along a trajectory that minimizes a given objective within the Wasserstein geometry. This concept was introduced in the seminal work of Jordan, Kinderlehrer, and Otto (JKO) \cite{jordan1998variational}, who demonstrated that the evolution of marginal distributions along the Langevin diffusion can be interpreted as a gradient flow of the Kullback-Leibler (KL) divergence over the 2-Wasserstein space $(\cP_2(\RR^d), \W_2)$.\footnote{Here and throughout, $\cP_2(\RR^d)$ represents the class of Borel probability measures with a bounded second absolute moment.} Since then, Wasserstein gradient flows have profoundly impacted various fields, including optimal transport (OT) \cite{ambrosio2005gradient, santambrogio2015optimal, carlier2017convergence}, partial differential equations (PDEs) \cite{otto2001geometry}, physics \cite{adams2011large, carrillo2022primal}, machine learning \cite{cheng2018convergence,frogner2020approximate,lin2021wasserstein,alvarez2021optimizing,bunne2022proximal}, and sampling \cite{bernton2018langevin, cheng2018convergence, wibisono2018sampling, lambert2022variational}. Advancements in Wasserstein gradient flows, in turn, revealed the Riemannian structure of the 2-Wasserstein space \cite{otto2001geometry}, whose tangent space $T_\mu\cP_2(\RR^d)$ at $\mu$ is given by (closure of) $\{\nabla\varphi:\,\varphi:\RR^d\to\RR^d\}$ endowed with the inner product of $L^2(\mu;\RR^d)$, which induces the~Riemannian~metric~tensor.

However, the Wasserstein geometry is not always the appropriate choice, no matter the application at hand. As a simple example, consider interpolating between a cat shape and its rotation, as illustrated in \cref{fig:GW_interpolation}. Interpolation methods based on Wasserstein geodesics (shown in second row), which are prevalent in computer vision for shape morphing tasks \cite{solomon2015convolutional,bonneel2018optimal,bonneel2019spot,zhang2022wassersplines}, do not fit the bill as they can introduce arbitrary deformations and distortion. This is because the Wasserstein geodesic displaces particles along the path that minimizes transportation cost, with no regard to preserving the global shape. Instead, the desired interpolation should behave as a rigid body OT that maintains the global structure of the data. As this work will demonstrate, a suitable choice of geometry over $\mathcal{P}_2(\mathbb{R}^d)$ for such tasks is induced by the Gromov-Wasserstein (GW) alignment problem. Our exploration begins with a JKO-like implicit scheme for GW gradient flows, which will reveal the structure of GW gradients and lead us to identify the Riemannian structure of GW spaces. Among other promising~applications, this approach enables the desired structure-preserving interpolations, as the first row of \cref{fig:GW_interpolation} shows.

\begin{figure}[t!]
    \centering
    \includegraphics[width=1\textwidth]{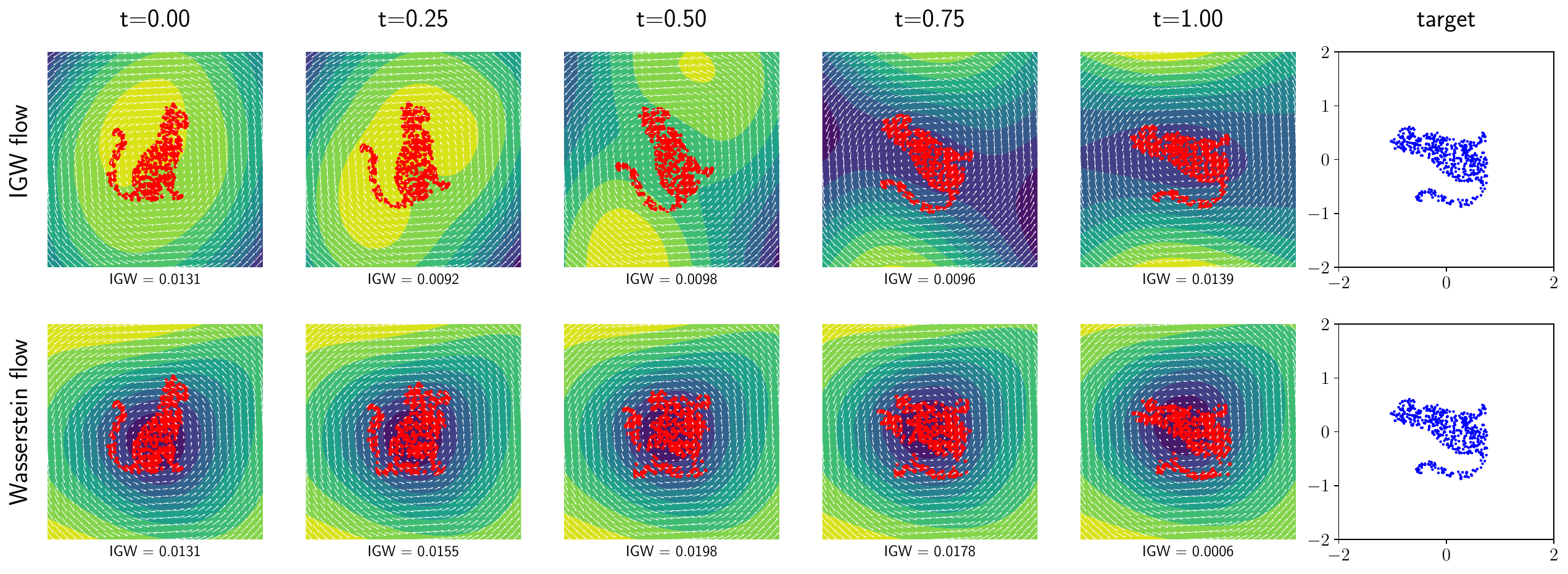}
    \caption{GW versus Wasserstein interpolation between rotated cat shapes: the Wasserstein interpolation (second line) breaks the structure of the shape to minimize the transportation cost, while the GW interpolation (first line) respects the structure and produces the desired effect. See \cref{sec:experiment} for details on this experiment.
    }\label{fig:GW_interpolation}
\end{figure}

\subsection{Gromov-Wasserstein Alignment}

Alignment of heterogeneous datasets varying in modality, location, or semantics, is fundamental to data science, spanning applications to language models \cite{alvarez2018gromov}, computer vision \cite{memoli2009spectral,xu2019gromov,xu2019scalable,koehl2023computing}, and genomics \cite{demetci2020gromov,blumberg2020mrec}. The GW distance, introduced by M\'emoli \cite{memoli2007use,memoli2011gromov} as a relaxation of the Gromov–Hausdorff distance, provides a mathematical framework for alignment by abstracting datasets into metric measure (mm) spaces and seeking to optimally match them. Specifically, the $(p,q)$-GW alignment cost between two mm spaces $(\cX,\mathsf{d}_\cX,\mu)$ and $(\cY,\mathsf{d}_\cY,\nu)$~is
\begin{equation}
    \mathsf{GW}_{p,q}(\mu,\nu)\coloneqq \inf_{\pi\in\Pi(\mu,\nu)}\left( \int_{\cX\times\cY}\int_{\cX\times\cY}
 \Delta_q\big((x,x'),(y,y')\big)^pd\pi\otimes\pi(x,y,x',y')\right)^{\frac{1}{p}},\label{eq:GW_intro}
\end{equation}
where $\Delta_q\big((x,x'),(y,y')\big)\coloneqq\big| \mathsf{d}_\cX(x,x')^q-\mathsf{d}_\cY(y,y')^q\big|$ is the distance distortion cost.\footnote{The original GW distance in \cite{memoli2007use} was defined with $q=1$. The general $(p,q)$-GW distance was introduced later in \cite{sturm2012space,sturm2023space}.} An optimal $\pi^\star$~yields an alignment plan that minimizes distortion and reveals how inherently different the datasets are. The GW distance defines a metric on the space of all mm spaces modulo measure preserving isometries.\footnote{Two mm spaces $(\cX,\mathsf{d}_\cX,\mu)$ and $(\cY,\mathsf{d}_\cY,\nu)$ are isomorphic if there exists a measure-preserving isometry between them, namely, an isometry $T:\cX\to\cY$ with $T_\sharp\mu=\nu$. The quotient space is the one induced by this equivalence relation.} More generally, the GW framework can be used to preserve any arbitrary notion of similarity between the considered spaces (i.e., not necessarily induced by metrics), by defining the similarity functions $c_\cX:\cX^2\to \RR$ and $c_\cY:\cY^2\to \RR$ and replacing the distance distortion cost $\Delta_q$ in~\eqref{eq:GW_intro} with $\big| c_\cX(x,x')-c_\cY(y,y')\big|$; see, e.g., \cite{sturm2012space,sturm2023space,chowdhury2019gromov,arya2024gromov}

While rooted in OT theory, GW alignment is hard to analyze due to its quadratic nature in $\pi$, leading to severe lack of convexity and leaving techniques developed for the (linear) OT problem inapplicable. To overcome this impasse, \cite{zhang2024gromov} have recently derived a variational representation---loosely speaking, a dual form---of the quadratic GW distance between the Euclidean mm spaces $(\RR^{d_x},\|\cdot\|,\mu)$ and $(\RR^{d_y},\|\cdot\|,\nu)$ given~by
\begin{equation}
\GW_{2,2}(\mu,\nu)=C_{\mu,\nu}+\inf_{\bA\in \cD_M} \big\{32\|\bA\|_\F^2+\mathsf{OT}_{\bA}(\mu,\nu)\big\},\label{eq:dual_intro}
\end{equation}
where $C_{\mu,\nu}$ is a constant that depends only on the moments of the marginals, $\cD_M \subset \RR^{d_x\times d_y}$ is a compact rectangle, and $\mathsf{OT}_{\bA}(\mu,\nu)$ is an OT problem with a particular cost function that depends on the auxiliary variable $\bA$. A similar result, tailored for solving the optimization in GW through Fenchel-Moreau duality, was first presented in \cite[Theorem 4.2.5]{vayer2020contribution}. This connection to the well-understood OT problem unlocked it as a tool for the study of GW, leading to notable progress. The dual was used to derive the sample complexity of estimating the quadratic GW alignment cost with/without entropic regularization \cite{zhang2024gromov,groppe2023lower}, as well as to develop inaugural algorithms with convergence guarantees for approximate computation of the GW problem via fast gradient methods \cite{rioux2023entropic}. The dual representation in \eqref{eq:dual_intro} also plays a pivotal role in our development of GW gradient flows and Riemannian structure.

\subsection{Contributions}

This work makes the fist steps in uncovering the GW differential geometry over $\cP_2(\RR^d)$ by considering arguably the simplest variant: the inner product GW (IGW) problem over the same $d$-dimensional Euclidean space 
\begin{equation}
    \mathsf{IGW}(\mu,\nu)\coloneqq\left(\inf_{\pi\in\Pi(\mu,\nu)}\int_{\RR^d\times \RR^d}\int_{\RR^d\times \RR^d} \left|\llangle x,x'\rrangle-\llangle y,y'\rrangle\right|^2d\pi\otimes\pi(x,y,x',y')\right)^{\frac 12},   \label{eq:IGW_intro}
\end{equation}
which is obtained by instantiating $\cX=\cY=\RR^d$ and taking the similarity functions as $c_\cX=c_\cY=\langle \cdot,\cdot\rangle$. The IGW distance captures the change in angles, and as such it is invariant under orthogonal transformations (but not translations). This is an appealing starting point for two main reasons. First, via a simple adaptation of the argument from \cite[Theorem 1]{zhang2024gromov}, the IGW distance adheres to a similar variational form as the $(2,2)$-GW distance, given in \eqref{eq:dual_intro}. Second, and more importantly, IGW between Euclidean spaces is the only known example for which, under mild conditions, Gromov-Monge maps exist. That is, there is a deterministic measure-preserving function that induces an optimal IGW coupling \cite{vayer2020contribution,dumont2024existence}, which is crucial for our construction of IGW gradient flows.\footnote{We note that Gromov-Monge maps exist for other variants of $\GW_{p,q}$ under symmetry assumptions \cite{sturm2012space,sturm2023space,delon2022gromov,arya2024gromov}, but the conditions are too restrictive for our framework.}

As we a priori do not know the differential structure of the IGW space, we draw inspiration from the JKO scheme \cite{jordan1998variational} and initiate our study from an implicit scheme. Given a functional $\sF:\cP_2(\RR^d)\to\RR$ to optimize (assumed to also have orthogonal invariance) and an initial distribution $\mu_0$, we consider a minimizing movement sequence of measures defined recursively via
\begin{equation}
\begin{cases}
    &\rho_0 = \mu_0,\\
    &\rho_{i+1} \in \argmin_{\rho\in\cP_2(\RR^d)} \sF(\rho) + \frac{1}{2\tau}\IGW(\rho,\rho_i)^2,\\
\end{cases}\label{eq:minimizing_movement_intro}
\end{equation}
where $\tau>0$ is the step size. Notably, the invariance to orthogonal transformations means that each $\rho_{i+1}$ in the $\argmin$ represents its entire orbit along the orthogonal group $\mathrm{O}(d)$. To obtain a sequence of distributions, rather than equivalence classes, we instantiate the orthogonal transformations using the SVD decomposition of an optimal $\bA^\star$ matrix in \eqref{eq:dual_intro} between each two consecutive steps. As $\bA^\star$ is given by the cross-correlation matrix under an optimal IGW coupling $\pi^\star$, the transformation aligns the steps with each other in the 2-Wasserstein sense, resulting in a sequence $\{\rho_i\}_{i=0}^n$ that converges to a limit that is continuous not only in $\IGW$, but in $\W_2$ as well. This construction renders our sequence compliant with the natural Wasserstein gradient flow structure associated with the continuity equation 
\begin{equation}
\partial_t\rho_t+\nabla\cdot(v_t\rho_t)=0,\label{eq:cont_intro}
\end{equation}
in the limit of $\tau\to 0$. In other words, the limit of our minimizing movement sequence $\{\rho_i\}_{i=0}^n$ instantiates IGW gradient flows in the Wasserstein space.

Our main result for the IGW gradient flow is twofold: (i) we prove convergence in $\IGW$ (and in $\W_2$ along a subsequence) of the minimizing movement sequence, as $\tau\to 0$, towards a $\W_2$-continuous curve $(\rho_t)_t$; and (ii) the limiting sequence follows the continuity equation \eqref{eq:cont_intro} with a velocity field $v_t\in L^2(\rho_t;\RR^d)$ that is given in terms of the gradient of the first variation of $\sF$ via the partial integro-differential equation (PIDE)
\begin{equation}
v_t=\cL_{\bsigma_t,\rho_t}^{-1}\left[-\nabla\delta \sF(\rho_t)\right],\quad \rho_t\mbox{-a.s.},\label{eq:velocity_intro}
\end{equation}
where $\bsigma_t$ is the covariance matrix of $\rho_t$ and $\cL_{\bsigma,\rho}[v](x) \coloneqq 2\left(\bsigma  v(x) + \int_{\RR^d} y  \llangle v(y),x\rrangle d\rho(y)\right)$ is called the mobility operator, whose inverse we prove exists. To prove this, we utilize the concavity structure of IGW, that gives rise to the variational form \eqref{eq:dual_intro}, which allows us to freeze the dual variable $\bA$ in each JKO step and conduct differential calculus as in the Wasserstein gradient flow. Note the global nature of $\cL_{\bsigma,\rho}[v]$, whose direction at any $x\in\RR^d$ depends on the entire velocity field $v(y)$, $y\in\RR^d$. Unlike the Wasserstein gradient flow, the transformed velocity field $\cL^{-1}_{\bsigma_t,\rho_t}[-\nabla\delta\sF(\rho_t)]$ encodes not only local information, but also global structure, as expected for IGW. Our proof of convergence is constructive, as it builds the velocity field by combining Gromov-Monge maps between the steps of the minimizing movement scheme (following the orthogonal transformations). The analysis employs the framework of generalized minimizing movement (GMM) and Fr\'echet subdifferentials from \cite{ambrosio2005gradient}. This approach circumvents the subdifferential calculus typically employed in Wasserstein gradient flows, which depends on a well-defined understanding of the Wasserstein tangent space at $\mu\in\cP_2(\RR^d)$ as the closure of $\{\nabla\varphi:\,\varphi:\RR^d\to\RR^d\}$ in $L^2(\mu;\RR^d)$. In contrast, the IGW tangent space is not yet well understood and appears to constitute only a small subset of $L^2(\mu;\RR^d)$. Clarifying the structure of the IGW tangent space remains a key open question for future research.

Recalling that the Wasserstein gradient flow obeys the velocity field $\tilde v_t=-\nabla\delta\sF(\rho_t)$, the expression from \eqref{eq:velocity_intro} presents us with a natural candidate for the IGW gradient and leads us to identifying the Riemannian structure that gives rise to the intrinsic~IGW~geometry. Upon defining the intrinsic IGW metric and the induced geodesics, we identify $g_\rho(v,w)\coloneqq \llangle v,  \cL_{\bsigma_\rho,\rho}[w]\rrangle_{L^2(\rho;\RR^d)}$, $v,w\in L^2(\rho;\RR^d)$, as the Riemannian metric tensor that induces the intrinsic metric $\mathsf{d}_{\IGW}$, giving rise to a Benamou-Brenier-like formula \cite{benamou2000computational}  for IGW:
\begin{equation}
        \mathsf{d}_{\IGW}(\mu_0,\mu_1)^2 = \min_{\mu\in\{\mu_1,\bI^-_\sharp\mu_1\}} \inf_{\substack{(\rho_t,v_t):\\\partial_t \rho_t + \nabla\cdot  (\rho_t v_t)=0\\
        \rho_0=\mu_0, \rho_1=\mu}} \int_0^1  g_{\rho_t}(v_t,v_t) dt,\label{eq:IGW_BB_intro}
    \end{equation}
    where $\bI^-$ is any fixed reflection matrix. This dynamical formulation represents the shortest IGW path length as the accumulated kinetic energy, quantified in terms of the metric tensor $g_\rho$ evaluated on the underlying velocity field $v_t$, subject to the continuity equation.\footnote{Note that for the 2-Wasserstein space, the intrinsic metric coincides with $\W_2$, which results in the celebrated Benamou-Brenier formula \cite{benamou2000computational}. In the GW case, geodesic between points in $\cP_2(\RR^d)$ may not be realizable as curves in $\cP_2(\RR^d)$. Consequently, the dynamical formulation in \eqref{eq:IGW_BB_intro} is for $\mathsf{d}_\mathrm{int}$, which may not coincide with $\IGW$ in general.} We conclude with a formal derivation \'a la Otto calculus of the IGW gradient as 
\[
    \grad_\IGW \sF(\rho)=\cL_{\bsigma_\rho,\rho}^{-1}[\nabla\delta\sF(\rho)].
\]
Notably, this corresponds to the PIDE from \eqref{eq:velocity_intro}, whereby $v=-\grad_\IGW \sF(\rho)$ is the direction of the steepest descent of $\sF$ in the IGW geometry. Recalling that the 2-Wasserstein gradient is given by  $\grad_\W \sF(\rho)=\nabla\delta\sF(\rho)$ \cite{otto2001geometry}, we see that $\grad_\IGW \sF(\rho)=\cL_{\bsigma_\rho,\rho}^{-1}[\grad_\W \sF(\rho)]$. The IGW gradient is thus obtained by transforming the Wasserstein gradient using the mobility operator, which enforces the preservation of global structure. Numerical experiments validating our theory and demonstrating the distinctive global nature of IGW interpolations are provided to complement the theory. 
We present experiments for (i) the IGW gradient flow of the potential energy, interaction energy, and entropy functionals, initiated from various shapes (represented as uniform discrete distributions over points in $\RR^d$), and (ii) flow matching between different shapes via the Benamou-Brenier-like IGW formula. The results are compared to their Wasserstein counterparts, highlighting the difference between the global structure-preserving nature of IGW flows versus the Wasserstein flows, that only seek to minimize transportation cost while (possibly) significantly distorting the shape.

\subsection{Literature Review}

The GW distance was originally proposed by M\'emoli in \cite{memoli2007use,memoli2011gromov} as a relaxation of the Gromov-Hausdorff distance, where various structural properties were studied (another GW variant, which is quite different from the one considered herein, was proposed even earlier by Sturm \cite{sturm2006geometry}). The quotient geometry of GW spaces was later explored by Sturm in \cite{sturm2012space,sturm2023space}, where completion, curvature, geodesics, and the tangent structure were analyzed under the general framework of gauged measure spaces. That work provided an account of gradient flows over the quotient space and constructed a functional akin to the Einstein-Hilbert functional, which enabled drawing connections to Ricci flows. However, even when specialized to $\RR^d$, the geodesics (between points in $\cP_2(\RR^d)$, identified with their natural mm spaces) and gradient flows (initiated at a point in $\cP_2(\RR^d)$) arising from Sturm's framework cannot be generally instantiated in $\cP_2(\RR^d)$, as intermediate points typically no longer correspond to Euclidean mm spaces themselves. In contrast, we opt to study IGW gradient flows and differential geometry in $\cP_2(\RR^d)$ without quotient operations, so as to preserve the dynamics arising from OT. This enables us to uncover additional structure, such as the PIDE that characterizes our IGW gradient flows and the Riemannian metric tensor that induces the IGW~geometry.

The increasing interest in GW alignment has driven the exploration of its various facets, encompassing existence of Gromov-Monge maps \cite{vayer2020contribution,memoli2022distance,dumont2024existence}, solutions to the one-dimensional \cite{vayer2020sliced,beinert2022assignment} and Gaussian \cite{delon2022gromov,le2022entropic} GW problems, entropic regularization \cite{solomon2016entropic,peyre2016gromov,scetbon2022linear,rioux2023entropic}, computation \cite{vayer2020sliced,sejourne2021unbalanced,rioux2023entropic}, and statistics \cite{zhang2024gromov,groppe2023lower}. In particular, the characterization of Gromov-Monge maps for the IGW distance \cite{vayer2020contribution,dumont2024existence} and the variational form that connects IGW to OT \cite{vayer2020contribution,zhang2024gromov,rioux2023entropic,sebbouh2024structured} play a crucial role in our~derivations.

While IGW gradient flows and dynamical forms in $\cP_2(\RR^d)$ have not been previously explored, other related metrics and discrepancies on $\cP_2(\RR^d)$ have been studied in this context. 
This includes entropic OT \cite{chen2016relation,gentil2017analogy,conforti2021formula,gigli2020benamou,chizat2020faster}, the Wasserstein-Fisher-Rao metric \cite{chizat2018interpolating,kondratyev2016new}, Sobolev-Fisher Discrepancy \cite{mroueh2020unbalanced}, Stein Discrepancy \cite{liu2016stein,korba2020non,he2022regularized,duncan2023geometry}, Wasserstein gradient flow for maximum mean discrepancy (MMD) \cite{arbel2019maximum}, and Hessian transport \cite{li2019hessian}. Recently, \cite{burger2023covariance} proposed a covariance-modulated dynamical formulation of OT by modifying the Riemannian structure and action functional in the standard Benamou-Brenier formula \cite{benamou2000computational}. Their modulated energy resembles the first term of our mobility operator $\cL_{\bsigma,\rho}$ and induces a modified geometry, under which geodesics and gradient flows are explored. In contrast to the ad-hoc definition from \cite{burger2023covariance}, our operator organically arises from the IGW structure and includes the second (global) term in $\cL_{\bsigma,\rho}$, which accounts for the alignment.

\section{Background and Preliminaries}\label{sec:background}

We briefly review basic definitions and preliminary results concerning the OT and the GW problems. 

\subsection{Notations}

Let $\|\cdot\|$ be the Euclidean norm and write $\llangle \cdot,\cdot\rrangle$ for the inner product. The $d$-dimensional orthogonal group is denoted by $\mathrm{O}(d)$, and the special orthogonal group by $\mathrm{SO}(d)$. The operator and Frobenious norms of a matrix $\bA\in\RR^{d\times k}$ are denoted by $\|\bA\|_\op$ and $\|\bA\|_\F$, respectively. The singular value decomposition (SVD) of $\bA$ is denoted by $\mathbf{A} = \mathbf{P}\blambda\mathbf{Q}^\intercal$, where $\mathbf{P},\mathbf{Q}\in \mathrm{O}(d)$ and $\blambda$ is a diagonal matrix with diagonal entries $\sigma_1(\mathbf{A})\geq\cdots\geq\sigma_d(\mathbf{A})$ sorted from largest to smallest. For a diagonalizable $\mathbf{A}\in\RR^{d\times d}$, we similarly write $\mathbf{A} = \mathbf{P}\blambda\mathbf{P}^\intercal$ for its diagonalization, where $\blambda$ has the eigenvalues $\lambda_{\max}(\mathbf{A}) = \lambda_1(\mathbf{A})\geq\cdots\geq\lambda_d(\mathbf{A})=\lambda_{\min}(\mathbf{A})$ on its diagonal. 
A symmetric matrix $\bA$ is positive semi-definite (PSD) if its eigenvalues are nonnegative. For two symmetric matrices $\bA,\bB\in\RR^{d\times d}$, we write $\bA\succcurlyeq\bB$ when $\bA-\bB$ is PSD.

Write $\cP(\RR^d)$ for the space of Borel probability measures over $\RR^d$, and let $\cP_p(\RR^d)$ be its restriction to distributions with finite absolute $p$-th moments. The subsets of distributions that have a density with respect to (w.r.t.) the Lebesgue measure on $\RR^d$ are denoted by $\cP^{\mathrm{ac}}(\RR^d)$ and $\cP_p^{\mathrm{ac}}(\RR^d)$, respectively. For $\mu\in\cP(\RR^d)$, we write $\supp(\mu)$ for its support, while $\bsigma_\mu$ and $M_p(\mu)$ denote the covariance matrix and the absolute $p$-th moment of $\mu$, respectively. We write $T_\sharp \mu$ for the pushforward measure of $\mu$ through a measurable map $T$, defined as $T_\sharp \mu(B)=\mu\big(T^{-1}(B)\big)$, for every Borel set $B$. 
A sequence of probability measures $\{\nu_k\}_k$ converges weakly to $\nu$, denoted as $\nu_n\stackrel{w}{\to}\nu$, if $\lim_{k\to\infty}\int f\,d\nu_k = \int f\, d\nu$, for any bounded continuous function $f$.

For $\mu\in\cP(\RR^d)$, let $L^2(\mu;\RR^d)\coloneqq\big\{v:\RR^d\to\RR^d\,:\, \|v\|_{L^2(\mu;\RR^d)} \coloneqq \left(\int\|v(x)\|^2 d\mu(x)\right)^{1/2}<\infty\big\}$. The inner product over $L^2(\mu;\RR^d)$ is  $\llangle v, w\rrangle_{L^2(\mu;\RR^d)}\coloneqq \int \llangle v(x),w(x)\rrangle d\mu(x)$. Denote the space of compactly supported smooth function on a metric space $\cX$ with value in $\RR^k$ as $C_c^\infty(\cX;\RR^k)$, and write $C_c^\infty(\cX)$ when $k=1$ for brevity. For a metric space $(\cX,\sd)$, denote the ball with radius $r>0$ at $x\in \cX$ as $\cB_{\sd}(x,r)\coloneqq\{y\in\cX: \sd(y,x)\leq r\}$, where typically we instantiate $\sd$ as the Wasserstein or the IGW distance. For any two (pseudo)metric spaces $(\cX,\sd_\cX),(\cY,\sd_\cY)$, denote by $\lip(\cX;\cY)\coloneqq\{ f:\cX\to\cY, \sup_{\sd_\cX(x,y)\neq0} \frac{\sd_\cY(f(x),f(y))}{\sd_\cX(x,y)}<\infty\}$ the collection of Lipschitz mappings, and for $\gamma\in\lip(\cX;\cY)$, define the Lipschitz constant by $\lip(\gamma) \coloneqq  \sup_{\sd_\cX(x,y)\neq0} \frac{\sd_\cY(\gamma(x),\gamma(y))}{\sd_\cX(x,y)}$. When the second space is $(\cY,\mathsf{d}_\cY)=\big(\cP_2(\RR^d),\IGW\big)$, we emphasize this in our notation by writing $\lip_{\IGW}\big(\cX;\cP_2(\RR^d)\big)$ for the class and $\lip_\IGW(\gamma)$ for the Lipschitz constant; a similar convention is used when $\IGW$ is replaced with $\W_2$. We use $\lesssim_x$ to denote inequalities up to constants that only depend on $x$; the subscript is dropped when the constant is universal.

\subsection{Optimal Transport}\label{subsec:OT}

Let $\cX,\cY$ be two Polish spaces and consider a lower semi-continuous cost function $c:\cX\times\cY\to\RR$. The OT problem \cite{villani2008optimal,santambrogio2015optimal,peyre2019computational} between $\mu\in\cP(\cX)$ and $\nu\in\cP(\cY)$ with cost function $c$ is
\begin{equation}
    \mathsf{OT}_c(\mu,\nu) \coloneqq \inf_{\pi\in\Pi(\mu,\nu)} \int_{\cX\times\cY} c(x,y)  d \pi(x,y),\label{eq:OT}
\end{equation}
where $\Pi(\mu,\nu)$ is the set of all couplings of $\mu$ and $\nu$. The special case of the $p$-Wasserstein distance, for $p \in [1,\infty)$, is given by $\mathsf{W}_p(\mu,\nu)\coloneqq\left(\mathsf{OT}_{\|\cdot\|^p}(\mu,\nu)\right)^{1/p}$. It is well known that $\mathsf{W}_p$ is a metric on $\cP_p(\R^d)$, and metrizes weak convergence plus convergence of $p$-th moments, i.e., $\Wp(\mu_n,\mu) \to 0$ if and only if $\mu_n \stackrel{w}{\to} \mu$ and $M_p(\mu_n) \to M_p(\mu)$. The Wasserstein space $\mathfrak{W}_p=\big(\cP_p(\R^d),\mathsf{W}_p\big)$ entails a rich geometry, where one may reason about geodesic curves, barycenters, gradient flows, and even Riemannian structure; cf. \cite{villani2008optimal,santambrogio2015optimal} for details.

OT is a linear program (generally, an infinite-dimensional one) and as such it admits strong duality. Suppose that the cost $c$ satisfies $c(x,y)\geq a(x)+b(y)$, for all $(x,y)\in\cX\times\cY$, for some upper semi-continuous functions $(a,b)\in L^1(\mu)\times L^1(\nu)$. Then (cf. \cite[Theorem 5.10]{villani2008optimal}): 
\begin{equation}
    \mathsf{OT}_c(\mu,\nu) = \sup_{(\varphi,\psi)\in\Phi_c} \int_\cX \varphi d \mu + \int_\cY \psi d \nu,\label{eq:OT_dual}
\end{equation}
where $\Phi_c\coloneqq\big\{(\varphi,\psi)\in C_b(\cX)\times C_b(\cY): \varphi(x)+\psi(y)\leq c(x,y),\,\forall (x,y)\in\cX\times\cY\big\}$. Furthermore, defining the $c$- and $\bar{c}$-transform of $\varphi\in C_b(\cX)$ and $\psi\in C_b(\cY)$ as $\varphi^c(y)\coloneqq\inf_{x\in\cX} c(x,y)-\varphi(x)$ and $\psi^{\bar{c}}(x)\coloneqq\inf_{y\in\cY} c(x,y)-\psi(y)$, respectively, the optimization above can be restricted to pairs $(\varphi,\psi)$ such that $\psi=\varphi^c$ and $\varphi=\psi^{\bar{c}}$.

A key ingredient for analysis of classical gradient flows in $\mathfrak{W}_2$ is the celebrated Brenier's theorem \cite{brenier1991polar} (see also \cite[Theorem 9.4]{villani2008optimal} or \cite[Section 6.2.3]{ambrosio2005gradient}). Under appropriate conditions, this result establishes the existence of OT (also known as, Brenier or Monge) maps, thus equating the Kantorovich formulation from \eqref{eq:OT} and the Monge problem \cite{mongeOT1781}
\begin{align*}
    \inf_{\substack{T:\RR^d\to\RR^d\\ T_\sharp\mu=\nu}} \int c\big(x,T(x)\big) d\mu(x).
\end{align*}
\begin{theorem}(Brenier's Theorem; simplified)\label{thm:brenier}
    If $\mu\in\cP^{\ac}_2(\RR^d)$ and $\nu\in\cP_2(\RR^d)$, then there exists a unique OT coupling $\pi^\star\in\Pi(\mu,\nu)$ for $\W_2$, which is induced by a transport map $T^{\mu\to\nu}:\RR^d\to\RR^d$, i.e., $\pi^\star = (\id,T^{\mu\to\nu})_\sharp\mu$. Furthermore, $T^{\mu\to\nu}$ is given by the gradient of a convex function, i.e., $T^{\mu\to\nu}=\nabla \varphi$ a.e., for a convex $\varphi:\RR^d\to\RR$. 
\end{theorem}

\subsection{Inner Product Gromov-Wasserstein Distance}
We consider the GW distance with inner product cost (abbreviated IGW) between the Euclidean spaces $(\RR^{d_x},\|\cdot\|,\mu)$ and $(\RR^{d_y},\|\cdot\|,\nu)$, given~by
\begin{equation}
    \mathsf{IGW}(\mu,\nu)\coloneqq\left(\inf_{\pi\in\Pi(\mu,\nu)}\iint \left|\llangle x,x'\rrangle-\llangle y,y'\rrangle\right|^2d\pi\otimes\pi(x,y,x',y')\right)^{\frac 12}.   \label{eq:IGW}
\end{equation}
The minimum is achieved thanks to the compactness of the coupling set and the weak continuity of the objective, see, e.g., \cite{chowdhury2019gromov}. Unlike the standard GW distance from \eqref{eq:GW_intro}, originally defined in \cite{memoli2007use,memoli2011gromov}, the above cost function does not quantify distortion of distances but rather captures the distortion in similarities, i.e., the change in angles. As such, while IGW is invariant under orthogonal transformations, it does not possess translation invariance~\cite{le2022entropic}. IGW received recent attention due to its analytic tractability and since it captures a meaningful notion of discrepancy between mm spaces with a natural inner product structure \cite{le2022entropic,delon2022gromov,vayer2020contribution}. Building on this tractability, herein we present the first IGW gradient flow algorithm, study its convergence, and derive the PIDE characterizing the solution in the continuous-time limit. This, in turn, leads us to identify the Riemannian metric tensor that gives rise to the intrinsic geometry and differential structure of IGW spaces.

\medskip
We rely on two key properties of the IGW distance: (i) it adheres to a variational/dual representation that connects back to OT duality, and (ii) any optimal IGW alignment plan with nonsingular cross-covariance is induced by a map. We next discuss both these aspects and the relevant result.

Recently, \cite[Lemma 2]{rioux2023entropic} derived a variational representation of IGW, following an argument similar to \cite[Theorem 1]{zhang2024gromov}, which originally accounted for the quadratic GW distance (see also \cite{vayer2020contribution,sebbouh2024structured}). This dual, which we restate below, connects IGW to a certain OT cost, thereby enabling us to borrow tools from OT theory (in particular, Monge/Brenier maps). To state this result, we first expand the squared cost from~\eqref{eq:IGW} to decompose IGW as $\mathsf{IGW}^2=\sF_1+\sF_2$, where 
\begin{align*}
    \sF_1(\mu,\nu)&=\int |\llangle x,x'\rrangle|^2d\mu\otimes \mu(x,x')+\int |\llangle y,y'\rrangle|^2d\nu\otimes \nu(y,y')
    \\
    \sF_2(\mu,\nu)&=\inf_{\pi\in\Pi(\mu,\nu)} -2\int \llangle x,x'\rrangle\llangle y,y'\rrangle d\pi\otimes \pi(x,y,x',y').
\end{align*}
We have the following dual for the $\sF_2$ functional, proven in \cref{appen:F2Dual_proof} for completeness.

\begin{lemma}[IGW duality]
\label{lem:F2Dual}
Fix  $(\mu,\nu)\in\cP_4(\RR^{d_x})\times\cP_4(\RR^{d_y})$, and define $M_{\mu,\nu}:=\sqrt{M_2(\mu)M_2(\nu)}$. We have
    \begin{equation}
    \sF_2(\mu,\nu)=\inf_{\mathbf{A}\in\RR^{d_x\times d_y}} 8\|\mathbf{A}\|_\F^2+\mathsf{IOT}_{\mathbf{A}}(\mu,\nu),\label{eq:F2Decomp}
    \end{equation}
    where $\mathsf{IOT}_{\bA}(\mu,\nu)$ is the OT problem with cost function $c_{\mathbf{A}}\mspace{-3mu}:\mspace{-3mu}   
    (x,y)\in\RR^{d_x}\times \RR^{d_y}\mapsto\mspace{-3mu}-8x^{\intercal}\mathbf{A}y$ and the infimum is achieved at some $\mathbf{A}^{\star}\mspace{-3mu}\in\mspace{-3mu}\cD_{M_{\mu,\nu}}\mspace{-3mu}\coloneqq\mspace{-3mu}[-M_{\mu,\nu}/2,M_{\mu,\nu}/2]^{d_x\times d_y}$. Furthermore, denoting an optimal coupling for $\mathsf{IOT}_{\mathbf{A}}(\mu,\nu)$ by $\pi^\star_{\bA}$, then the following two statements hold:
    \begin{enumerate}[leftmargin=*]
    \item[(i)] if $\bA^\star$ achieves the infimum in \eqref{eq:F2Decomp}, then $\pi^\star_{\bA^\star}\in\Pi(\mu,\nu)$ is optimal for $\IGW(\mu,\nu)$ in \eqref{eq:IGW} and $\bA^\star = \frac{1}{2}\int xy^\intercal\pi^\star_{\bA^\star}(x,y)$;
    \item[(ii)] conversely, there exists $\pi^\star\in\Pi(\mu,\nu)$ that is optimal for $\IGW(\mu,\nu)$ in \eqref{eq:IGW}, such that $\bA^\star = \frac{1}{2} \int xy^\intercal\pi^\star(x,y)$, in which case we further have $\pi^\star_{\bA^\star}=\pi^\star$.
    \end{enumerate}    
\end{lemma}

Under mild conditions, the IGW distance between Euclidean spaces enjoys the existence of Gromov-Monge maps. That is, there is a deterministic measure-preserving function that induces an optimal IGW alignment plan. This observation was made in \cite[Theorem 4.2.3]{vayer2020contribution} and \cite[Theorem 4]{dumont2024existence}, but we rederive it in \cref{appen:gromov_monge_proof} in a form that is compatible with our needs and \cref{lem:F2Dual}.

\begin{lemma}[Gromov-Monge map]\label{lem:gromov_monge}
    For $\mu\in\cP_2^{\mathrm{ac}}(\RR^d)$ and $\nu\in\cP_2(\RR^d)$, let $\pi^\star\in\Pi(\mu,\nu)$ be an optimal IGW coupling such that $\bA^\star=\frac{1}{2}\int xy^\intercal\pi^\star(x,y)$ is nonsingular. Then $T^\star\coloneqq (8\bA^\star)^{-1}T^{\mu\to (8\bA^\star)_\sharp\nu}$ is a Gromov-Monge map for $\IGW(\mu,\nu)$, i.e., we have $\pi^\star = (\id,T^\star)_\sharp\mu$ and $\bA^\star = \frac{1}{2}\int x\, T^\star(x)^\intercal d\mu(x)$.
\end{lemma}

\subsection{Subdifferential Calculus}
We recall some basic definition from subdifferential calculus in the space of probability measures. The first variation of a functional describes its first-order change when the input measure is perturbed (see, e.g., \cite[Definition 7.12]{santambrogio2015optimal}). We later use it in our characterization of the PIDE that governs the IGW gradient~flow.

\begin{definition}[First variation]
Given a functional $\sF:\cP(\RR^d)\to\RR\cup\{+\infty\}$, we say that $\mu\in\cP(\RR^d)$ is regular for $\sF$ if $\ \sF((1-\epsilon)\mu+\epsilon\nu)<+\infty$, for every $\epsilon\in[0,1]$ and $\nu\in\cP^{\mathrm{ac}}(\RR^d)$ with bounded density and compact support. If $\mu$ is regular for $\sF$, the first variation of $\sF$ at $\mu$, if exists, is any measurable function $\delta\sF(\mu):\RR^d\to\RR$ such that
    \begin{align*}
        \frac{d}{d\epsilon}\sF\big(\mu+\epsilon(\nu-\mu)\big)\big|_{\epsilon=0} = \int \delta\sF(\mu) d(\nu-\mu),
    \end{align*}
    for any $\nu\in\cP^{\ac}(\RR^d)$ with bounded density and compact support.
\end{definition}

While the first variation is used to describe the solution (as well as to gain intuition on some proofs), the rigorous derivation employs the Fr\'echet subdifferential \cite[Definition 10.1.1]{ambrosio2005gradient}. 

\begin{definition}[Fr\'echet subdifferential]\label{definition:frechet_subdifferential}
    Given a proper and lower semi-continuous functional $\sF:\cP_2(\RR^d)\to\RR\cup\{+\infty\}$ with $\dom(\sF) \coloneqq\{\mu\in\cP_2(\RR^d):\sF(\mu)<+\infty\} \subset\cP_2^{\ac}(\RR^d)$, its Fr\'echet subdifferential $\partial\sF(\mu)$ is the collection of all $\xi\in L^2(\mu;\RR^d)$ such that
    \begin{align*}
        \sF(T_\sharp\mu)-\sF(\mu)\geq\int \llangle\xi(x),T(x)-x\rrangle d\mu(x) +o(\|T-\id\|_{L^2(\mu;\RR^d)})
    \end{align*}
    for any map $T$. Any element $\xi\in\partial\sF(\mu)$ is called a strong subdifferential.
\end{definition}

The Fr\'echet subdifferential $\partial\sF(\mu)$ relies on the tangent structure of Wasserstein space $\mathfrak{W}_2$ and exists under certain convexity assumptions. The first variation $\delta\sF$, on the other hand, conicides with the linear G\^ateaux derivative of $\sF$, whose existence typically requires  further regularity assumptions. We note that when both exist, $\nabla\delta\sF(\mu)\in\partial\sF(\mu)$ holds, for instance, when $\sF$ is variational integral with high smoothness, see \cite[Lemma 10.4.1]{ambrosio2005gradient}. Thus without much loss of generality, we will present our main results in terms of $\delta\sF$ for conciseness.

\section{Wasserstein Comparisons and Local Convexity}\label{sec:IGW_structure}

We establish structural properties of IGW that are later employed for deriving the convergence of the gradient flow algorithm and characterizing its continuous-time limit. Importantly, in \cref{subsec:local_convexity} we provide a detailed study of the local convexity of IGW along generalized geodesics, which we believe may be of independent interest. Proofs for the results of this section are deferred to \cref{appen:IGW_structure_proofs}.

\medskip

To begin, we explore the metric structure of the IGW distance. Past works \cite{delon2022gromov,le2022entropic,vayer2020contribution} partially account for these aspects but do not establish the metric properties in full (e.g., the triangle inequality follows from the general results in \cite[Theorem 16]{chowdhury2019gromov}, but not the nullification condition).
The next proposition, proven in \cref{appen:prop:IGW_pmetric_proof}, closes this gap, showing that the IGW distance defines a pseudometric on probability distributions over a shared Hilbert space.

\begin{proposition}[IGW pseudometric]\label{prop:IGW_pmetric}
For a Hilbert space $\cH$, the IGW distance defines a pseudometric on $\cP_2(\cH)$, with $\IGW(\mu,\nu) = 0$ if and only if there exists a $\mu\otimes\mu$-a.s. unitary isomorphism $T:\supp(\mu)\to\supp(\nu)$ with $T_\sharp\mu=\nu$.\footnote{Namely, $T$ is a bijection with $T_\sharp\mu=\nu$ and $\llangle x,x'\rrangle = \llangle T(x),T(x')\rrangle$, for $x,x'$ $\mu\otimes\mu$-a.s. Such maps are $\mu$-a.s. linear.}
\end{proposition}

While the pseudometric structure is enough for our needs, this proposition directly implies that $\IGW$ metrizes the quotient space of $\cP_2(\cH)$, modulo the above unitary isomorphic relationship. We leave formalizing this observation and the exploration of the quotient topology for future work.

\subsection{Wasserstein Comparisons}

We next present comparison results between $\IGW$ and $\W_2$ over Euclidean spaces, which play a crucial role in our construction of the GMM scheme. First, we define a certain subset of the orthogonal group in $\RR^d$ using which cross-covariance matrices induced by optimal IGW couplings can be symmetrized. This symmetrization is crucial for instantiating the IGW gradient flow in the Wasserstein space, so as to adhere to the continuity equation (see \cref{sec:IGW_gradient_flow}).

\begin{definition}[Cross-covariance PSD transform]\label{def:PSD_transform}
    Fix $\mu,\nu\in\cP_2(\RR^d)$ and let $\pi^\star\in\Pi(\mu,\nu)$ be an optimal IGW coupling. Consider the SVD of the cross-covariance matrix $\int xy^\intercal d\pi^\star(x,y) =\bP\blambda\bQ^\intercal$ and let $\bO=\bP\bQ^\intercal$. Define $\cO_{\mu,\nu}\subset\Od$ as the collection of all orthogonal matrices $\bO$ constructed as above, from any optimal IGW coupling (indeed, optimal IGW couplings need not be~unique).
\end{definition}

A key consequence of the above definition that we repeatedly use in the sequel is the following. 

\begin{lemma}[Cross-covariance PSD tranform]\label{lem:symmetrization}
For $\mu,\nu\in\cP_2(\RR^d)$ and any $\bO\in\cO_{\mu,\nu}$, there exists an optimal IGW coupling $\pi^\star$ for $(\mu,\bO_\sharp\nu)$, such that the cross-covariance matrix $\int xy^\intercal d\pi^\star(x,y)$ is PSD and $\bA^\star= \frac{1}{2}\big(\int xy^\intercal d\pi^\star\big) $ achieves optimality in the dual form in \eqref{eq:F2Decomp}.
\end{lemma}
See \cref{appen:lem:symmetrization_proof} for proof. The cross-covariance symmetrization further allows establishing a certain equivalence between Wasserstein and IGW distances defined on the same ambient space. The following lemma, proven in \cref{appen:lem:equivalence_igw_w_proof}, is subsequently used to realize $\IGW$-continuous curve that are also $\W_2$-continuous.

\begin{lemma}[IGW and Wasserstein comparison]\label{lem:equivalence_igw_w}
For $\mu,\nu\in\cP_{2}(\RR^d)$, we have
\[
    \IGW(\mu,\nu) \leq \big(2 M_2(\mu)+2M_2(\nu) \big)^{\frac{1}{2}} \W_2(\mu,\nu).
\]
Conversely, if $\bsigma_\mu$ and $\bsigma_\nu$ are nonsingular, then for any $\bO\in\cO_{\mu,\nu}$, we have
\[
\left(\frac{1}{2}\big(\lambda_{\min}(\bsigma_\mu)^2+\lambda_{\min}(\bsigma_\nu)^2\big)\right)^{\frac 14} \W_2(\mu,\bO_\sharp\nu) \leq  \IGW(\mu,\nu).
\]
\end{lemma}

Lastly, we show that Lipschitz continuous curves in IGW can be transformed into Lipschitz continuous curves in the 2-Wasserstein distance. We do not directly use this fact later, but find it of independent interest, as it testifies to the regularity of IGW Lipschitz curves. For proof see \cref{appen:prop:IGW_curve_realization_proof}

\begin{proposition}[Lipschitz IGW and Wasserstein curves]\label{prop:IGW_curve_realization}
 Let $\rho\in\lip_{\IGW}\big([0,1];\cP_2(\RR^d)\big)$ with $\lip_\IGW(\rho)=L$, i.e., $\IGW\big(\rho_s,\rho_t\big) \leq L|s-t|$, and suppose $\inf_{t\in[0,1]}\lambda_{\min} \big(\bsigma_{\rho_t}\big)\geq c>0$. Then there exists $\tilde\rho\in\lip_{\W_2}\big([0,1];\cP_2(\RR^d)\big)$ with $\lip_{\W_2}(\tilde\rho)=\frac{L}{\sqrt{c}}$ and $\sup_{t\in[0,1]}\IGW\big(\rho_t,\bar{\rho}_t\big)=0$.
\end{proposition}

\subsection{Local Convexity along Generalized Geodesics}\label{subsec:local_convexity}

Crucial for our convergence analysis of the IGW gradient flow is its convexity profile. Specifically, we establish local convexity of the IGW distance along generalized geodesics, as defined next and illustrated in \cref{fig:generalized_geodesic_diagram}.

\begin{figure}[t!]
    \centering
\includegraphics[width=0.5\textwidth]{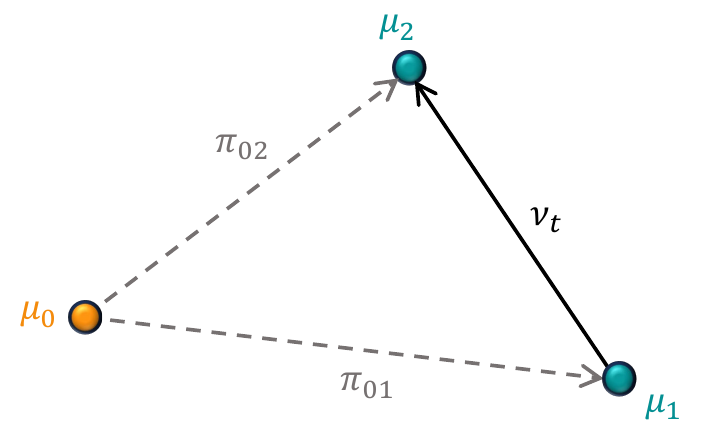}    
\caption{Generalized IGW geodesic between between $\mu_1$ and $\mu_2$ w.r.t. $\mu_0$.}\label{fig:generalized_geodesic_diagram}
\end{figure}

\begin{definition}[Generalized IGW geodesics]\label{def:gen_geo_IGW}
    For measures $\mu_0,\mu_1,\mu_2\in\cP(\RR^d)$, denote by $\pi_{01},\pi_{02}$ optimal IGW couplings for $(\mu_0,\mu_1)$ and $(\mu_0,\mu_2)$, respectively. Denote the glued joint distribution by $\pi\in\Pi(\mu_0,\mu_1,\mu_2)$ \cite[Lemma 7.6]{villani2003topics} and suppose that $\mu_1, \mu_2$ is already rotated w.r.t. $\mu_0$ such that both optimal coupling have PSD and nonsingular (uncentered) cross-covariance matrices $\int xy^\intercal d\pi_{01}(x,y)$, $\int xz^\intercal d\pi_{02}(x,z)$. The generalized IGW geodesic between $\mu_1$ and $\mu_2$ w.r.t. $\mu_0$ is given by $\nu_t\coloneqq ((1-t)y + tz)_\sharp\pi$, $t\in[0,1]$. 
\end{definition}

Note that $\nu_t$ depends on the choice of the (nonunique) optimal couplings, which we always assume to be fixed and consider the resulting glued joint distribution. The assumption that $\mu_1,\mu_2$ are rotated is introduced for convenience to simplify subsequent derivations. This does not limit the generality of our IGW gradient flow framework since the relevant objects are rotationally invariant. Therefore, we can always rotate one of the measures (usually $\mu_2$) to achieve the PSD property. The nonsingularity requirement on the cross-covariance matrices is more stringent. It is introduced to account for the entropy objective functional, whose convexity is contingent on this assumption (see \cref{appen:entropy}); other objectives of interest, such as potential or interaction energy, do not need this assumption. Nevertheless, in our subsequent derivations, we make careful choices of various parameters to ensure this nonsingulatiry (see \cref{lem:bar_delta}).

\medskip
IGW satisfies the following property, which is in parallel to the convexity of $\W_2$ along Wasserstein generalized geodesics, see \cite[Lemma 9.2.1]{ambrosio2005gradient}.

\begin{lemma}[Local convexity of $\IGW$]\label{lem:local_convexity_igw}
    Let $\mu_0,\mu_1, \mu_2\in\cP_2(\RR^d)$ be such that $\lambda_{\min}(\bsigma_{\mu_0}) \geq c$, for some constant $c>0$ and $\IGW(\mu_0,\mu_1)\leq \frac{c}{16\sqrt{2}}$. Then, along the generalized geodesic $\nu_t$ from $\mu_1$ to $\mu_2$  w.r.t. $\mu_0$, we have
    \begin{align*}
        \IGW(\nu_t,\mu_0)^2 &\leq (1-t)\IGW(\mu_1,\mu_0)^2 + t\IGW(\mu_2,\mu_0)^2 - t\left(1- \frac{8\sqrt{2}}{c}\IGW(\mu_0,\mu_1)\right)\mspace{-4mu} \IGW(\mu_1,\mu_2)^2 \\
        &\qquad + \frac{12\sqrt{2}}{c}t\IGW(\mu_0,\mu_1)^3 + O(t^2)
    \end{align*}
    where in $O(t^2)$ we have omitted constants depending only on the second moments of $\mu_0,\mu_1, \mu_2$.
\end{lemma}

The proof of this lemma, given in \cref{appen:lem:local_convexity_igw_proof}, relies on the polynomial expansion of the primal form of $\IGW(\nu_t,\mu_0)^2$ and \cref{lem:equivalence_igw_w}, which leads to a formula similar to usual convexity, up to an additive higher order terms. This suggests that $\IGW^2$ behaves like a convex functional near the starting point (i.e., near $t=0$), which we refer to as local convexity. This notion of convexity is sufficient to run and prove the convergence of the proximal point method, akin to the JKO scheme for $\W_2$ \cite{jordan1998variational}, as described in the next section.

\section{IGW Gradient Flows}\label{sec:IGW_gradient_flow}

Our goal is to develop a gradient flow in $\cP_2(\RR^d)$ w.r.t. the IGW geometry. Alas, there is currently no understanding of the differential/Riemannian structure of the IGW space. To circumvent this issue, we draw inspiration from the celebrated Jordan-Kinderlehrer-Otto (JKO) scheme \cite{jordan1998variational}, and consider a proximal point method with $\sW_2$ replaced  by $\IGW$.  We next describe the setting, the employed proximal point algorithm, and conclude this section with the main result, accounting for convergence and a characterization of the limiting curve. Throughout, we impose the following assumption.

\begin{assumption}[Objective and initialization]\label{assumption:main}
The initialization point $\mu_0\in\dom(\sF)\subset\cP_2(\RR^d)$ has a nonsingular covariance matrix $\bsigma_{\mu_0}=\int xx^\intercal d \mu_0(x)$ and the target functional $\sF:\cP_2(\RR^d)\to\RR\cup\{\infty\}$~is:
\begin{enumerate}[leftmargin=*]
    \item[(i)] rotation invariant (e.g., negative entropy  $\sF(\mu) = \int \log\mu\, d \mu$), lower bounded (i.e. $\sF^\star\coloneqq\inf_\rho \sF(\rho)>-\infty$), and weakly lower semi-continuous;
    \item[(ii)] regular in the sense of \cite[Definition 10.1.4]{ambrosio2005gradient} and has $\dom(\sF)\subset\cP_2^{\mathrm{ac}}(\RR^d)$.
    \item[(iii)] $\lambda$-convex along any generalized IGW geodesics with some $\lambda\in\RR$, i.e., for $\mu_0,\mu_1,\mu_2 \in \dom(\sF)$ with a glued joint distribution $\pi\in\Pi(\mu_0,\mu_1,\mu_2)$ as defined in \cref{def:gen_geo_IGW}, and the resulting generalized geodesic $\nu_t$ from $\mu_1$ to $\mu_2$  w.r.t. $\mu_0$, we~have
    \begin{equation}
    \sF(\nu_t)\leq (1-t)\sF(\mu_1)+t\sF(\mu_2) - \lambda t(1-t) \int \left|\llangle y,y'\rrangle-\llangle z,z'\rrangle\right|^2d\pi\otimes\pi(x,y,z,x',y',z'), \ \forall t\in[0,1];\label{eq:lambda_convexity}
    \end{equation}
\end{enumerate}
\end{assumption}

We note that while $\lambda$-convexity is introduced to broaden the range of admissible functionals, most interesting examples satisfy it with $\lambda=0$, as described next.

\begin{remark}[Examples of functionals]\label{rem:func_examples}
   \cref{assumption:main} is satisfied by a wide range of functionals, see \cite[Section 9.3, Remark 9.2.5]{ambrosio2005gradient}, in particular, those which are convex in $\W_2$. 
   For instance, for the potential energy $\sV(\mu)\coloneqq \int V(x)d\mu(x)$ and the interaction energy $\sW(\mu)\coloneqq \int W(x_1,\ldots,x_k)d\mu^{\otimes k}(x_1,\ldots,x_k)$, convexity along generalized IGW geodesic (in fact, any linearly interpolating curves, see \cite[Proposition 9.3.2, Proposition 9.3.5]{ambrosio2005gradient}) follows from convexity of functions $V$ and $W$, respectively. For the important example of the entropy functional $\sH(\mu)\coloneqq \int \log\left(\frac{d\mu}{dx}\right)d\mu $, we could only establish convexity along a modified version of generalized geodesics, which hinges on a proper choice of PSD rotations and the structure of Gromov-Monge maps (see \cref{appen:entropy} for details and the formal derivation). The modified displacement is not linear, which leaves it unclear whether the potential and interaction energy functionals are also convex along it. We leave the characterization of generalized geodesics along which $\sV$, $\sW$, and $\sH$ are simultaneously convex for future work. 
\end{remark}

\begin{figure}
    \centering
\includegraphics[width=0.8\textwidth]{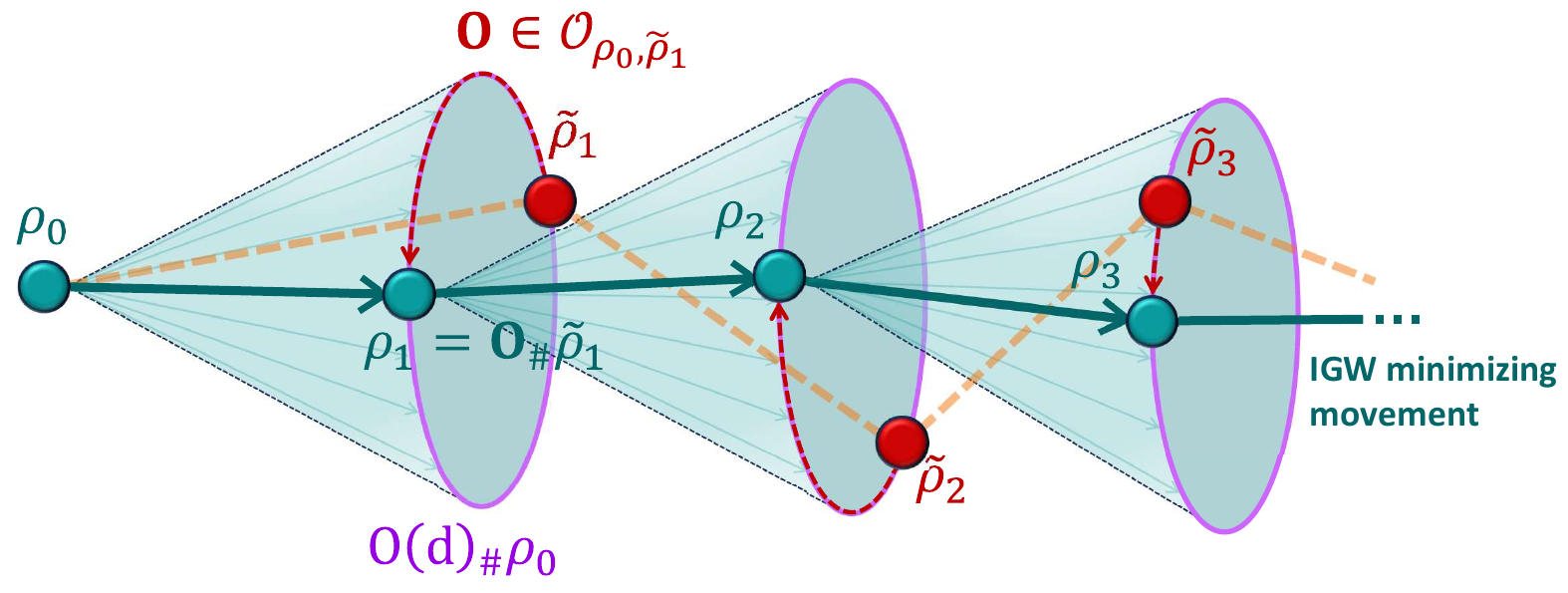}    
\caption{Illustration of steps along IGW proximal point method from \eqref{eq:minimizing_movement}: each purple ring represents an orbit of $\tilde\rho_i$ along the orthogonal group $\mathrm{O}(d)$ in $\RR^d$, i.e., $\mathrm{O}(d)_\sharp\tilde\rho_i\coloneqq \{\bU_\sharp\tilde\rho_i:\,\bU\in\mathrm{O}(d)\}$; after each update step of the algorithm, the obtained distribution $\tilde\rho_{i+1}$ (shown in red, arbitrarily chosen from $\mathrm{O}(d)_\sharp\rho_i$) is transformed using any $\bO\in\cO_{\rho_i,\tilde\rho_{i+1}}$ to obtain $\rho_{i+1}=\bO_\sharp\tilde\rho_{i+1}$ (shown in teal; see \cref{def:PSD_transform}). The transformation aligns each two consecutive steps and guarantees the the cross-covariance matrix between any two consecutive steps is PSD. This results in a regular and controlled path (marked in dark teal), whose convergence and limiting characterization we establish in \cref{thm:main}. The dashed orange line shows the path that would have been obtained without the transformations, which notably lacks stability.}\label{fig:IGW_steps}
\end{figure}

\paragraph{\textbf{Minimizing movement scheme.}} 
We first describe the scheme for a finite total amount of time $\delta>0$ (to be specified in \cref{prop:existence_of_scheme} ahead), and account for its extension to the infinite-time horizon afterwards. Fix the number of steps $n\in\NN$ and set $\tau=\frac{\delta}{n}$ as the step size. The minimizing movement scheme follows a proximal point method that generates a sequence of measures $\{\rho_i\}_{i=0}^n\subset\cP_2(\RR^d)$, termed the \emph{discrete solution}, which is defined recursively, for $i=0,\ldots,n-1$, via
\begin{equation}
\begin{cases}
    &\rho_0 = \mu_0,\\
    &\tilde\rho_{i+1} \in \argmin_{\rho\in\cP_2(\RR^d)} \sF(\rho) + \frac{1}{2\tau}\IGW(\rho,\rho_i)^2,\\
    &\rho_{i+1} = \bO_\sharp\tilde{\rho}_{i+1}, \ \ \bO\in\cO_{\rho_i,\tilde{\rho}_{i+1}}.
\end{cases}\label{eq:minimizing_movement}
\end{equation}
Note that the rotational invariance of the objective in the second line implies that if $\tilde\rho_{i+1}$ belongs to the $\argmin$, then so does the entire orbit 
$\mathrm{O}(d)_\sharp\tilde\rho_{i+1}\coloneqq \{\bU_\sharp\tilde\rho_{i+1}:\,\bU\in\mathrm{O}(d)\}$.\footnote{Each orbit $\mathrm{O}(d)_\sharp\tilde\rho_{i+1}$ has two connected components, depending on whether the matrix entails a reflection or not. This fact is overlooked in our illustration of the scheme in \cref{fig:IGW_steps}, where each orbit is represented by a single ring for simplicity.}
Instead of defining a flow over equivalence classes (i.e., in the quotient space), at each step, we pick a specific representative by starting from an arbitrary minimizing $\tilde{\rho}_{i+1}$ and rotating it w.r.t. $\rho_i$ as described in \cref{lem:symmetrization}, namely, setting $\rho_{i+1} = \bO_\sharp\tilde{\rho}_{i+1}$, for any $\bO\in\cO_{\rho_i,\tilde{\rho}_{i+1}}$.  While we do not specify a selection criteria of $\bO\in\cO_{\rho_i,\tilde{\rho}_{i+1}}$, our theory holds for any sequence $\{\rho_i\}_{i=0}^n$ that adheres to the aforementioned structure. We henceforth fix such an arbitrary minimizing movement sequence and provide our results for it. 

The above construction aligns $\rho_{i+1}$ with $\rho_{i}$ in the $\W_2$ sense, resulting in a sequence $\{\rho_i\}_{i=0}^n$ that complies with the natural Wasserstein gradient flow structure associated with the continuity equation
\begin{equation}
\partial_t\rho_t+\nabla\cdot(v_t\rho_t)=0,\label{eq:cont}
\end{equation} 
in the limit of small step size. The alignment further guarantees that there exists an optimal IGW coupling $\pi^\star\in\Pi(\rho_i, \rho_{i+1})$ such that $\int xy^\intercal d\pi^\star$ is PSD. These properties of the discrete solution are crucial for our convergence analysis and the characterization of the limiting IGW gradient flow curve, stated in \cref{thm:main}. As seen in the theorem, the limiting flow is instantiated in the Wasserstein space and follows the continuity equation. The overall scheme is illustrated in \cref{fig:IGW_steps}, which also shows how the alignment of adjacent steps forms a more regular path.

\medskip

\paragraph{\textbf{Convergence of minimizing movement sequence and continuous-time limit.}} We study the convergence of the discrete solution $\{\rho_i\}_{i=0}^n$ to a curve in $\cP_2(\RR^d)$ as the step size $\tau\to 0$. The limiting curve is identified as the IGW gradient flow and we derive the PIDE that characterizes it. Recall that the Wasserstein gradient of $\sF$ at $\rho$ is given by gradient of its first variation, i.e., $\grad_{\W_2}\sF(\rho)=\nabla\delta \sF(\rho)$ \cite{otto2001geometry}. The Wasserstein gradient flow then evolves according to the velocity field $v_t=-\nabla\delta\sF(\rho_t)$, subject to the continuity equation $\partial_t\rho_t=-\nabla\cdot(\rho_t v_t)$. Interestingly, the corresponding velocity field for the IGW gradient flow is given by a certain inverse (global) linear operator acting on the Wasserstein~gradient.

To describe it, for any PSD matrix $\bA\in\RR^{d\times d}$, probability measure $\mu\in\cP_2(\RR^d)$, and vector field $v\in L^2(\mu;\RR^d)$, define the linear operator $\cL_{\bA,\mu}:L^2(\mu;\RR^d)\to L^2(\mu;\RR^d)$ by 
\begin{equation}
    \cL_{\bA,\mu}[v](x) \coloneqq 2\left(\bA  v(x) + \int_{\RR^d} y  \llangle v(y),x\rrangle d\mu(y)\right).\label{eq:operator}
\end{equation}

In particular, when $\bA=\bsigma_{\mu}$, we call $\cL_{\bsigma_\mu,\mu}$ the \emph{mobility operator}.\footnote{This terminology is borrowed from \cite{burger2023covariance}, where a variant of Benamou-Brenier formula with a certain covariance modulation is studied.} Note the global nature of $\cL_{\bA,\mu}[v]$, whose direction at any $x\in\RR^d$ depends not only on $v(x)$, but also on all other velocities $v(y)$, $y\in\RR^d$, through the second term. Writing $\bsigma_t$ for the covariance matrix of $\rho_t$, \cref{thm:main} ahead, shows that the IGW gradient flow velocity field is $v_t=-\cL^{-1}_{\bsigma_t,\rho_t}[\nabla\delta\sF]$, subject to the continuity equation (the inverse exists under mild conditions; see \cref{rem:property_of_operator} below). 
Unlike the Wasserstein gradient flow, the transformed velocity field $\cL^{-1}_{\bsigma_t,\rho_t}[\nabla\delta\sF]$ encodes not only local information, but also global structure. In the later \cref{subsec:IGW_grad}, this characterization of the limiting velocity field leads us to identifying the IGW gradient as the inverse operator acting on the Wasserstein gradient, i.e., $\grad_\IGW \sF(\rho)=\cL_{\bsigma_\rho,\rho}^{-1}[\grad_{\W_2} \sF(\rho)]$. 

\begin{remark}[Properties of operator]\label{rem:property_of_operator}
We collect here some facts about $\cL_{\bA,\mu}$, proven in \cref{appen:rem:property_of_operator}:
\begin{enumerate}[leftmargin=*]
    \item Clearly, $\cL_{\bA,\mu}$ is self-adjoint on $L^2(\mu;\RR^d)$, and whenever $\bA \succcurlyeq \bsigma_\mu$, it is also PSD.
    \item A direct computation shows that
    $\cI_\mu\coloneqq\{v\in L^2(\mu;\RR^d): \int xv(x)^\intercal d\mu(x) \text{ is symmetric}\}$ is 
    an invariant space of~$\cL_{\bsigma_\mu,\mu}$.
    \item $\cL_{\bsigma_\mu,\mu}$ has a nontrivial kernel. To see this, recall that the tangent space of $\SOd$ at the identity $\bI$ is the set of skew-symmetric matrices. For any tangent element $\bS$ and $\mu\in\cP_2(\RR^d)$, set $v(x)\coloneqq \bS x$ and observe
    \begin{align*}
        \cL_{\bsigma_{\mu},\mu}[v] &= 2 \bsigma_{\mu}\bS x + 2\int_{\RR^d} y y^\intercal  d\mu(y) \bS^\intercal x= 2 \bsigma_{\mu}\bS x + 2\bsigma_{\mu} (-\bS) x=0.
    \end{align*}
    Consequently, the skew-symmetric transformation is nullified by the mobility operator.
    \item As $\cL_{\bsigma_\mu,\mu}$ has a nontrivial kernel it does not have a canonical inverse. Nevertheless, we can invert it on the invariant space $\cI_\mu$ given that $\mu$ has a nonsingular covariance $\bsigma_\mu$. First, we show in \cref{appen:rem:property_of_operator} that for a general nonsingular PSD $\bA$, 
    if $\cL_{\bA,\mu}[v]=w$ with $v\in\cI_\mu$, then $v$ is uniquely determined by $w$ via
    \begin{align*}
        v(x)
        = \frac{1}{2}\bA^{-1}w(x) - \frac{1}{2}x^\intercal\otimes \bI (\bI\otimes\bA^2+\bsigma_\mu\otimes \bA)^{-1} \int (y\otimes \bI) w(y) d \mu(y),\quad x\in\RR^d.
    \end{align*}
    Due to the uniqueness, we call this formula the principle inverse of $\cL_{\bA,\mu}$, and slightly abusing notation, we write $v = \cL_{\bA,\mu}^{-1}[w]$. Now take $\bA=\bsigma_\mu$ and observe that the restriction $\cL_{\bsigma_\mu,\mu}\big|_{\cI_\mu}$ is bijective, with an inverse given by the formula above with $\bA$ replaced by $\bsigma_\mu$.
    \item Restricted to the invariant space $\cI_\mu$, the spectrum of $\cL_{\bsigma_\mu,\mu}$ is discrete and contains at most $d^2+d$ elements. See \cref{appen:rem:property_of_operator} for a detailed study of the eigenvalues and eigenfunctions.
\end{enumerate}
Overall, the mobility operator and its inverse both tend to align the entire velocity field, as is shown in \cref{fig:L_example}. As will be seen subsequently, in the context of the IGW gradient flow, this action will serve to align the movement so as to preserve the global structure of the distribution. 
\begin{figure}
    \centering
    \includegraphics[width=1\linewidth]{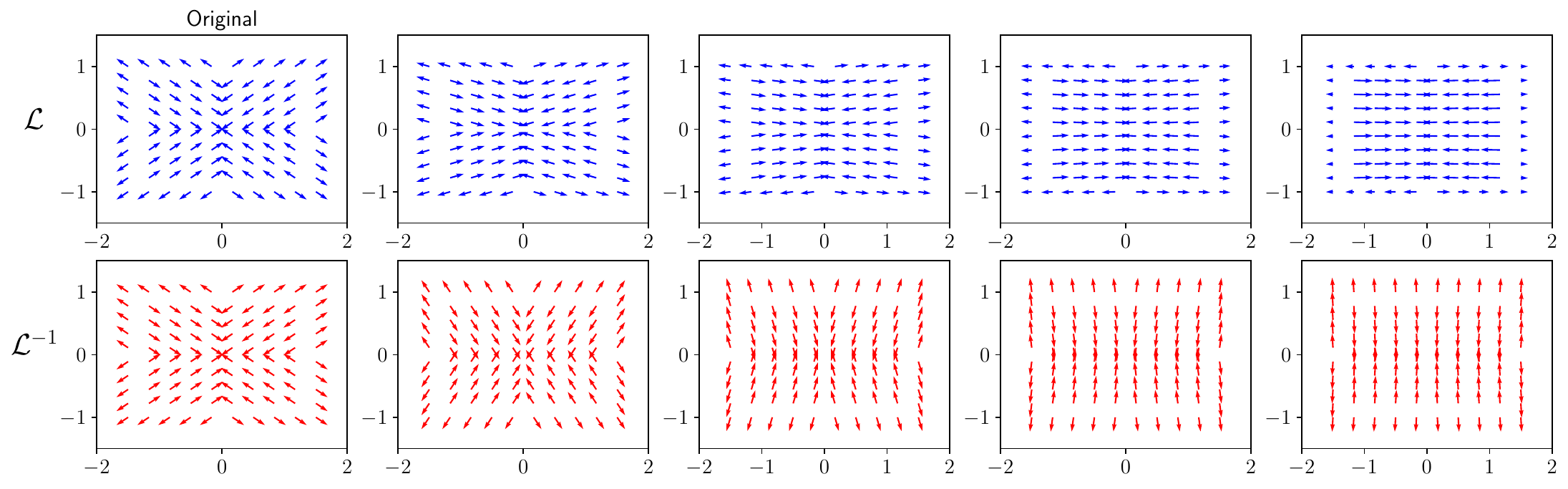}
    \caption{Illustration of the action of the mobility operator $\cL_{\bsigma_\mu,\mu}$ and its inverse. Fixing $\mu$ as uniform distribution over the grid, and we initiate the vector field as pointing inward except for the periphery. The first shows the repeated application of $\cL_{\bsigma_\mu,\mu}$ to the initial vector field, while the second shows the inverse $\cL_{\bsigma_\mu,\mu}^{-1}$. 
    }
    \label{fig:L_example}
\end{figure}

\end{remark}

To formally state our main result in terms of the gradient of the first variation of $\sF$, we impose the following differentiability assumption. This is not strictly required for the derivation, which relies on the more general notion of Fr\'echet subdifferential.

\begin{assumption}[Differentiability of $\sF$]\label{assumption:differentiability}
    Suppose functional $\sF$ has first variation $\delta\sF$ and, for each $\mu\in\dom(\sF)$, the Fr\'echet subdifferential is the singleton $\partial \sF(\mu)=\{\nabla\delta\sF(\mu)\}$.
\end{assumption}

The following theorem is the main result on the IGW gradient flow. It first states the convergence of the discrete solution as $\tau\to 0$ by considering the piecewise constant curve $\bar{\rho}_n(t)=\rho_i$, for $t\in ((i-1)\tau,i\tau]$ and $i=1,\ldots,n$. In addition, the theorem provides a PIDE characterization of the limiting curve in $\cP_2(\RR^d)$. Notably, the result accounts for the convergence and limiting characterization of the IGW gradient flow within a finite time interval $[0,\delta]$, where $\delta$ depends on $\sF$ and $\mu_0$; following the statement, \cref{rem:infinite_time_extension} discusses the extension to the infinite-time horizon.

\begin{theorem}[IGW gradient flow]\label{thm:main}
    Let $\mu_0\in\cP_2(\RR^d)$ and $\sF:\cP_2(\RR^d)\to\RR$ satisfy Assumptions \ref{assumption:main}~and~\ref{assumption:differentiability}. Then, for any $0<\delta< \frac{(1-1/\sqrt{2})\lambda_{\min}(\bsigma_{\mu_0})}{2^{5/4}(\sF(\mu_0) - \sF^\star)}$, the following statements hold:
    \begin{enumerate}[leftmargin=*]
        \item \textbf{Convergence:} The piecewise constant curve $\bar{\rho}_n:[0,\delta]\to\cP_2(\RR^d)$ with step size $\tau=\frac{\delta}{n}$, converges uniformly in $\IGW$ to a $\W_2$-continuous curve $\rho:[0,\delta]\to \cP_2(\RR^d)$ with error $O(\sqrt{\tau})$, and the convergence can be lifted to $\W_2$ convergence along a subsequence. The limiting curve is called the IGW GMM.

        \medskip
        \item \textbf{Limiting solution:} The IGW GMM $\rho:[0,\delta]\to\dom(\sF)\subset\cP_2(\RR^d)$ 
        satisfies the PIDE 
        \begin{align*}
            \begin{cases}
            \rho_0 = \mu_0 &\\
            \partial_t\rho_t + \nabla\cdot (\rho_t v_t)= 0, &\mbox{in distributional sense on }\RR^d\times[0,\delta]\\
             v_t=-\cL_{\bsigma_t,\rho_t}^{-1}\left[\nabla\delta \sF(\rho_t)\right], &\rho_t\mbox{-a.s. }t\in[0,\delta]\mbox{ a.e.}
            \end{cases}
        \end{align*}
        where $\bsigma_t\coloneqq \int_{\RR^d} xx^\intercal d \rho_t(x)$ and $\nabla\delta\sF(\rho_t)\in\cI_{\rho_t}$, for all $t\in[0,\delta]$, whence the inverse $-\cL_{\bsigma_t,\rho_t}^{-1}$ is well-defined (see \cref{rem:property_of_operator}).
    \end{enumerate}
\end{theorem}

The remainder of this section deals with the proof of \cref{thm:main}, which is dissected into four~steps:
\begin{enumerate}[leftmargin=*]
    \item[(i)] Show that the discrete solution $\{\rho_i\}_{i=0}^n$ is well-defined and derive basic estimates on the corresponding covariance matrices and stepwise IGW and 2-Wasserstein displacements (\cref{subsec:min_exist}).
    \item[(ii)] Study the first-order optimality condition for $\{\rho_i\}_{i=0}^n$ to arrive at a discrete-time analogue of the above gradient flow equation by constructing a velocity field $v_i$ that relates $\rho_{i+1}$ and $\rho_i$~(\cref{subsec:discrete_optimality}). This is done by utilizing the concavity of IGW that leads to the variational form \cref{lem:F2Dual} and allows transforming the JKO step \eqref{eq:minimizing_movement} into a joint infimization over $\rho$, $\pi$, and $\bA$. Freezing the latter to the optimal $\bA^\star$ (see \cref{lem:F2Dual}) then enables to conduct the Wasserstein differentiation. We then obtain a quadratic equation in the discrete-time velocity field $v_i$, whose quadratic term vanishes as $\tau\to 0$, resulting in the linearization \eqref{eq:operator}, as explained in (iv) below. 
    \item[(iii)] Define the piecewise constant interpolation $\bar{v}_n(t)\coloneqq v_i$, for $t\in((i-1)\tau,i\tau]$ and $i=1,\ldots,n$, and show that as $n\to\infty$ (whence $\tau\to 0$), $(\bar{\rho}_n,\bar{v}_n)_{t\in[0,\delta]}$ weakly converges to a pair $(\rho_t,v_t)_{t\in[0,\delta]}$ that solves the continuity equation $\partial_t \rho_t + \nabla\cdot(\rho_t v_t) = 0$ (\cref{subsec:weak_convergence}).
    \item[(iv)] Strengthen the weak convergence from the previous step to 2-Wasserstein convergence (requires showing that $\bar{\bsigma}_n(t)\coloneqq\bsigma_{\rho_i}$, $t\in((i-1)\tau,i\tau]$, converges to $\bsigma_t$ as $n\to\infty$), using which the discrete-time gradient flow equation from Step (ii) can be transferred to the limit, thereby establishing the PIDE above (\cref{subsec:strong_convergence}). 
\end{enumerate}

These steps are stated and proven across several propositions and corollaries in the following subsections. Throughout these derivations, we state various technical lemmas concerning structural properties of IGW gradient flows and related objects (bounds, convexity, variational forms, etc.). Any such lemma holds under the conditions of the proposition/corollary being proven, possibly plus some local assumptions. To simplify the lemma statements, we do not repeat the general assumptions and only mention the local ones. Proofs of the technical lemmas are deferred to \cref{appen:IGW_gradientflow_aux_proofs}.

\begin{remark}[Analysis challenges and approach]
    Since $\IGW$ defines a discrepancy measure on $\cP_2(\RR^d)$ that is weaker than $\W_2$ (see \cref{lem:equivalence_igw_w}), our arguments do not directly follow from those in \cite{ambrosio2005gradient}, where 2-Wasserstein gradient flows are analyzed. Several marked differences are:
\begin{enumerate}[leftmargin=*]
    \item The $\IGW$ displacement interpolation, namely, the curve between $\mu_0,\mu_1$ defined by $\mu_t\coloneqq (g_t)_\sharp \pi^\star$, for $g_t(x,y)\coloneqq (1-t)x+ty$ and an IGW optimal $\pi^\star\in\Pi(\mu_0,\mu_1)$, is generally not an $\IGW$ nor $\W_2$ geodesic between $\mu_0,\mu_1$, and a priori, an $\IGW$ geodesic may not be realizable in~$\cP_2(\RR^d)$. 
    \item The IGW tangent space at $\mu\in\cP_2(\RR^d)$ is markedly different from that under $\W_2$. This precludes us from being able to define the subdifferential of the functional $\sF$ with respect to $\IGW $ in a natural way, as done for Wasserstein gradient flows. 
    \item As we consider a rotationally invariant setting (indeed, this is assumed for $\sF$ and holds for $\IGW$), the steps of the minimizing movement scheme are not unique. This renders most arguments from the $\W_2$ gradient flows analysis \cite{ambrosio2005gradient} inapplicable. We address this non-uniqueness using the orthogonal PSD transformation step described in our scheme above. 
\end{enumerate}
Given the above, we avoid direct usage of the subdifferential calculus from \cite{ambrosio2005gradient}, which is employed for Wasserstein gradient flows. Rather, we study the convergence and the limiting characterization of the IGW minimizing movement scheme from \eqref{eq:minimizing_movement} using the GMM framework from Definition 2.0.6 of \cite{ambrosio2005gradient}. We note that in general metric spaces, the GMM  coincides with a curve of maximal slope (see \cite[Theorems 2.3.1]{ambrosio2005gradient}) whenever upper gradients exist, which is the case for $\W_2$. \cite[Theorem 11.1.3]{ambrosio2005gradient} further implies that the curve of maximal slope coincides with the 2-Wasserstein gradient flow. 
Thus, the GMM generalizes gradient flows, and their equivalence in 2-Wasserstein case underpins the reason we adopt the GMM framework for the study of gradient flows in the IGW geometry. By the same token, $v_t$ from Part (2) of \cref{thm:main} is regarded as a \textit{generalized subdifferential} of $\sF$ w.r.t. $\IGW$; see \cref{sec:riemann} for a detailed discussion on the Riemannian structure of IGW spaces.
\end{remark}

\begin{remark}[Infinite-time extension]
    \label{rem:infinite_time_extension}
    \cref{thm:main} characterizes the GMM curve on a finite time interval $[0,\delta]$, where $\delta$ depends on the problem parameters. This choice ensures that the covariance matrix of any distribution within some IGW ball around $\mu_0$ remain nonsingular, which is crucial for our convergence analysis (see \cref{prop:existence_of_scheme}). Nevertheless, following standard arguments from ordinary differential equations (ODEs) and dynamical systems, we can show that this procedure can be repeated indefinitely. Specifically, upon running the scheme for time $\delta$ and obtaining the corresponding GMM $(\rho_t)_{t\in[0,\delta]}$, we may initiate another GMM at the end point of the previous run, namely, $\rho_\delta$. This procedure can be repeated so long as $\bsigma_{\rho_t}$ remain nonsingular, which is guaranteed by the existence of density from \cref{assumption:main}. We can thus define a 'finitely assembled' piecewise GMM curve via concatenation. The details of the construction are provided in \cref{appen:rem:infinite_time_extension}, where we prove that, unless the minimum $\sF^\star$ is achieved at a finite time $T$, the piecewise GMM curve can be extended to an interval of any length. 
\end{remark}

\subsection{Existence of minimizing movement sequences and stepwise estimates}\label{subsec:min_exist}

Recall that $\bsigma_{\rho_0}$ is nonsingular, and that $\lambda_{\max}(\bsigma_{\rho_0}),\lambda_{\min}(\bsigma_{\rho_0})>0$ denote, respectively, its largest and smallest eigenvalues.
Given a sequence of distributions $\{\rho_i\}_{i=0}^n\subset\cP_2(\RR^d)$, consider the dual representation of $\IGW(\rho_{i+1},\rho_i)$ from \cref{lem:F2Dual}, and write $\bA^\star_{i+1}$ for an optimizer. The lemma further states that $\bA^\star_{i+1}$ is induced by a coupling $\pi^\star_{i+1}\in\Pi(\rho_{i+1},\rho_i)$. Similar to the definition of $\bar{\rho}_n$, we consider the piecewise constant interpolation $\bar{\bA}_n(t)\coloneqq \bA^\star_i$, for $t\in((i-1)\tau,i\tau]$ and $i=1,\ldots,n$.

\begin{proposition}[Local existence in time]\label{prop:existence_of_scheme}
    Under \cref{assumption:main}(i), define $\bar{\delta} \coloneqq \frac{(1-1/\sqrt{2})\lambda_{\min}(\bsigma_{\mu_0})}{2^{1/4}}$. For any $0<\delta\leq \bar{\delta}^2/\big(2\sF(\mu_0) - 2\sF^\star\big)$ and $n\in\NN$, the discrete sequence $\{\rho_i\}_{i=0}^n$ obtained from the minimizing movement scheme \eqref{eq:minimizing_movement} with time step $\tau=\delta/n$ is well-defined. Furthermore:
    \begin{enumerate}[leftmargin=*]
        \item[(i)] for any $\tau>0$ as above, we have
        \begin{align}
        &\IGW (\rho_{i+1},\rho_i)^2\leq 2\tau\big(\sF(\rho_i) - \sF(\rho_{i+1})\big) \leq 2\tau \big(\sF(\rho_0) - \sF^\star\big),\quad \forall i=1,\ldots,n\label{eq:step_move:prop:existence_of_scheme}\\
        &\max_{i=1,\ldots,n}M_2(\rho_i)\leq \frac{4\sqrt{2} \delta \big(\sF(\rho_0) - \sF^\star\big)}{\lambda_{\min}(\bsigma_{\rho_0)}} + 2M_2(\rho_0) \label{eq:moment_bound:prop:existence_of_scheme}       \\
        &\frac{\lambda_{\min}(\bsigma_{\rho_0})}{2}\mspace{-1mu}\leq\mspace{-4mu} \lambda_{\min}(\bsigma_{\rho_i})\mspace{-4mu}\leq \mspace{-4mu}\lambda_{\max}(\bsigma_{\rho_i}) \mspace{-4mu}\leq \mspace{-6mu}\left(\mspace{-4mu}\sqrt{\mspace{-2mu}\lambda_{\max}(\bsigma_{\rho_0})} \mspace{-2mu}+\mspace{-2mu} 2\sqrt{\frac{\delta\big(\sF(\rho_0) \mspace{-2mu}-\mspace{-2mu} \sF^\star\big)}{\lambda_{\min}(\bsigma_{\rho_0})}} \right)^{\mspace{-5mu}2}\mspace{-5mu}\mspace{-4mu},  \forall i=1,\ldots,n;\label{eq:eigenval:prop:existence_of_scheme}
        \end{align}
        \item[(ii)] if further $\tau\leq \frac{\lambda_{\min}(\bsigma_{\rho_0})^4}{8(\sF(\rho_0)-\sF^\star)}$, then $\bA^\star_1,\ldots,\bA^\star_n$ are all nonsingular and satisfy
        \begin{align}
            \max_{i=1,\ldots,n}\|\bsigma_{\rho_{i+1}}-2\bA^\star_{i+1}\|_\F^2\vee \|\bsigma_{\rho_i}-2\bA^\star_{i+1}\|_\F^2 \lesssim_{\sF(\rho_0),\sF^\star,M_2(\rho_0),\lambda_{\min}(\bsigma_{\rho_0})} \tau \label{eq:cov_var_difference:prop:existence_of_scheme}.
        \end{align}
    \end{enumerate}
\end{proposition}

We briefly describe the main ideas of the proof, with the full derivation presented below. Our first step is to identify a radius $\bar{\delta}>0$, such that probability measures inside the IGW ball of that radius centered at $\rho_0$ are guaranteed certain regularity of their covariance matrix. This enables lower bounding the
eigenvalues of $\bsigma_{\rho_i}$ and upper bounding $M_2(\rho_i)$, for all $i=1,\ldots,n$. Having that, we invoke \cref{lem:equivalence_igw_w} and conclude the weak compactness of IGW balls around each $\rho_i$, thereby establishing the well-definedness of the sequence $\{\rho_i\}_{i=0}^n$. We bound the stepwise movement for each proximal step as $\IGW(\rho_{i+1},\rho_i)=O(\sqrt{\tau})$, from which we further conclude that each $2\bA^\star_i$ is close to $\bsigma_{\rho_{i+1}}$ and $\bsigma_{\rho_i}$ for $\tau$ small enough. Note that given our covariance matrix eigenvalue lower bound, \cref{lem:equivalence_igw_w} yields
\begin{align*}
    \W_2(\rho_{i+1},\rho_i)^2\leq\int \|x-y\|^2 d\pi^\star_{i+1}(x,y)\leq\frac{2}{\lambda_{\min}(\bsigma_{\rho_0})}\IGW(\rho_{i+1},\rho_i)^2,\numberthis\label{eq:igw_w_comparison_intermediate_bound}
\end{align*}
which is used repeatedly in the derivation.

\begin{proof}
    We start by deriving the total time length $\delta>0$ and the well-definedness of the discrete sequence. To that end, we present two technical lemmas---see Appendices \ref{appen::weak_topo_lsc_proof} and \ref{appen:bar_delta_proof} for the proofs. The first lemma establishes weak lower semi-continuity of the IGW distance and weak compactness of IGW balls. This result is then used to prove the well-definedness of discrete solution. The argument relies on reducing the IGW compactness to that under $\W_2$, whenever the center is in $\cP_2(\RR^d)$.

\begin{lemma}\label{lem:weak_topo_lsc}
    For any $\mu\in\cP_2(\RR^d)$, we have 
    \begin{enumerate}[leftmargin=*]
        \item[(i)] $\forall r>0$, $\cB_{\IGW}(\mu,r)$ is weakly compact, i.e., any sequence has a weakly convergent subsequence;
        \item[(ii)]   $\IGW(\mu,\nu)$ is weakly l.s.c. in $\nu$.
    \end{enumerate}
\end{lemma}

The next lemma provides quantitative control over the trace and smallest eigenvalue of $\bsigma_\mu$, for any distribution $\mu$ inside the IGW ball of radius $\bar{\delta}$ around $\rho_0=\mu_0$. To be able to use the convexity along generalized geodesics, the lemma further guarantees nonsingularity of the cross-covariance matrix between distributions inside that ball.

    \begin{lemma}\label{lem:bar_delta}
    For any $\mu,\nu\in\cP_2(\RR^d)\cap \cB_{\IGW}(\rho_0,\bar{\delta})$, we have
    \begin{enumerate}[leftmargin=*]
        \item[(i)] the following bounds:
        \[\lambda_{\min}(\bsigma_\mu)\geq \frac{\lambda_{\min}(\bsigma_{\rho_0})}{2}\quad \mbox{and} \quad M_2(\mu) \leq \frac{2\sqrt{2}}{\lambda_{\min}(\bsigma_{\rho_0})}\IGW (\mu,\rho_0)^2 + 2M_2(\rho_0);\]
        \item[(ii)] there exists an optimal IGW coupling $\pi^\star\in\Pi(\mu,\nu)$, such that $\int xy^\intercal d\pi^\star(x,y)$ is nonsingular with smallest singular value lower bounded by $\frac{\lambda_{\min}(\bsigma_{\rho_0})}{2\sqrt{2}}$.
    \end{enumerate}
    \end{lemma}

For proofs see \cref{appen::weak_topo_lsc_proof} and \cref{appen:bar_delta_proof}. Given these lemmas, we proceed to establish the existence of minimizers for the minimizing movement scheme \eqref{eq:minimizing_movement}, thereby addressing well-definedness. 
Pick an arbitrary $0<\delta \leq \bar{\delta}^2/\big(2\sF(\rho_0) -2 \sF^\star\big)$ and recall that $\tau=\delta/n$. Suppose, for now, that this choice of $\delta>0$ guarantees all the distributions mentioned below, stemming from the minimizing movement scheme from \eqref{eq:minimizing_movement}, have uniformly lower bound of eigenvalues of their covariance. Afterwards, we will show that this is indeed the case, given our choice of $\bar\delta$. Consider a sequence $\{\nu_k\}_{k\in\NN}$ that converges towards $\inf_\nu \sF(\nu) + \frac{1}{2\tau}\IGW(\nu,\rho_0)^2$. Since $\sF$ is lower bounded, $\sup_{k\in\NN}\IGW(\nu_k,\rho_0)^2$ is upper bounded. By weak compactness of IGW ball around $\rho_0$ and weak lower semi-continuity (follow from \cref{assumption:main}(i) and \cref{lem:weak_topo_lsc}), we conclude that $v_k\stackrel{w}{\to}\tilde\rho_1$, such that $\tilde \rho_1$ minimizes $\sF(\nu) + \frac{1}{2\tau}\IGW (\nu,\rho_0)^2$. As per the update rule from \eqref{eq:minimizing_movement}, we then pick $\bO\in\cO_{\rho_0,\tilde{\rho}_1}$, and define $\rho_1\coloneqq \bO_\sharp\tilde{\rho}_1$. By \cref{lem:symmetrization} there is an optimal IGW coupling $\pi^\star_1$ such that the corresponding IGW dual optimizer $\bA^\star_1\coloneqq\frac{1}{2}\int xy^\intercal d\pi^\star_1(x,y)$ is PSD. This construction is naturally extended to the subsequent steps of the minimizing movement scheme, so long that each iterate $\rho_{i+1}$ is computed from $\rho_i$ with $M_2(\rho_i)<\infty$, which enables using the weak compactness argument. Given $\rho_i$, for $i=1,\ldots,n-1$, we find a minimizer $\tilde{\rho}_{i+1}$ of $\sF(\nu) + \frac{1}{2\tau}\IGW(\nu,\rho_i)^2$ as above, and define $\rho_{i+1}\coloneqq \bO_\sharp\tilde{\rho}_{i+1}$ for some $\bO\in\cO_{\rho_i,\tilde{\rho}_{i+1}}$. Again, there exists a coupling $\pi^\star_{i+1}\in\Pi(\rho_{i+1},\rho_i)$ for which $\bA^\star_{i+1}\coloneqq\frac{1}{2}\int xy^\intercal d\pi^\star_{i+1}(x,y)$ is PSD and optimal for \eqref{eq:F2Decomp}. To conclude the well-defindedness part of the proposition, it remains to show that our choice of $\delta$ guarantees the uniform lower bound of $\lambda_{\min}(\bsigma_{\rho_i})$, we bound the stepwise movement.

Clearly  $\IGW (\rho_{i+1},\rho_i)^2 /(2\tau)\leq \sF(\rho_i) - \sF(\rho_{i+1})$ by minimality of $\rho_{i+1}$ for $\sF(\rho)+\frac{1}{2\tau}\IGW(\rho,\rho_i)^2$. This implies that $\sF(\rho_i)$ is nonincreasing in $i$, and recalling that $\sF^\star=\inf_\rho \sF(\rho)$, we further obtain $\IGW(\rho_{i+1},\rho_i)^2 \leq 2\tau(\sF(\rho_i) - \sF(\rho_{i+1})) \leq 2\tau \big(\sF(\rho_0) -\sF^\star\big)$, which implies~\eqref{eq:step_move:prop:existence_of_scheme}. For any step $j=1,\ldots,n$, summing over the intermediate steps, we obtain
\[\frac{1}{2\tau}\sum_{i=0}^{j-1}\IGW(\rho_{i+1},\rho_i)^2 \leq \sF(\rho_0) - \sF(\rho_{j}) \leq \sF(\rho_0) - \sF^\star,\]
which further implies 
\begin{equation}\label{eq:IGW_total_bound}
\IGW (\rho_j,\rho_0) \mspace{-2mu}\leq\mspace{-2mu} \sum_{i=0}^{j-1}\mspace{-2mu}\IGW (\rho_{i+1},\rho_i) \mspace{-2mu}\leq \mspace{-2mu}\sqrt{j\sum_{i=0}^{j-1}\mspace{-2mu}\IGW(\rho_{i+1},\rho_i)^2} \leq \mspace{-2mu}\sqrt{\mspace{-2mu}2j\mspace{-2mu}\tau\mspace{-2mu}\big(\sF\mspace{-2mu}(\rho_0\mspace{-2mu})\mspace{-4mu}-\mspace{-4mu}\sF^\star\big)} \leq \mspace{-2mu}\sqrt{\mspace{-2mu}2\delta\mspace{-2mu}\big(\mspace{-2mu}\sF\mspace{-2mu}(\rho_0\mspace{-2mu})\mspace{-4mu}-\mspace{-4mu}\sF^\star\big)}.
\end{equation}
Thus, the fact that $\delta \leq \bar{\delta}^2/\big(2\sF(\rho_0) -2 \sF^\star\big)$ ensures that for any $n\in\NN$, the entire sequence $\{\rho_i\}_{i=0}^n$ lies in $\cB_{\IGW}(\rho_0,\bar{\delta})$. By \cref{lem:bar_delta}, this guarantees uniformly lower bounded eigenvalues of all $\bsigma_{\rho_i}$, $i=1,\ldots,n$, as desired. 

\medskip
For Item (i), we have already established \eqref{eq:step_move:prop:existence_of_scheme}, which leaves \eqref{eq:moment_bound:prop:existence_of_scheme} and \eqref{eq:eigenval:prop:existence_of_scheme}. For the former, we plug in the bound from \eqref{eq:IGW_total_bound} into \eqref{eq:eigenval_ULB} and \eqref{eq:M2_bound}, respectively, to obtain
\begin{align*}
\lambda_{\max}(\bsigma_{\rho_i})&\leq \left(\sqrt{\lambda_{\max}(\bsigma_{\rho_0})} +  2^{3/4}\sqrt{\frac{\delta\big(\sF(\rho_0) - \sF^\star\big)}{\lambda_{\min}(\bsigma_{\rho_0})}} \right)^2\\
M_2(\rho_i) &\leq \frac{4\sqrt{2} \delta \big(\sF(\rho_0) - \sF^\star\big)}{\lambda_{\min}(\bsigma_{\rho_0})} + 2M_2(\rho_0),
\end{align*}
for all $i=1,\ldots,n$, which are the bounds in question. 

\medskip
It remains to prove Item (ii). For the nonsingularity of $\bA^\star_{i+1}$, which is PSD by \cref{lem:F2Dual}, we show that $2\bA^\star_{i+1}$ is close to $\bsigma_{\rho_i}$, which we know is nonsingular. Recall that $2\bA^\star_{i+1}=\int xy^\intercal d \pi^\star_{i+1}(x,y)$, where $\pi^\star_{i+1}\in\Pi(\rho_{i+1},\rho_i)$ is an optimal IGW coupling used in construction of $\rho_{i+1}$.
Using the notation from the proof of \cref{lem:equivalence_igw_w}, we write $\bP\blambda^\star\bP^\intercal=\int xy^\intercal d \pi^\star_{i+1}$ for the diagonalization of the cross-covariance matrix (indeed, recall that it is PSD by the choice of rotation) and denote its eigenvalues by $\lambda_1,\ldots,\lambda_d$. Let $\tilde{\pi}=(\bP^\intercal,\bP^\intercal)_\sharp\pi^\star_{i+1}$, and  write $a_1,\ldots,a_d$ and $b_1,\ldots,b_d$ for the diagonal entries of $\bP^\intercal\bsigma_{\rho_{i+1}}\bP$ and $\bP^\intercal\bsigma_{\rho_{i}}\bP$, respectively. Following same upper bounding steps from the proof of \cref{lem:equivalence_igw_w}, for each $m=1,\ldots,d$, we have 
\begin{align*}
   |a_m-b_m|&\leq\sqrt{2\int \|x\|^2 d \rho_{i+1}(x) +2\int \|y\|^2 d\rho_i(y)}\sqrt{\int\|x-y\|^2 d\tilde{\pi}(x,y)}\\
   &\leq \sqrt{2M_2(\rho_{i+1})+2M_2(\rho_i)} \sqrt{\frac{2}{\lambda_{\min}(\bsigma_{\rho_0})}}\IGW(\rho_{i+1},\rho_i)
\end{align*}
where we have used the fact that $\lambda_{\min}(\rho_i)\geq \lambda_{\min}(\bsigma_{\rho_0})/2$,  for all $i=1,
\ldots, n$.

Since $\tau\leq \frac{\lambda_{\min}(\bsigma_{\rho_0})^4}{8\big(\sF(\rho_0)-\sF^\star\big)}$ by assumption, we have $\IGW(\rho_{i+1}\mspace{-2mu},\mspace{-2mu}\rho_i)\mspace{-4mu}\leq\mspace{-4mu} \sqrt{2\tau\big(\sF(\rho_0)-\sF^\star\big)}\mspace{-2mu}\leq\mspace{-2mu} \frac 12\lambda_{\min}(\bsigma_{\rho_0})^2$. By the relation $a_m^2+b_m^2-2\lambda_m^2\leq \IGW^2(\rho_{i+1},\rho_i)$, which holds for all $m=1,\dots,d$ and $i=0,\ldots,n-1$, as shown in \eqref{eq:IGW_primal_decomp_lower_bound},
we further obtain

\begin{equation}\label{eq:eigenval_bound_in_IGW_ball}
\sqrt{\frac{a_m^2+b_m^2-\IGW^2(\rho_{i+1},\rho_i)}{2}} \leq \lambda_m\leq\frac{a_m+b_m}{2}\quad\mbox{and}\quad  |\lambda_m-b_m|\lesssim_{\sF(\rho_0),\sF^\star,M_2(\rho_0),\lambda_{\min}(\bsigma_{\rho_0})}\sqrt{\tau}.
\end{equation}
Thus, for small enough $\tau$, all the $\lambda_m$'s are positive, which implies that $\bA^\star_{i+1}$ is nonsingular.

Having control over the sum of the diagonal entries of $\bP^\intercal\bsigma_{\rho_{i+1}}\bP$ by $\IGW(\rho_{i+1},\rho_i)^2$, we derive~\eqref{eq:cov_var_difference:prop:existence_of_scheme}:
\begin{align*}
    \|\bsigma_{\rho_{i+1}}&-2\bA^\star_{i+1}\|_\F^2\\
    &\leq \IGW (\rho_{i+1},\rho_i)^2 + \sum_{m=1}^d |a_m-\lambda_m|^2 \\
    &\leq \IGW(\rho_{i+1},\rho_i)^2 + \frac{1}{2}\sum_{m=1}^d \left(|a_m + b_m - 2\lambda_m|^2 + |a_m-b_m|^2\right)\\
    &\leq \IGW(\rho_{i+1},\rho_i)^2 + \frac{1}{2}\sum_{m=1}^d |a_m + b_m - 2\lambda_m|^2 + \frac{4M_2(\rho_{i+1})+4M_2(\rho_i)}{\lambda_{\min}(\bsigma_{\rho_0})}\IGW(\rho_{i+1},\rho_i)^2\\
    &\leq \left(2+\frac{2M_2(\rho_{i+1})+2M_2(\rho_i)}{\lambda_{\min}(\bsigma_{\rho_0})}\right)\IGW(\rho_{i+1},\rho_i)^2\\
    &\lesssim_{\sF(\rho_0),\sF^\star,M_2(\rho_0),\lambda_{\min}(\bsigma_{\rho_0})} \tau,
\end{align*}
where we have used $\sum_{m=1}^d|a_m-b_m|^2\leq \big(4M_2(\rho_{i+1})+4M_2(\rho_i)\big)\frac{\IGW(\rho_{i+1},\rho_i)^2}{\lambda_{\min}(\bsigma_{\rho_0})}$, and $|a_m+b_m-2\lambda_m|^2\leq 2\big|\sqrt{a_m^2+b_m^2}-\sqrt{2}\lambda_m\big|^2 \leq 2(a_m^2+b_m^2-2\lambda_m^2)$. Similarly, we obtain $\|\bsigma_{\rho_i}-2\bA^\star_{i+1}\|_\F^2 \lesssim_{\sF,\rho_0} \tau$, which concludes the proof \end{proof}

\subsection{Discrete-time optimality condition}\label{subsec:discrete_optimality}

\cref{thm:main} shows that, in the continuous-time limit, the gradient of the first variation at time $t$ is given by $-\cL_{\Sigma_t,\rho_t}[v_t]$. Towards proving this limiting relationship, we now show that the corresponding transform of the discrete-time vector field $v_i$ (defined below) satisfies a similar first-order optimality condition, namely, it belongs to the Fr\'echet subdifferential of the objective function at $\rho_i$. The next result is a parallel of \cite[Lemma 10.1.2]{ambrosio2005gradient}. 

\begin{proposition}[First-order optimality]\label{prop:discrete_vector_field_eq}
    Under \cref{assumption:main}(i)-(ii) and the condition on $\delta>0$ from \cref{prop:existence_of_scheme}, for each $i=1,\ldots,n$, we have:
    \begin{enumerate}[leftmargin=*]
        \item[(i)] $T^\star_{i}\coloneqq (8\bA^\star_{i})^{-1}T^{\rho_{i}\to(8\bA^\star_{i})_\sharp\rho_{i-1}}$ is a Gromov-Monge map from $\rho_{i}$ to $\rho_{i-1}$, i.e., $(\id,T^\star_{i})_\sharp\rho_{i}=\pi^\star_{i}$ such that $\bA^\star_{i}=\frac{1}{2}\int x\, T^\star(x)^\intercal d\rho_{i}(x)$;
        \item[(ii)] $\sF$ is Fr\'echet differentiable at $\rho_{i}$;
        \item[(iii)] for the discrete-time velocity field $v_i:x\mapsto \tau^{-1}\big(x-T^\star_i(x)\big),\ \RR^d\to\RR^d$, where $i=1,\ldots,n$, and the (uncentered) cross-covariance matrix $\bL_{i}\coloneqq\int x v_i(x)^\intercal d\rho_{i}(x)\in\RR^{d\times d}$, we have that     
    \begin{align*}
        -\cL_{2\bA^\star_i,\rho_i}[v_i] = -2(\bsigma_{\rho_{i}} - \tau \bL_{i})v_i-2\bL_{i} \id \in\partial\sF(\rho_{i})
    \end{align*}
    is a strong subdifferential.
    \end{enumerate}
\end{proposition}
\begin{proof}
    The first item is a direct consequence of \cref{lem:gromov_monge}, where the nonsingularity of $\bA^\star_{i+1}$ follows from \cref{prop:existence_of_scheme}. 

    \medskip
    For Item (ii), note that $\rho_i\in\cP_2(\RR^d)$ by \eqref{eq:moment_bound:prop:existence_of_scheme} and consider
\begin{align*}
    &\inf_{\rho\in \cP_2(\RR^d)} 2\tau \sF(\rho) + \IGW^2 (\rho,\rho_i)\\
    &=  \inf_{\rho\in \cP_2(\RR^d)} 2\tau \sF(\rho) + \int \llangle x,x'\rrangle^2  d \rho\otimes \rho(x,x') + \int \llangle y,y'\rrangle^2  d \rho_i\otimes \rho_i(y,y'), \\
    &\qquad\qquad\qquad\qquad\qquad\qquad\qquad\qquad+ \inf_{\bA\in \RR^{d\times d}} \left\{8\|\bA\|_\F^2 + \inf_{\pi\in\Pi(\rho,\rho_i)} \int -8x^\intercal \bA y  d \pi(x,y)\right\}\\
    &=\int \llangle y,y'\rrangle^2  d \rho_i\otimes \rho_i(y,y')   + 8\|\bA^\star_{i+1}\|_\F^2\\
    &\quad\quad\quad+\inf_{\rho\in \cP_2(\RR^d)}\left\{ 2\tau \sF(\rho) + \int \llangle x,x'\rrangle^2  d \rho\otimes\rho(x,x') +  \inf_{\pi\in\Pi(\rho,\rho_i)} \int -8x^\intercal \bA^\star_{i+1} y  d \pi(x,y)\right\},    
\end{align*}
where the last step interchanges the order of the infima over $\rho$ and $\bA$ (as the corresponding domains are noniteracting), and $\bA_{i+1}^\star$ is an optimizer of the dual representation (\cref{lem:F2Dual}) of $\IGW(\rho_{i+1},\rho_i)$, as clearly $\rho_{i+1}$ solves the optimization problem in the last line. 
    Using the shorthand $\rho^{\otimes 2}=\rho\otimes\rho$, the minimality of $\rho_{i+1}$ further implies that for any $T\in L^2(\rho_{i+1};\RR^d)$, we have 
    \begin{align*}
        &2\tau\big(\sF(T_\sharp\rho_{i+1})-\sF(\rho_{i+1})\big)\\
        &\geq \int \llangle x,x'\rrangle^2  d \rho_{i+1}^{\otimes 2}(x,x') + \inf_{\pi\in\Pi(\rho_{i+1},\rho_i)} \int -8x^\intercal \bA^\star_{i+1} y  d \pi(x,y) \\
        &\qquad- \int \llangle x,x'\rrangle^2  d\, (T_\sharp\rho_{i+1})^{\otimes 2}(x,x') - \inf_{\pi\in\Pi(T_\sharp\rho_{i+1},\rho_i)} \int -8x^\intercal \bA^\star_{i+1} y  d \pi(x,y)\\
        &\geq  \int \left(\llangle x,x'\rrangle^2 - \llangle T(x),T(x')\rrangle^2  \right)d \rho_{i+1}^{\otimes 2}(x,x') + \int \llangle 8\bA^\star_{i+1}T^\star_{i+1}(x),T(x)-x\rrangle  d\rho_{i+1}(x)\\
        &=\mspace{-5mu}\int\mspace{-5mu}\left(\llangle x,x'\rrangle^2 \mspace{-4mu}-\mspace{-4mu} \big(\llangle T(x)\mspace{-4mu}-\mspace{-4mu}x,T(x')\mspace{-4mu}-\mspace{-4mu}x'\rrangle\mspace{-4mu} +\mspace{-4mu} \llangle T(x)\mspace{-4mu}-\mspace{-4mu}x,x'\rrangle\mspace{-4mu} +\mspace{-4mu} \llangle x,T(x')\mspace{-4mu}-\mspace{-4mu}x'\rrangle\mspace{-4mu} +\mspace{-4mu} \llangle x,x'\rrangle\big)^2\right)  d \rho_{i+1}^{\otimes 2}(x,x') \\
        &\qquad+ \int \llangle 8\bA^\star_{i+1}T^\star_{i+1}(x),T(x)-x\rrangle  d\rho_{i+1}(x)\\
        &=  \int -4\llangle T(x)-x,x'\rrangle \llangle x,x'\rrangle d \rho_{i+1}^{\otimes 2}(x,x') + o\big(\|T(x)-x\|_{L^2(\rho_{i+1};\RR^d)}\big)\\
        &\qquad+ \int \llangle 8\bA^\star_{i+1}T^\star_{i+1}(x),T(x)-x\rrangle  d\rho_{i+1}(x)\\
        &= \int \llangle 8\bA^\star_{i+1}T^\star_{i+1}(x) - 4\bsigma_{\rho_{i+1}}x,T(x)-x\rrangle  d\rho_{i+1}(x) + o\big(\|T(x)-x\|_{L^2(\rho_{i+1};\RR^d)}\big),
    \end{align*}
    where we have used the Cauchy–Schwarz inequality. 
    This concludes the proof of Fr\'echet subdifferentiability, and further implies that $\tau^{-1}\left(4\bA^\star_{i+1}T^\star_{i+1} - 2\bsigma_{\rho_{i+1}} x\right) \in\partial\sF(\rho_{i+1})$ is a strong subdifferential.

\medskip
    To establish the last item, write the operator from \eqref{eq:operator} as $\cL_{2\bA^\star_i,\rho_i}[v_i] = 4\bA^\star_i v_i+2\bL_{i} \id$ and note that, for any $x\in\RR^d$, we have
    \begin{align*}
        -\cL_{2\bA^\star_i,\rho_i}[v_i](x) &= -2(\bsigma_{\rho_{i}}-\tau \bL_{i})v_i(x)-2\bL_{i}x\\
        &=\frac{1}{\tau}\big( 2(\bsigma_{\rho_{i}}-\tau \bL_{i})   \big(x-\tau v_i(x)\big)- 2\bsigma_{\rho_{i}} x \big)\\
        &=\frac{1}{\tau}\left( 2 \int y\big(y-\tau v_i(y)\big)^\intercal d\rho_{i}(y)  (x-\tau v_i(x))- 2\bsigma_{\rho_{i}} x \right) \\
        &=\frac{1}{\tau}\left(4\bA^\star_{i}T^\star_{i}(x) - 2\bsigma_{\rho_{i}} x\right),
    \end{align*}        
        where the latter is indeed a strong subdifferential, as shown before.\end{proof}

To pass this optimality condition to the limit, which we do in \cref{subsec:strong_convergence}, we consider the convergence of the piecewise constant interpolation $\big(\bar{v}_n(t),\bar{\bsigma}_n(t), \bar{\bL}_n(t)\big)\coloneqq 
(v_i,\bsigma_{\rho_i},\bL_{i})$, for $t\in((i-1)\tau,i\tau]$ and $i=1,\ldots,n$. With this definition,  \cref{prop:discrete_vector_field_eq} is equivalently stated~as
\begin{align*}
    -\cL_{2\bar{\bA}_n(t),\bar{\rho}_n(t)} [\bar{v}_n(t,\cdot)]=-2\big(\bar{\bsigma}_n(t)-\tau \bar{\bL}_n(t)\big)\bar{v}_n(t) - 2\bar{\bL}_n(t) \id  \in\partial\sF\big(\bar{\rho}_n(t)\big),\quad \forall t\in[0,\delta].
\end{align*}
Evidently, the term $-\tau \bar{\bL}_n\bar{v}_n$ in the above operator is quadratic in $\bar{v}_n$, and is present due to the quadratic nature of IGW. Nevertheless, when we drive $n\to\infty$, this quadratic term will vanish and the limit of $-\cL_{2\bar{\bA}_n,\bar{\rho}_n} [\bar{v}_n]$ will coincide with $-\cL_{\bsigma_t,\rho_t}[v_t]$ from \eqref{eq:operator}. In particular, note that each $2\bar{\bA}_n(t)$ is the (uncentered) cross-covariance matrix between two consecutive steps of the discrete-time sequence $\{\rho_i\}_{i=0}^n$. As $n\to\infty$, i.e., the step size $\tau\to 0$, $2\bar{\bA}_n(t)$ converges to the auto-covariance matrix $\bsigma_t=\bsigma_{\rho_t}$ at time $t$. 

\medskip

While the formal derivation of the mentioned convergence uses the strong Fr\'echet subdifferential, \cref{prop:discrete_vector_field_eq} enables relating the IGW discrete-time velocity field to the gradient of the first variation of $\sF$, analogously to the continuous-time relation from Item (2) of \cref{thm:main}.

\begin{corollary}[Discrete-time gradient flow equation] 
\label{cor:opt_condition}
    Under Assumptions \ref{assumption:main}(i)-(ii) and \ref{assumption:differentiability}, along with the conditions from \cref{prop:existence_of_scheme}, we have \begin{equation}\label{eq:first_variation_discrete_vector_field_eq}
        \nabla\delta\sF(\rho_i)=-\cL_{2\bA^\star_i,\rho_i}[v_i], \quad \forall i=1,\ldots,n.
    \end{equation}
\end{corollary}

This is an immediate consequence of \cref{prop:discrete_vector_field_eq} and thus the proof is omitted. It is interesting, however, to comment on a heuristic derivation of the first-order optimality condition from \eqref{eq:first_variation_discrete_vector_field_eq}, via variational analysis.\footnote{Overlooking the various regularity and boundary conditions that are required for making this argument rigorous.} Indeed, the optimality of $\rho_{i+1}$ implies that for some constant $c\in\RR$, we have 
\begin{align*}
    2\tau \delta\sF(\rho_{i+1}) + 2x^\intercal\bsigma_{\rho_{i+1}} x + \varphi^\star = c, \quad x\in\RR^d,
\end{align*}
where $\varphi^\star,\psi^\star$ are the optimal potentials for the OT problem $\inf_{\pi\in\Pi(\rho,(8\bA^\star_{i+1})_\sharp\rho_i)} \int -x^\intercal y  d \pi(x,y)$. 
Note that the Gromov-Monge map from $\rho_{i+1}$ 
to $\rho_i$ is $T^\star_{i+1}(x) = -(8\bA^\star_{i+1})^{-1} \nabla \varphi^\star$ (c.f., e.g., Theorems 9.4 and 10.28 in \cite{villani2008optimal}), and the invertibility of $\bA^\star_{i+1}$ was established in \cref{prop:existence_of_scheme}. 
Taking the gradient of the optimality equation above, we obtain $2\tau \nabla\delta\sF(\rho_{i}) + 4\bsigma_{\rho_{i}} x - 8\bA^\star_{i} T^\star_{i} = 0$. Upon rearranging, this leads to
\[-\cL_{2\bA^\star_{i},\rho_i} [v_i] = \nabla\delta\sF(\rho_i),\quad i=1,\ldots,n,
\]
which coincides with \eqref{eq:first_variation_discrete_vector_field_eq}, as desired.

\subsection{Weak convergence of discrete solutions and velocity fields}\label{subsec:weak_convergence}

We now study the convergence of the piecewise constant interpolation of the discrete-time sequences of measures $\{\bar{\rho}_n\}_{n\in\NN}$ and velocity fields $\{\bar{v}_n\}_{n\in\NN}$ along the minimizing movement scheme from \eqref{eq:minimizing_movement}. For now, we account for weak convergence of each sequence individually, along properly chosen common subsequence. In \cref{subsec:strong_convergence}, we shall strengthen the claim to uniform  $\W_2$ convergence of the discrete-time solution (along a subsequence) and demonstrate that the optimality condition from \cref{cor:opt_condition}, which connects $(\rho_i,v_i)$, also transfers to the continuous-time limit.

\subsubsection{Pointwise weak convergence of $\bar{\rho}_n$}
We show that the piecewise constant interpolation $\bar{\rho}_n:[0,\delta]\to\cP_2(\RR^d)$ converges (up to a subsequence in $n$) to a curve in $\cP_2(\RR^d)$ that satisfies a generalized continuity equation. Combined with lower semi-continuity of $\sF$, this further implies $\rho\subset\dom(\sF)$.

\begin{proposition}[Pointwise weak convergence of $\bar\rho_n$]\label{prop:limit_of_minimizing_movement}
Under \cref{assumption:main}(i)-(ii) and the condition on $\delta>0$ from \cref{prop:existence_of_scheme}, we have that $\{\bar{\rho}_n\}_{n\in\NN}$ converges pointwise weakly, along a subsequence, to a uniformly $\W_2$-continuous curve $\rho:[0,\delta]\to\cP_2(\RR^d)$.
\end{proposition}
\begin{proof}
    We first bound the 2-Wasserstein distance between different time steps of the piecewise constant curve and then invoke a generalized Arzel\`a-Ascoli theorem to conclude. Fix $n\in\NN$, $s<t$, and consider
    \begin{align*}
    \W_2\big(\bar{\rho}_n(s),\bar{\rho}_n(t)\big) &= \W_2(\rho_{\lceil s/\tau \rceil}, \rho_{\lceil t/\tau \rceil}) \\
    &\leq \frac{\sqrt{2}\sum_{i=\lceil s/\tau \rceil}^{\lceil t/\tau \rceil-1}\IGW(\rho_{i+1},\rho_i)}{\sqrt{\lambda_{\min}(\bsigma_{\rho_0})}} \\
    &\leq 2\sqrt{\frac{\tau|\lceil s/\tau \rceil-\lceil t/\tau \rceil|(\sF(\rho_0)-\sF^\star)}{\lambda_{\min}(\bsigma_{\rho_0})}},
    \end{align*}
    where we have used the fact that for $k>j$, $\sum_{i=j}^{k-1}\IGW(\rho_{i+1},\rho_i)\leq\sqrt{2\tau(k-j)(\sF(\rho_0)-\sF^\star)}$, which follows a similar argument as \eqref{eq:IGW_total_bound} in proof of \cref{prop:existence_of_scheme}.
    Notice that $\lim_{n\to\infty} |\lceil s/\tau \rceil-\lceil t/\tau \rceil|\tau = |s-t|$, and define the modulus of continuity \[\omega(s,t) = 2\sqrt{\frac{|s-t|(\sF(\rho_0) - \sF^\star)}{\lambda_{\min}(\bsigma_{\rho_0})}}.\]
    By \cite[Proposition 3.3.1]{ambrosio2005gradient}, which provides a refinement of the Arzel\`a-Ascoli theorem, we conclude that there is a $\W_2$-continuous curve $\rho:[0,\delta]\to\cP_2(\RR^d)$ and a subsequence $\bar{\rho}_{n_k}$ such that $\bar{\rho}_{n_k}(t)\stackrel{w}{\to}\rho_t$ and $\W_2\big(\rho_t,\rho_s\big)\leq \omega(s,t)$, for all $t,s\in[0,\delta]$.
\end{proof}

\subsubsection{Convergence of velocity fields and continuity equation}
                
The convergence of the vector fields $\{\bar{v}_n\}_{n\in\NN}$ requires special care. Note that for each $(n,t)\in\NN\times [0,\delta]$, we have $\bar{v}_n(t)\in L^2(\bar{\rho}_n(t);\RR^d)$, which makes it unclear what would be a natural $L^2$ space in which to expect convergence. Instead, we shall arrive at a limiting statement by 'measurizing' the vector fields, and establishing weak convergence. This idea is inspired by \cite[Theorem 11.1.6]{ambrosio2005gradient}. With a slight abuse of notation, we write $\bar\rho_n(t,x)$ for $(\bar\rho_n(t))(x)$ as a measure of variable $x$, and recall that $\rho_t$ denotes the continuous-time limit (derived in \cref{prop:limit_of_minimizing_movement}) at time $t\in[0,\delta]$.

\medskip
For $\delta>0$, denote by $\upsilon_\delta$ the uniform distribution on $[0,\delta]$. Note that the distribution\footnote{Specifically, let $T\sim \upsilon_\delta$ and given $T=t$, let $X_n(t)\sim \bar{\rho}_n(t)$, so that the joint distribution of $\big(T,X_n(T)\big)$ is $\upsilon_\delta\bar{\rho}_n\in \cP([0,\delta]\times\RR^d)$} $\upsilon_\delta\bar{\rho}_{n_k}\in \cP([0,\delta]\times\RR^d)$ converges weakly, along the subsequence, to $\upsilon_\delta\rho$, where $\rho$ is the weak limit. Indeed, for any continuous function $g:[0,\delta]\times \RR^d\to\RR$ with $\|g\|_{\infty}\leq C$, define $f_n(t)\coloneqq \int_{\RR^d} g(t,x) d\bar{\rho}_n(t,x)$. Clearly $|f_n|\leq C$ and $\lim_{k} f_{n_k}(t) = \int_{\RR^d} g(t,x) d\rho_t(x)$. Thus, by the dominated convergence theorem
\begin{align*}
    \lim_k \int_0^\delta \int_{\RR^d} g(t,x) d\bar{\rho}_{n_k}(t,x) dt = \lim_k \int_0^\delta f_{n_k}(t) dt = \int_0^\delta \int_{\RR^d} g(t,x) d\rho_t(x) dt,
\end{align*}
i.e., $\upsilon_\delta\bar{\rho}_{n_k} \stackrel{w}{\to}\upsilon_\delta\rho$.

Define measure $\nu_n \coloneqq \upsilon_\delta \left[(\id,\bar{v}_n)_\sharp\bar{\rho}_n\right]\in\cP([0,\delta]\times\RR^d\times\RR^d)$; in the notation from the footnote below, we have $\big(T,X_n(T),\bar{v}_n(T,X_n(T))\big)\sim\nu_n$. Note that
\begin{align*}
\int\limits_{[0,\delta]\times\RR^d\times\RR^d}\mspace{-25mu} \big(\,t^2 + \|x\|^2 + \|y\|^2\big)  d  \nu_n(t,x,y) = \frac{\delta^2}{3} + \frac{\tau }{\delta}\sum_{i=1}^n M_2(\rho_i) + \frac{1}{\delta}\int_0^\delta \int_{\RR^d} \|\bar{v}_n(t,x)\|^2 d \bar{\rho}_n(t,x)  d  t,
\end{align*}
and
\begin{align*}
    \int_0^\delta \int \|\bar{v}_n(t,x)\|^2 d \bar{\rho}_n(t,x)  d  t &=  \sum_{i=1}^{n} \int \tau\left\|\frac{T^\star_i(x)-x}{\tau}\right\|^2 d \rho_{i}(x)\\
    &\leq \sum_{i=1}^{n} \frac{2}{\tau\lambda_{\min}(\bsigma_{\rho_0})}\IGW(\rho_{i-1},\rho_{i})^2\\
    &\leq \frac{4(\sF(\rho_0)- \sF^\star)}{\lambda_{\min}(\bsigma_{\rho_0})}.\numberthis\label{eq:vector_field_L2_bound}
\end{align*}
Thus $\{\nu_n\}_{n=0}^\infty\subset\cP_2([0,\delta]\times\RR^d\times\RR^d)$ is tight, and we can find a subsequence $\{\nu_{n_k}\}_{k=0}^\infty$ (a further subsequence of the one from \cref{prop:limit_of_minimizing_movement}, not relabeled to simplify notation) and a weak limit $\nu\in\cP_2([0,\delta]\times\RR^d\times\RR^d)$, whose marginal over the first two variables is $\upsilon_\delta\rho$. Let $\nu_{t,x}\in\cP(\RR^d)$ be the
disintegration of $\nu$ w.r.t. the marginal $\upsilon_\delta\rho$ (namely, the conditional distribution of the third coordinate given that the first two take the value $(t,x)\in[0,\delta]\times\RR^d$. Define the velocity field $v:[0,\delta]\times\RR^d\to\RR^d$ as the conditional expectation
\begin{align}
    v_t(x) \coloneqq \int y  d  \nu_{t,x}(y).\label{eq:limit_of_vector_fields}
\end{align}

We conclude this subsection by showing that the velocity field $v$ from \eqref{eq:limit_of_vector_fields} and the path $\rho:[0,\delta]\to \cP_2(\RR^d)$ solves the continuity equation, in the distributional sense. The derivation follows ideas similar to \cite[Proof of Theorem 11.1.6, Step 4]{ambrosio2005gradient}.

\begin{proposition}[Continuity equation]\label{prop:limit_continuity_equation}
Under the conditions from \cref{prop:limit_of_minimizing_movement}, the limiting pair $(\rho,v)_{t\in[0,\delta]}$ above satisfies the continuity equation in the distributional sense, i.e., 
    \[  \int_0^\delta\int_{\RR^d} \partial_t g(t,x) d \rho_t(x)dt=-\int_0^\delta\int_{\RR^d} \llangle\nabla g(t,x), v_t(x)\rrangle d  \rho_t(x)dt,\quad \forall g\in C_c^\infty((0,\delta)\times\RR^d).
    \]
\end{proposition}
\begin{proof}
    For any $t<0$, set $\bar{\rho}_n(t,\cdot)=\rho_0$, and write $\tau_k=\frac{\delta}{n_k}$, with respect to the subsequence mentioned after \eqref{eq:vector_field_L2_bound}.
    Pick a smooth and compactly supported test function $g$ on $(0,\delta)\times\RR^d$, and compute
    \begin{align*}
        \iint \partial_t g(t,x) d \rho_t(x)dt &= \lim_{k\to\infty} \iint \partial_t g(t,x) d \bar{\rho}_{n_k}(t,x)dt\\
        &= \lim_{k\to\infty} \frac{1}{\tau_k}\iint \big(g(t+\tau_k,x)-g(t,x)\big) d \bar{\rho}_{n_k}(t,x)dt\\
        &= \lim_{k\to\infty} \frac{1}{\tau_k}\iint g(t,x) d \big(\bar{\rho}_{n_k}(t-\tau_k,x)-\bar{\rho}_{n_k}(t,x)\big)dt\\
        & \stackrel{\mathrm{(a)}}{=}  \lim_{k\to\infty} \frac{1}{\tau_k}\iint \big(g(t,x-\tau_k\bar{v}_{n_k}) - g(t,x)\big) d \bar{\rho}_{n_k}(t,x)dt\\
        & \stackrel{\mathrm{(b)}}{=}  \lim_{k\to\infty} -\iint \llangle\nabla g(t,x),\bar{v}_{n_k}(t,x)\rrangle d \bar{\rho}_{n_k}(t,x)dt\\
        & = \delta\lim_{k\to\infty} -\int \llangle\nabla g(t,x), y\rrangle d  \nu_{n_k}(t,x,y)\\
        & \stackrel{\mathrm{(c)}}{=}  -\delta\int \llangle\nabla g(t,x), y\rrangle d  \nu(t,x,y)\\
        & = -\int \llangle\nabla g(t,x),v_t(x)\rrangle d  \rho_t(x)dt,
    \end{align*}
where (a) follows by the fact that $\rho_{i-1}=(x-\tau_k v_i)_\sharp\rho_i$; (b) is because for any $\tau>0$, we have
\begin{align*}
  \left|\iint\mspace{-3mu}\big( g(t,x-\tau\bar{v}_n) \mspace{-2mu}-\mspace{-2mu} g(t,x)\mspace{-2mu}+\mspace{-2mu} \tau\llangle\nabla g(t,x),\bar{v}_n(t,x)\rrangle\big) d \bar{\rho}_n(t,x)dt \right| &\leq c_g \tau^2 \mspace{-2mu}\int \mspace{-2mu}\|\bar{v}_n(t,x)\|^2 d \bar{\rho}_n(t,x)dt \\
  &\leq c_g \frac{4(\sF(\rho_0)- \sF^\star)}{\lambda_{\min}(\bsigma_{\rho_0})}\tau^2,
\end{align*}
where $c_g>0$ is a constant depending only on second derivatives of $g$; while (c) uses the fact that since $M_2(\nu_n)$ are uniformly bounded, any function $f(t,x,y)$ with at most first-order growth (viz., $|f(t,x,y)|\lesssim |t|+\|x\|+\|y\|$) verifies $\lim_k \int f(t,x,y)d\nu_{n_k}(t,x,y)=\int f(t,x,y)d\nu(t,x,y)$.
\end{proof}

\subsection{2-Wasserstein convergence of discrete solutions}\label{subsec:strong_convergence}
 
From the previous section, we know that the piecewise constant sequences $\bar\rho_n$ and $\bar v_n$ converge weakly to their respective continuous-time limits $\rho,v$ along a subsequence, with the limiting pair satisfying the the continuity equation. From \cref{cor:opt_condition}, we also know that for each $n\in\NN$, the discrete-time solutions satisfy an optimality condition that ties them to one another. Transferring this relationship to the continuous-time limit requires strengthening the notion of convergence of discrete solutions, from weak to $\W_2$ convergence. 

The argument hinges on controlling the 2-Wasserstein distance by IGW using \cref{lem:equivalence_igw_w}, and then deriving the convergence rate of the latter by showing that $\{\bar\rho_n\}_{n\in\NN}$ is a Cauchy sequence in IGW. To that end, we use the local convexity of IGW along generalized geodesics (see \cref{lem:local_convexity_igw}) and the cross-partition comparison method from \cite[Corollary 4.1.5]{ambrosio2005gradient}.

\subsubsection{Cross-partition sequence comparison}

A key technical step towards deriving the 2-Wasserstein convergence of discrete solutions $\bar\rho_n:[0,\delta]\to\cP_2(\RR^d)$ to the continuous-time limit $\rho$, is a Cauchy-type estimate of the uniform (in time) IGW-gap. To emphasize the dependence on $n$, or, equivalently, the time-step $\tau=\frac{\delta}{n}$, write $\{\rho^\tau_i\}_{i=0}^n$ for a solution to \eqref{eq:minimizing_movement} with parameter $\tau$ and denote by $\bar{\rho}_n:[0,\delta]\to\cP_2(\RR^d)$ the corresponding piecewise constant interpolation. For two such discrete-time sequences with parameters $\tau$ and $\eta$, we derive a uniform $O(\sqrt{\tau}+\sqrt{\eta})$ bound on the IGW-gap between them.

To that end, we define a cross-partition `error' function, which will be used to control the IGW-gap. Fix $n\in\NN$, and define the piecewise linear function $\ell_\tau:[0,\delta]\to\RR$ by $\ell_\tau(t)\coloneqq \frac{t-(i-1)\tau}{\tau}$, for $i=1,\ldots,n,\ t\in ((i-1)\tau,i\tau]$. For $\nu\in\cP_2(\RR^d)$, set 
\[d_\tau(t;\nu)^2\coloneqq (1-\ell_\tau(t))\IGW(\rho^\tau_{i-1},\nu)^2 + \ell_\tau(t) \IGW(\rho^\tau_i,\nu)^2, \quad \ i=1,\ldots,n,\ t\in((i-1)\tau,i\tau].\]

\begin{definition}[Cross-partition error function]
    For $n,m\in\NN$ with $\tau=\frac{\delta}{n}$ and $\eta=\frac{\delta}{m}$, define the function $d_{\tau\eta}:[0,\delta]^2\to\RR$ by 
    \begin{equation}
    d_{\tau\eta}(t,s)^2\coloneqq (1-\ell_\eta(s))d_\tau(t;\rho^\eta_{j-1})^2 + \ell_\tau(s)d_\tau(t;\rho^\eta_{j})^2,\label{eq:error_func}
    \end{equation}
    for $(t,s)\in((i-1)\tau,i\tau]\times ((j-1)\eta,j\eta]$ and $(i,j)\in\{1,\ldots,n\}\times \{1,\ldots,m\}$.  
\end{definition}

We shall control the time derivative of the function $d_{\tau\eta}(t,t)^2$, which, together with the Gr\"onwall lemma, leads to a uniform bound in $t$ on $\IGW(\bar{\rho}_n(t),\bar{\rho}_m(t))^2$ that decays to 0 as $n,m\to\infty$. This approach is inspired by \cite[Section 4.1.2]{ambrosio2005gradient} and illustrated in \cref{fig:cross_partition_error_diagram}, which represents $d_{\tau\eta}(t,t)^2$ as a convex combination of the distances noted by red dotted lines. This will imply that the sequence $\{\bar{\rho}_n\}_{n\in\NN}$ is Cauchy in $\IGW$, which enables lifting its weak pointwise convergence from \cref{prop:limit_of_minimizing_movement} to the desired $\W_2$ convergence.

\begin{proposition}[Cauchy-type IGW bound]\label{prop:discrete_solution_comparison}
    Under Assumptions \ref{assumption:main}(i) and \ref{assumption:main}(iii), suppose that $\tau\vee\eta\leq\frac{\lambda_{\min}(\bsigma_{\rho_0})^2}{2304(\sF(\rho_0)-\sF^\star)}$, and if the convexity parameter $\lambda<0$ (see \cref{assumption:main}(iii)), further let $\tau\vee\eta < \frac{1}{4|\lambda|}$. Then, for any $n,m\in\NN$, we~have
    \begin{align*}
    \sup_{t\in[0,\delta]} \IGW\big(\bar{\rho}_n(t),\bar{\rho}_m(t)\big) \lesssim_{\sF(\rho_0),\sF^\star,\lambda_{\min}(\bsigma_{\rho_0}), \lambda} \sqrt{\tau}+\sqrt{\eta},
    \end{align*}
    which implies that 
    $\{\bar{\rho}_n\}_{n\in\NN}$ is a Cauchy sequence of curves in IGW.
\end{proposition}

\begin{figure}[!t]
 \centering
\includegraphics[width=0.8\textwidth]{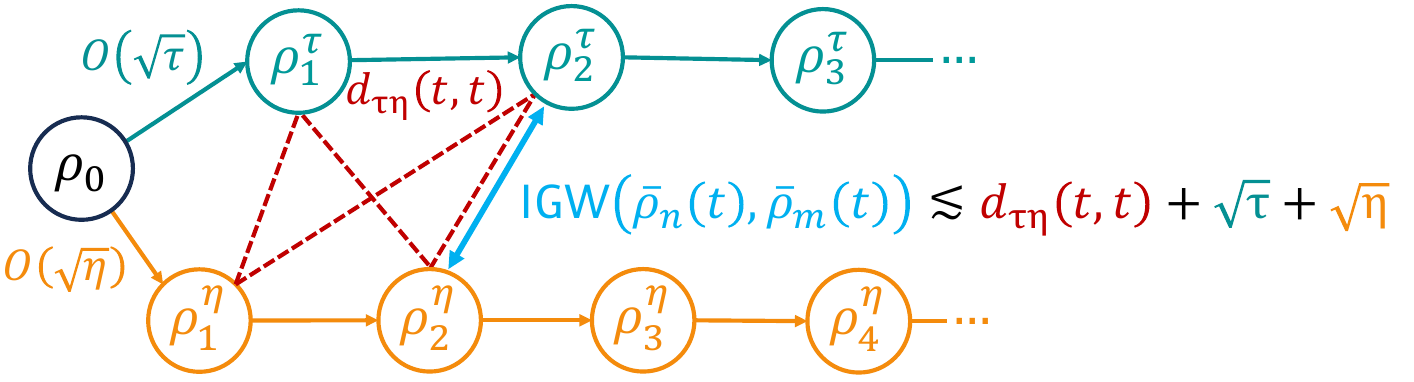} 
\caption{Visualization of cross-partition error function: the IGW distance between the sequences is at most the cross-partition error function, plus an $O(\sqrt{\tau}+\sqrt{\eta})$ term.} 
\label{fig:cross_partition_error_diagram}
\end{figure}

\begin{proof} 
The derivation follows similar lines to that of \cite[Corollary 4.1.7]{ambrosio2005gradient}. To control the uniform IGW gap $ \sup_{t\in[0,\delta]}\IGW\big(\bar{\rho}_n(t),\bar{\rho}_m(t)\big)$, first note that by definition:
\begin{align*}
    d_{\tau\eta}(t,t)^2 =&(1-\ell_\eta(t))(1-\ell_\tau(t))\IGW(\rho^\tau_{i-1},\rho^\eta_{j-1})^2 + (1-\ell_\eta(t))\ell_\tau(t) \IGW(\rho^\tau_i,\rho^\eta_{j-1})^2\\
    &\quad\quad+  \ell_\tau(t)(1-\ell_\tau(t))\IGW(\rho^\tau_{i-1},\rho^\eta_{j})^2 + \ell_\tau(t)\ell_\tau(t) \IGW(\rho^\tau_i,\rho^\eta_{j})^2,
\end{align*}
for $t\in ((i-1)\tau,i\tau]\cap ((j-1)\eta,j\eta]$, whereas $\IGW(\bar{\rho}_n(t),\bar{\rho}_m(t)) = \IGW(\rho^\tau_i,\rho^\eta_{j})$. Consequently, $d_{\tau\eta}(t,t)^2$ is a convex combination of $\IGW^2$ between $\rho^\tau_{i-1},\rho^\tau_i$ and $\rho^\eta_{j-1},\rho^\eta_{j}$; see \cref{fig:cross_partition_error_diagram}. Hence
\begin{align*}
    \big|\IGW\big(\bar{\rho}_n(t),\bar{\rho}_m(t)\big) - d_{\tau\eta}(t,t)\big|&\lesssim_{\sF(\rho_0),\sF^\star,\lambda_{\min}(\bsigma_{\rho_0})}  \IGW(\rho^\tau_{i-1},\rho^\tau_i) + \IGW(\rho^\eta_{j-1},\rho^\eta_{j})\\
    & \lesssim_{\sF(\rho_0),\sF^\star,\lambda_{\min}(\bsigma_{\rho_0})} \sqrt{\tau} +\sqrt{\eta}
\end{align*}
where the second step uses \cref{prop:existence_of_scheme}. Thus, it suffices to establish a uniform upper bound of
\begin{align*}
    d_{\tau\eta}(t,t)\lesssim_{\sF(\rho_0),\sF^\star,\lambda_{\min}(\bsigma_{\rho_0}),\lambda} \sqrt{\tau} +\sqrt{\eta}\numberthis\label{eq:cross_partition_bound}
\end{align*}
to conclude the proof, which is our focus for the rest of the proof. 

\medskip
To control $d_{\tau\eta}$ on the diagonal, we shall employ the variational inequality to bound $\frac{d}{dt}d_{\tau\eta}(t,t)$ and then invoke a Gr\"onwall-type estimate. We define some notation to simplify the subsequent derivation. For $i=1,\ldots,n$ and $t\in((i-1)\tau,i\tau]$, set $\sD_\tau(t) \coloneqq \IGW(\rho^\tau_{i-1},\rho^\tau_i)$ and $\sigma_\tau(t)\coloneqq\lambda-\frac{4\sqrt{2}}{\tau\lambda_{\min}(\bsigma_{\rho_0})}\sD_\tau(t)$. Recalling the weighting function $\ell_\tau$, define the interpolations
\begin{align*}
\sF_\tau(t) &\coloneqq (1-\ell_\tau(t)) \sF(\rho^\tau_{i-1}) + \ell_\tau(t) \sF(\rho^\tau_i)\\
\sR_\tau(t) &\coloneqq \sF_\tau(t) - \sF(\rho^\tau_i),
\end{align*}
where $i=1,\ldots,n$ and $t\in((i-1)\tau,i\tau]$.

\medskip
To proceed, we require the following technical lemma which states a variational inequality on the proximal map from \eqref{eq:minimizing_movement}. See \cref{appen:convexity_of_proximal_functional_proof} for the proof. 

\begin{lemma}[Variational inequality]\label{lem:convexity_of_proximal_functional}
Suppose that $\lambda_{\min}(\bsigma_{\mu_i})\geq c$, $i=0,1$, and $\IGW(\mu_0,\mu_1)\leq\frac{c}{16\sqrt{2}}$, for some $c>0$, and if $\lambda<0$, we further assume $\tau < \frac{1}{4|\lambda|}$. For any  $\mu_1\in\argmin\sF(\mu)+\frac{1}{2\tau}\IGW(\mu,\mu_0)^2$, if there is a generalized geodesic between $\mu_1,\mu_2$ w.r.t. $\mu_0$, then we have
    \begin{align*}
        \sF(\mu_1)+\frac{\IGW(\mu_1,\mu_0)^2}{2\tau} \mspace{-4mu}\leq\mspace{-4mu} & \ \sF(\mu_2) + \frac{\IGW(\mu_2,\mu_0)^2}{2\tau} - \left(\lambda + \frac{1-8\sqrt{2}\IGW(\mu_0,\mu_1)/c}{2\tau}\right) \IGW(\mu_1,\mu_2)^2\\
        & + \frac{6\sqrt{2}}{c\tau} \IGW(\mu_1,\mu_0)^3.
    \end{align*}
\end{lemma}

This result is a consequence of the local convexity of $\IGW$ and convexity of $\sF$ along $\nu_t$ connecting $\mu_1,\mu_2$, instantiated at small $t\in[0,\delta]$. The above lemma essentially yields $\left(\lambda + \frac{1-8\sqrt{2}\IGW(\mu_0,\mu_1)/c}{2\tau}\right)$-'convexity' along the generalized geodesic (compare to \cite[Assumption 4.0.1, Lemma 9.2.7]{ambrosio2005gradient}, where a $\lambda+\tau^{-1}$ convexity was derived). Note that to use this bound, we require that $\IGW(\mu_0,\mu_1)$ is small, which, when applied to $(\mu_0,\mu_1)=(\rho_{i-1},\rho_i)$, will be satisfied through \cref{prop:existence_of_scheme}.

\medskip

The existence of generalized geodesic is guaranteed by choice of $\bar{\delta}$ from \cref{lem:bar_delta}. Under the assumptions on $\tau$ from \cref{prop:discrete_solution_comparison}, applying \cref{lem:convexity_of_proximal_functional} to $(\mu_0,\mu_1,\mu_2)=(\rho_{i-1},\rho_{i},\nu)$, where $\nu\in\cB_{\IGW}(\rho_0,\bar{\delta})$ is arbitrary, yields
\begin{align*}
    \sF(\rho^\tau_i)+\frac{\sD_\tau(t)^2}{2\tau} \leq \sF(\nu) + \frac{1}{2\tau}\IGW(\nu,\rho^\tau_{i-1})^2 - \left(\sigma_\tau(t) + \frac{1}{2\tau}\right) \IGW(\rho^\tau_i,\nu)^2 + \frac{6\sqrt{2}}{\tau \lambda_{\min}(\bsigma_{\rho_0})} \sD_\tau(t)^3.
\end{align*}
Notice that the piecewise linear function $d_\tau(t;\nu)^2$ satisfies 
\[\frac{d}{dt} d_\tau(t;\nu)^2 = \frac{\IGW(\rho^\tau_i,\nu)^2-\IGW(\rho^\tau_{i-1},\nu)^2}{\tau},\quad t\in ((i-1)\tau,i\tau].\]
Combining the above two displays and rearranging, we arrive at
\begin{align*}
    \frac{1}{2}\frac{d}{dt} d_\tau(t;\nu)^2 + \sigma_\tau(t)\IGW(\bar{\rho}_n(t),\nu)^2 &\leq \sF(\nu) - \sF_\tau(t) +  \sR_\tau(t) - \frac{\sD_\tau(t)^2}{2\tau} + \frac{6\sqrt{2}}{\tau\lambda_{\min}(\bsigma_{\rho_0})} \sD_\tau(t)^3\\
    &\leq \sF(\nu) - \sF_\tau(t) +  \sR_\tau(t) - \frac{\sD_\tau(t)^2}{4\tau},\numberthis\label{eq:variational_differential_inequality}
\end{align*}
where we have used that $\tau\leq \frac{1}{2(\sF(\rho_0)-\sF^\star)} \left(\frac{\lambda_{\min}(\bsigma_{\rho_0})}{24\sqrt{2}}\right)^2$ such that $\frac{1}{4}\geq \frac{6\sqrt{2}}{\lambda_{\min}(\bsigma_{\rho_0})} \sD_\tau(t)$, by the condition on $\tau$ and \cref{prop:existence_of_scheme}. Notice that $\big|d_\tau(t;\nu)^2- \IGW\big(\bar{\rho}_n(t),\nu\big)^2\big|\leq 2d_\tau(t;\nu)\sD_\tau(t) + \sD_\tau(t)^2$ and invert into the above to further obtain
\begin{align*}
    \frac{1}{2}\frac{d}{dt} d_\tau(t;\nu)^2 + \sigma_\tau(t) d_\tau(t;\nu)^2 &- 2|\sigma_\tau(t)|d_\tau(t;\nu)\sD_\tau(t)\\& \leq |\sigma_\tau(t)|\sD_\tau(t)^2 + \sF(\nu) - \sF_\tau(t) +  \sR_\tau(t) - \frac{\sD_\tau(t)^2}{4\tau}.
\end{align*}
    
\medskip
Now fix $t$, and consider a linear combination with coefficients $1-\ell_\eta(s),\ell_\eta(s)$ of the above inequality instantiated at $\nu=\rho^\eta_{j-1},\rho^\eta_j$, respectively. By the fact that $(1-\ell_\eta(s))d_\tau(t;\rho^\eta_{j-1}) + \ell_\eta(s)d_\tau(t;\rho^\eta_j)\leq d_{\tau\eta}(t,s)$, we have 
\begin{align*}
    \frac{1}{2}\frac{d}{dt} d_{\tau\eta}(t,s)^2 + \sigma_\tau(t) d_{\tau\eta}(t,s)^2 - 2|\sigma_\tau(t)|&d_{\tau\eta}(t,s)\sD_\tau(t)\\
    &\leq |\sigma_\tau(t)|\sD_\tau(t)^2 + \sF_\eta(s) - \sF_\tau(t) +  \sR_\tau(t) - \frac{\sD_\tau(t)^2}{4\tau};
\end{align*}
switching the roles of $\tau$ and $\eta$, similarly yields
\begin{align*}
    \frac{1}{2}\frac{d}{dt} d_{\eta\tau}(t,s)^2 + \sigma_\eta(t) d_{\eta\tau}(t,s)^2 - 2|\sigma_\eta(t)|&d_{\eta\tau}(t,s)\sD_\eta(t)\\
    &\leq |\sigma_\eta(t)|\sD_\eta(t)^2 + \sF_\tau(s) - \sF_\eta(t) +  \sR_\eta(t) - \frac{\sD_\eta(t)^2}{4\eta}.
\end{align*}
Using the symmetry of the error function, in the sense that $d_{\tau\eta}(t,s)^2 = d_{\eta\tau}(s,t)^2$, we combine the two latter display inequalities to obtain
\begin{align*}
    \frac{1}{2}\frac{d}{dt} d_{\tau\eta}(t,t)^2 + (\sigma_\tau(t)&+\sigma_\eta(t)) d_{\tau\eta}(t,t)^2  - 2\big(|\sigma_\tau(t)|\sD_\tau(t) + |\sigma_\eta(t)|\sD_\eta(t)\big) d_{\tau\eta}(t,t)\\
    &\leq |\sigma_\tau(t)|\sD_\tau(t)^2 +  |\sigma_\eta(t)|\sD_\eta(t)^2 +  \sR_\tau(t)+  \sR_\eta(t) - \frac{\sD_\tau(t)^2}{4\tau} - \frac{\sD_\eta(t)^2}{4\eta}.
\end{align*}

With this at hand, the proof is concluded using the following version of the Gr\"onwall lemma, which we prove in \cref{appen:gronwall_proof}.

\begin{lemma}[Gr\"onwall lemma]\label{lem:gronwall}
    Let $a,b,c:[0,\delta]\to \RR$ be locally integrable functions and $x:[0,\delta]\to\RR$ be continuous, such that
    \begin{align*}
        \frac{d}{dt} x^2(t) + a(t)x^2(t)  \leq c(t)+ b(t)x(t),\quad \forall t\in[0,\delta].
    \end{align*}
    Denoting $A(t)\coloneqq\int_0^t a(s) ds$, for every $T\in[0,\delta]$, we have
    \begin{align*}
        e^{A(T)/2} |x(T)| \leq \left(x^2(0) + \sup_{t\in[0,T]} \int_0^t  e^{A(s)}c(s) ds \right)^{1/2} + 2\int_0^T\big|b(t) e^{A(t)/2}\big| dt. 
    \end{align*}
\end{lemma}

To invoke the lemma, set $
    x(t)\coloneqq d_{\tau\eta}(t,t)$, $a(t)\coloneqq 2\sigma_\tau(t)+2\sigma_\eta(t)$, $b(t)\coloneqq 4\big(|\sigma_\tau(t)|\sD_\tau(t) + |\sigma_\eta(t)|D_\eta(t)\big)$,  and 
    \[
    c(t)\coloneqq 2|\sigma_\tau(t)|\sD_\tau(t)^2 +  2|\sigma_\eta(t)|D_\eta(t)^2 +  2\sR_\tau(t)+  2\sR_\eta(t) - \frac{\sD_\tau(t)^2}{2\tau} - \frac{\sD_\eta(t)^2}{2\eta}.
\]

From \cref{prop:existence_of_scheme} and \eqref{eq:IGW_total_bound}, we have 
\begin{align*}
    &\left|\int_0^t \sD_\tau(s) ds\right| = \tau \sum_{i=1}^n \IGW(\rho^\tau_{i-1},\rho^\tau_i) \leq \tau\sqrt{2\delta(\sF(\rho_0)-\sF^\star)}\\
    &\left|\int_0^t \sD_\tau(s)^2 ds\right| =  \tau \sum_{i=1}^n \IGW(\rho^\tau_{i-1},\rho^\tau_i)^2\leq 2\tau^2 (\sF(\rho_0)-\sF^\star)\\
    &\left|\int_0^t \sD_\tau(s)^3 ds\right| \leq \sqrt{2\tau(\sF(\rho_0)-\sF^\star)}\left|\int_0^t \sD_\tau(s)^2 ds\right|\\
    &\left|\int_0^t \sR_\tau(s) ds\right|\leq \frac{\tau}{2}\sum_{i=1}^n \sF(\rho^\tau_{i-1})-\sF(\rho^\tau_i)\leq \frac{\tau}{2} (\sF(\rho_0)-\sF^\star),
\end{align*}
from which it follows that 
\begin{align*}
    &\left|\int_0^t a(s)ds\right|\lesssim_{\sF(\rho_0),\sF^\star,\lambda_{\min}(\bsigma_{\rho_0}), \lambda}  1\\
    &\left|\int_0^t b(s)ds\right| \lesssim_{\sF(\rho_0),\sF^\star,\lambda_{\min}(\bsigma_{\rho_0}), \lambda} \tau+\eta\\
    & \left|\int_0^t c(s)ds\right| \lesssim_{\sF(\rho_0),\sF^\star,\lambda_{\min}(\bsigma_{\rho_0}), \lambda} \tau+\eta.
\end{align*}
This verifies the local integrability assumption and enable invoking \cref{lem:gronwall}. Recalling that 
$x(0)=d_{\tau\eta}(0,0)=0$, we obtain \[d_{\tau\eta}(t,t)\lesssim_{\sF(\rho_0),\sF^\star,\lambda_{\min}(\bsigma_{\rho_0}), \lambda} \sqrt{\tau}+\sqrt{\eta}, \quad \forall t\in[0,\delta],\]
from which we conclude \eqref{eq:cross_partition_bound}. The desired Cauchy-type IGW bound from \cref{prop:discrete_solution_comparison} follows.\end{proof}

\subsubsection{2-Wasserstein convergence} 
Recall that \cref{prop:limit_of_minimizing_movement} gives $\bar{\rho}_{n_k}\stackrel{w}{\to}\rho$ pointwise (in $t$). We now lift this to convergence in $\W_2$, hence concluding $\bar{\bsigma}_n(t)\to\bsigma_t$. The latter convergence is crucial for transferring the first-order optimality condition from \eqref{eq:first_variation_discrete_vector_field_eq} to the limit as $n\to\infty$.

\begin{proposition}($\W_2$ convergence)\label{prop:strong_convergence}
    Let $\{\bar{\rho}_{n_k}\}_{k\in\NN}$ be the pointwise weakly convergent subsequence from \cref{prop:limit_of_minimizing_movement} whose limit $\rho:[0,\delta]\to\cP_2(\RR^d)$ is uniformly $\W_2$-continuous. Under Assumptions \ref{assumption:main}(i), \ref{assumption:main}(iii) and 
    the conditions of \cref{prop:discrete_solution_comparison}, there exists a further subsequence (not relabeled for simplicity) that converges to $\rho$ uniformly in $\W_2$,~i.e., \[\sup_{t\in[0,\delta]} \W_2\big(\bar{\rho}_{n_k}(t),\rho_t\big)\to0.\]
\end{proposition}
\begin{proof}
    First, by \cref{prop:discrete_solution_comparison} and \cref{lem:weak_topo_lsc}, we have 
    \begin{align*}
        \IGW\big(\bar{\rho}_{n_{k}}(t), \rho_t\big) \leq \liminf_{k'\to\infty} \IGW\big(\bar{\rho}_{n_{k}}(t),\bar{\rho}_{n_{k'}}(t)\big)  \lesssim_{\sF(\rho_0),\sF^\star,\lambda_{\min}(\bsigma_{\rho_0}),\lambda} \sqrt{\frac{\delta}{n_{k}}},
    \end{align*}
    which establishes uniform IGW convergence of the subsequence. A direct decomposition as in \eqref{eq:eigenval_lower_bound_for_convergent_sequence} in the proof of \cref{lem:bar_delta} further yields that if $\bar{\rho}_{n}(t)$ converges to $\rho_t$ in $\IGW$, then since $\lambda_{\min}(\bar{\bsigma}_{n}(t))\geq \lambda_{\min}(\bsigma_{\rho_0})/2$, we conclude $\lambda_{\min}(\bsigma_t)\geq \lambda_{\min}(\bsigma_{\rho_0})/2$. 
    
    \medskip
    Given the above, we next show that IGW convergence together with weak convergence implies convergence in $\W_2$. Fix $t$, and let $\bO_{n}\in\cO_{\rho_t,\bar{\rho}_{n}(t)}$, then by \cref{lem:equivalence_igw_w}
    \[\W_2\big((\bO_{n_k})_\sharp\bar{\rho}_{n_{k}}(t),\rho_t\big)\leq\sqrt{\frac{2}{\lambda_{\min}(\bsigma_{\rho_0})}}\IGW\big(\bar{\rho}_{n_{k}}(t),\rho_t\big),\] 
    whereby $(\bO_{n_k})_\sharp\bar{\rho}_{n_{k}}(t)$ converges to $\rho_t$ in $\W_2$, for any $t\in[0,\delta]$. This further implies convergence of second absolute moments, and by rotation invariance we obtain $M_2(\bar{\rho}_{n_{k}}(t))=M_2\big((\bO_{n_k})_\sharp\bar{\rho}_{n_{k}}(t)\big)\to M_2(\rho_t)$. We now invoke Theorem 6.9 from \cite{villani2008optimal} to conclude that $\lim_{k\to\infty}\W_2(\bar{\rho}_{n_{k}}(t),\rho_t)=0$. Note that by the proof of \cref{prop:limit_of_minimizing_movement},
    \begin{align*}
    \W_2\big(\bar{\rho}_{n_{k}}(t),\bar{\rho}_{n_{k}}(s)\big)&\leq 2\sqrt{\frac{\tau|\lceil s/\tau \rceil-\lceil t/\tau \rceil|(\sF(\rho_0)-\sF^\star)}{\lambda_{\min}(\bsigma_{\rho_0})}}\\
    \W_2\big(\rho_t,\rho_s\big) &\leq 2\sqrt{\frac{|s-t|(\sF(\rho_0) - \sF^\star)}{\lambda_{\min}(\bsigma_{\rho_0})}}.
    \end{align*}
    Consequently, the pointwise $\W_2$-convergence can be lifted to a uniform convergence on a further subsequence by a covering argument for the compact interval. This concludes the proof of \cref{prop:strong_convergence}, but we note that the convergence rate is not implied by our argument.
\end{proof}

\subsection{Convergence of gradient flow equation}

To complete the proof of \cref{thm:main}, it remains to establish the continuous-time gradient flow equation (namely, the last line in the PIDE from Item (2) of \cref{thm:main}, or \cref{cor:grad_flow_eq} below). Recall that the optimal velocity field for the discrete-time solution is characterized by the first-order optimality condition from \cref{prop:discrete_vector_field_eq}; cf. \eqref{eq:first_variation_discrete_vector_field_eq}. Rewriting the conclusion of \cref{prop:discrete_vector_field_eq} in terms of the piecewise constant interpolation, it reads 
\begin{align*}
    -\cL_{2\bar{\bA}_n(t), \bar{\rho}_n(t)}  \big[\bar{v}_n(t,\cdot) \big]\in\partial\sF(\bar{\rho}_n(t)),\quad \forall t\in[0,\delta].
\end{align*}
The goal of this section is to transfer this relation to the continuous-time limit, as $n\to\infty$, thereby deriving the analogous relationship between the limiting $(\rho,v)$ from Propositions \ref{prop:limit_of_minimizing_movement} and \ref{prop:limit_continuity_equation}. The argument hinges on the result of the previous subsection, where we have shown that $\{\bar{\rho}_n\}_{n\in\NN}$ converges in $\W_2$, along a subsequence, to a $\W_2$-continuous curve $\rho$, uniformly in $t$. In addition, we know that the velocity field $v$ defined in \eqref{eq:limit_of_vector_fields} is the weak limit of $\bar{v}_n$. We first show that the transformed velocity field is a strong subdifferential of the objective.

\begin{proposition}[Gradient flow equation]\label{prop:limit_vector_field_eq}
    Under \cref{assumption:main}(i)-(iii),
    we have that
    \begin{align*}
        -\cL_{\bsigma_t, \rho_t} [v_t] \in\partial\sF(\rho_t), \quad t\in[0,\delta] \ \mbox{ a.e.}
    \end{align*}
    is a strong subdifferential.
\end{proposition}

\begin{proof}
    The proof consists of two main parts: we first establish that $-\cL_{2\bar{\bA}_n(t),\bar{\rho}_n(t)} [\bar{v}_n(t,\cdot)]$ has a limit (in some proper sense) that belongs to $\partial\sF(\rho_t)$. Afterwards, we show that this limit coincides with $-\cL_{\bsigma_t, \rho_t}  [v_t]$, which concludes the proof. For convenience of presentation, we denote the strong subdifferential from \cref{prop:discrete_vector_field_eq} by $w_i\coloneqq -\cL_{2\bA^\star_i,\rho_i}[v_i]$, for $i=1,\ldots,n$, and define corresponding piecewise constant interpolation as $\bar{w}_n(t)\coloneqq w_i$, $t\in\big((i-1)\tau,i\tau\big]$. Analogously to the construction of $\nu_n$ from $\bar{v}_n$, we further define $\omega_n(t,x,y) \coloneqq \upsilon_\delta \left[(\id,\bar{w}_n)_\sharp\bar{\rho}_n\right]$. 

    \medskip
    
    To find a limit for $\bar{w}_n(t)$, first note that for any $\mu\in\cP_2(\RR^d)$, $\bA\in \RR^{d\times d}$ symmetric and PSD, and $v\in L^2(\mu;\RR^d)$, we have $\|\cL_{\bA,\mu}[v]\|_{L^2(\mu;\RR^d)}^2\leq 8(\|\bA\|_\op^2+M_2(\mu)^2)  \|v\|_{L^2(\mu;\RR^d)}^2$. In particular,
    \begin{align*}
        \|w_i\|_{L^2(\rho_i;\RR^d)}^2 \leq 8\left(\|2\bA^\star_i\|_\op^2+M_2(\rho_i)^2\right)\|v_i\|_{L^2(\rho_i;\RR^d)}^2.\numberthis\label{eq:operator_norm_of_cL}
    \end{align*}
    The bounds from \eqref{eq:moment_bound:prop:existence_of_scheme} and \eqref{eq:cov_var_difference:prop:existence_of_scheme} in \cref{prop:existence_of_scheme} allow controlling $M_2(\rho_i)$ and $\|2\bA^\star_i\|_\F$, with the latter further bounding $\|2\bA^\star_i\|_\op$. Inserting these bound into the above display and utilizing \eqref{eq:vector_field_L2_bound}, we conclude that 
    \begin{equation}
        \sup_{n\in\NN}\int_0^\delta \int \|\bar{w}_n\|^2 d\bar{\rho}_n(t) d t \lesssim_{\sF(\rho_0),\sF^\star,M_2(\rho_0),\lambda_{\min}(\bsigma_{\rho_0})} \sup_{n}\int_0^\delta \int \|\bar{v}_n\|^2 d\bar{\rho}_n(t) d t <\infty.\label{eq:integrability}
    \end{equation}
This implies that $\{\omega_n\}_{n\in\NN}\subset\cP_2([0,\delta]\times\RR^d\times\RR^d)$ is tight, and therefore, there exists a further subsequence of $n_k$ (not relabeled for simplicity), such that $\omega_{n_k}\stackrel{w}{\to}\omega\in\cP_2([0,\delta]\times\RR^d\times\RR^d)$.

Consider the disintegration of $\omega$ w.r.t. the marginal $\upsilon_\delta\rho\in\cP_2([0,\delta]\times\RR^d)$ and denote it by $\omega_{t,x}$. Define the conditional expectation $w: (t,x)\mapsto \int y  d  \omega_{t,x}(y), [0,\delta]\times\RR^d\to\RR^d$. To see that $w_t\in\partial\sF(\rho_t)$ a.e. on $t\in[0,\delta]$, we shall employ Theorem 11.1.6. from \cite{ambrosio2005gradient}. The fact that $w_i\in\partial\sF(\rho_{i})$, for each $i=1,\ldots,n$, is a strong subdifferential, along with \cref{assumption:main} and the integrability property established in \eqref{eq:integrability}, we satisfy the conditions of \cite[Theorem 11.1.6 Step 5]{ambrosio2005gradient} and conclude that $w_t\in\partial\sF(\rho_t)$, for $t\in[0,\delta]$ a.e.

\medskip
We now move to the second part of the proof and work to show that $ -\cL_{\bsigma_{\rho_t}, \rho_t}  [v_t]= w_t$ a.e. on $[0,\delta]$. First note that for any $g\in C_c^\infty((0,\delta)\times\RR^d;\RR^d)$,
\begin{align*}        \lim_{k\to\infty}\frac{1}{\delta}\int_0^\delta \int \llangle g(t,x),\bar{w}_{n_k}(t,x)\rrangle d\bar{\rho}_{n_k}(t,x) dt = \frac{1}{\delta}\int_0^\delta \int\llangle g(t,x),w_t(x)\rrangle d\rho_t(x) dt.
    \end{align*}
    On the other hand, $\bar{w}_{n}(t,x) = -2\big(2\bar{\bA}_n(t) \bar{v}_n(t,x) + \big[\int y\bar{v}_n(t,y)^\intercal d\bar{\rho}_n(t,y)\big] x\big)$. To conclude the proof, we rely on the following technical lemma, proven in \cref{appen:lem:velocity_distributional_limits_equations_proof}. 
    
    \begin{lemma}[Distributional limits]\label{lem:velocity_distributional_limits_equations} For any $g\in C_c^\infty((0,\delta)\times\RR^d;\RR^d)$, the following limits hold:
    \begin{align}
        &\lim_{k\to\infty} \int_0^\delta \int\llangle g(t,x),2\bar{\bA}_{n_k}(t) \bar{v}_{n_k}(t,x) \rrangle d\bar{\rho}_{n_k}(t,x) dt = \int_0^\delta \int\llangle g(t,x),  \bsigma_t  v_t(x) \rrangle d\rho_t(x)d t,\label{eq:technical_eq_1}\\    
        &\lim_{k\to\infty} \int_0^\delta \int \llangle g(t,x),\int y\bar{v}_{n_k}(t,y)^\intercal d\bar{\rho}_{n_k}(t,y) x \rrangle d\bar{\rho}_{n_k}(t,x) dt \nonumber\\
        &\quad\quad\quad\quad\quad\quad\quad\quad\quad\quad\quad\quad\quad\quad\quad = \int_0^\delta \int\llangle g(t,x),\int yv_t(y)^\intercal d\rho_t(y) x \rrangle d\rho_t(x) dt. \numberthis\label{eq:technical_eq_2}
    \end{align}
    
    \end{lemma}

    Plugging \eqref{eq:technical_eq_1}-\eqref{eq:technical_eq_2} back into the preceding display equation, we conclude that 
    \begin{align*}
        \lim_{k\to\infty}\frac{1}{\delta}\int_0^\delta \int&\llangle g(t,x),\bar{w}_{n_k}(t,x)\rrangle  d\bar{\rho}_{n_k}(t,x) dt\\
        &\qquad\qquad= \lim_{k\to\infty}\frac{1}{\delta}\int_0^\delta \int\llangle g(t,x),-\cL_{2\bar{\bA}_n(t), \bar{\rho}_n(t)}\big[\bar{v}_n(t,\cdot)\big](x)\rrangle d\bar{\rho}_{n_k}(t,x) dt\\
        &\qquad\qquad= \frac{1}{\delta}\int_0^\delta \int\llangle g(t,x), -\cL_{\bsigma_t, \rho_t}  [v_t](x)\rrangle d\rho_t(x) dt,
    \end{align*}
    which means that 
    \begin{align*}
        \int_0^\delta \int\llangle g(t,x),w_t(x)\rrangle d\rho_t dt = \int_0^\delta \int\llangle g(t,x), -\cL_{\bsigma_t, \rho_t} [v_t](x)\rrangle d\rho_t(x) dt.
    \end{align*}
    Since $g\in C_c^\infty((0,\delta)\times\RR^d;\RR^d)$ is arbitrary, we have that $-\cL_{\bsigma_t, \rho_t}  [v_t] = w_t$ $\rho_t$-a.s. for a.e.~$t$, concluding the proof. 
\end{proof}
\cref{prop:limit_vector_field_eq} immediately gives rise to the desired continuous-time gradient flow equation, and concludes the proof of \cref{thm:main}. Moreover, by symmetry of $\bar{\bA}_{n_k}$, we have symmetry of $\int x\bar{v}_{n_k}(t,x)^\intercal d\bar{\rho}_{n_k}(t,x)$, and similar to \eqref{eq:technical_eq_2} it is straightforward to show 
\begin{align*}
    \lim_{k\to\infty} \int_0^\delta \tr\left(g(t) \int x\bar{v}_{n_k}(t,x)^\intercal d\bar{\rho}_{n_k}(t,x)\right) dt = \int_0^\delta \tr\left(g(t) \int xv_t(x)^\intercal d\rho_t(x)\right) dt,
\end{align*}
thus by choosing $g\in C_c^\infty((0,\delta);\RR^{d\times d})$ with $g^\intercal=-g$, we conclude that $v_t\in\cI_t$, and by \cref{rem:property_of_operator} we further have $\cL_{\bsigma_t,\rho_t}[v_t]\in\cI_t$, which concludes the last point in \cref{thm:main}.

\begin{corollary}[Gradient flow equation]\label{cor:grad_flow_eq}
    Under Assumptions \ref{assumption:main}(i)-(iii) and \ref{assumption:differentiability}, we have
    \begin{align*}
        \nabla\delta\sF(\rho_t) = -\cL_{\bsigma_t, \rho_t}  [v_t], \quad \rho_t\mbox{-a.s. }t\in[0,\delta]\mbox{ a.e.}
    \end{align*}
\end{corollary}

\section{Riemannian Structure and Dynamical Formulation}\label{sec:riemann}

The celebrated Otto Calculus \cite{otto2001geometry} along with the Benamou-Brenier formula \cite{benamou2000computational} have long been cornerstones for the study of Wasserstein geometry. Otto showed that, formally, one can define a Riemannian structure on $\cP_2(\RR^d)$, such that the tangent space $T_\mu\cP_2(\RR^d)$ at $\mu$ is isomorphic to $L^2(\mu;\RR^d)$, with the inner product structure therein inducing the Riemannian metric tensor. The Benamou-Brenier formula \cite{benamou2000computational} further says that the notion of distance induced by the Riemannian structure is precisely $\sW_2$, which justifies identifying $\cP_2(\RR^d)$ equipped with this Riemannian structure with the 2-Wasserstein space. With this formalism, Otto showed that the Wasserstein gradient of a functional $\sF:\cP_2(\RR^d)\to\RR\cup\{\infty\}$ at $\mu$ is $\nabla\delta\sF(\mu)$, which unlocked gradient flows in Wasserstein spaces, although a rigorous derivation of these ideas came only later in \cite{ambrosio2005gradient}. 

Inspired by the route paved by Otto, Benamou, and Brenier, we next identify the Riemannian structure on $\cP_2(\RR^d)$ that induces the intrinsic geometry of IGW. We start from defining the intrinsic IGW metric and the induced geodesics, then we set up the Riemannian structure that induces this intrinsic metric, giving rise to a Benamou-Brenier-like formula for IGW. Lastly, we trace back to gradient flows and define the IGW gradient, complementing the heuristic argument in \cref{sec:IGW_gradient_flow}.

\subsection{IGW curve length, intrinsic metric, and geodesic}

Geodesics between mm spaces under the GW distance were studied in \cite{sturm2012space,sturm2023space} under the framework of gauged measure spaces (see \cite{chowdhury2020gromov} for an implementation). The $(p,q)$-GW geodesic between two mm spaces $(\cX,\mathsf d_\cX,\mu)$ and $(\cY,\mathsf d_\cY,\nu)$ was identified as $\big(\cX\times\cY,(1-t)\mathsf{d}_\cX+t\mathsf{d}_\cY,\pi^\star\big)$, where $\pi^\star\in\Pi(\mu,\nu)$ is an optimal alignment plan for $\GW_{p,q}(\mu,\nu)$. It follows that intermediate points along the geodesic between points in $\cP_2(\RR^d)$ (identified with their natural mm spaces) are usually no longer Euclidean mm spaces themselves. The same argument applies to IGW. As a simple example, consider two uniform distributions on $n$ points in $\RR^d$ with  $n>d$. The geodesic corresponds to a linear interpolation of the inner product matrices of the two measures, which could have rank higher than $d$, i.e., cannot be realized as an inner product matrix of $n$ points in $\RR^d$.

Although the IGW metric does not give rise to a length space due to the nonexistence of geodesics in $\cP_2(\RR^d)$, we may still define the intrinsic metric based on curve length under IGW. In this section, we consider IGW absolutely continuous curves (also known as IGW-rectifiable), introduce the intrinsic metric, and study geodesics in this context; see \cite[Chapter 4]{ambrosio2004topics} for a detailed account of the intrinsic formulation of geodesics. We start from basic definitions (cf. \cite[Theorem 4.1.6]{ambrosio2004topics}).
\begin{definition}[IGW curve length and metric derivative]
    The IGW length of any Lipschitz curve $\rho\in\lip_{\IGW}\big([0,1];\cP_2(\RR^d)\big)$ is defined as 
    \[
        \ell_\IGW(\rho) \coloneqq \sup_{P=\{0=t_0<t_1<\ldots<t_n=1\}} \sum_{i=1}^n \IGW\big(\rho_{t_{i-1}},\rho_{t_{i}}\big),
    \]
    where the $\sup$ is taken over all partitions of $[0,1]$, i.e., $P=\{0=t_0<t_1<\ldots<t_n=1\}$ of any size $n\in\NN$. For each $t\in(0,1)$, define the corresponding metric derivative, whenever exists, by
    \[
        |\rho'|(t)\coloneqq \lim_{h\to0}\frac{\IGW\big(\rho_{t+h},\rho_{t}\big)}{h}.
    \]
\end{definition}
The class $\lip_{\IGW}\big([0,1];\cP_2(\RR^d)\big)$ is rich enough, as any absolutely continuous curve with finite variation can be reparametrized into the Lipschitz class; see \cite[Lemma 1.1.4]{ambrosio2005gradient}. In particular, \cref{lem:equivalence_igw_w} implies that any $\W_2$-Lipschitz curves are also (locally) $\IGW$-Lipschitz. We thus focus on the curves in $\lip_{\IGW}([0,1];\cP_2(\RR^d))$, each of which clearly has a finite length. Thanks to Theorems 4.1.6 and 4.2.1 from \cite{ambrosio2004topics} the metric derivative exists a.e. and we may assume, without loss of generality, that Lipschitz curves are of constant speed, as restated in the following lemma.
\begin{lemma}[Constant speed reparametrization and metric derivative \cite{ambrosio2004topics}]\label{lem:metric_derivative}
    For any curve $\rho\in\lip_{\IGW}\big([0,1];\cP_2(\RR^d)\big)$, the metric derivative $|\rho'|$ exists a.e. Moreover, we may reparametrize $\rho$ into $\tilde{\rho}\in\lip_{\IGW}\big([0,1];\cP_2(\RR^d)\big)$, such that $\tilde{\rho}([0,1])=\rho([0,1])$, and $|\tilde{\rho}'|(t)=\ell_\IGW(\rho)$ a.e., and $\ell_\IGW(\tilde{\rho}) = \int_0^1 |\tilde{\rho}'|(t) dt =\ell_\IGW(\rho)$.
\end{lemma}

Given those basic facts, we next define the induced intrinsic metric.
\begin{definition}[Intrinsic IGW metric]
    The intrinsic IGW metric is defined as 
    \begin{align*}
    \mathsf{d}_{\IGW}(\mu_0,\mu_1)\coloneqq \inf_{\substack{\rho\in\lip_{\IGW}([0,1];\cP_2(\RR^d))\\ \rho_0=\mu_0,\rho_1=\mu_1}}\ell_\IGW(\rho),
    \end{align*}
    where the $\inf$ is taken over all possible Lipschitz curves joining $\mu_0$ with $\mu_1$.
\end{definition}

The following theorem states that $\big(\cP_2(\RR^d),\mathsf{d}_{\IGW})$ is a pseudometric and, in fact, a geodesic space, i.e., $\mathsf{d}_{\IGW}$ is achieved by the length of a connecting curve $\rho$; see \cref{appen:thm:intrinsic_metric_is_geodesic} for the proof.

\begin{theorem}[Pseudometric, geodesics, and continuity]\label{thm:intrinsic_metric_is_geodesic}
    Let $\mu_0,\mu_1\in\cP_2(\RR^d)$ be arbitrary. The following statements hold:
        \begin{enumerate}[leftmargin=*]
        \item $\big(\cP_2(\RR^d),\mathsf{d}_{\IGW}\big)$ is a pseudometric space such that $\mathsf{d}_{\IGW}(\mu_0,\mu_1)=0$ if and only if $\IGW(\mu_0,\mu_1)=0$.
        \item There exists a $\rho\in\lip_{\IGW}\big([0,1];\cP_2(\RR^d)\big)$ connecting $\mu_0,\mu_1$ such that $\ell_\IGW(\rho)=\mathsf{d}_{\IGW}(\mu_0,\mu_1)$, i.e., $\big(\cP_2(\RR^d),\mathsf{d}_{\IGW}\big)$ is a geodesic space.
        \item $\mathsf{d}_{\IGW}$ is continuous in $\IGW$ around distributions with a nonsingular covariance matrix, i.e., if $\lambda_{\min}(\bsigma_{\mu_0}) >0$, then
    \begin{align*}
        \lim_{\IGW(\mu,\mu_0)\to0} \mathsf{d}_{\IGW}(\mu,\mu_0)=0.
    \end{align*}
    \end{enumerate}
\end{theorem}

\begin{example}[Length of IGW gradient flow]\label{ex:grad_flow_length}
    Recall from \cref{thm:main} that the IGW gradient flow curve for an objective $\sF:\cP_2(\RR^d)\to\RR\cup\{\infty\}$ is described by a velocity field $v_t$ that satisfies~the~continuity equation and is related to the local variational behavior of $\sF$ through $v_t=-\cL_{\bsigma_t,\rho_t}^{-1}\left[\nabla\delta \sF(\rho_t)\right]$. We now  observe that the stepwise movement is $\IGW$-rectifiable and provide an estimate for its length. 
    
    For $i=1,\ldots, n$, define the covariance matrix $\bK_i\coloneqq \int v_i(x)v_i(x)^\intercal d\rho_i(x)\in\RR^{d\times d}$, and let $\bar{\bK}_n$ be its piecewise constant interpolation. Also recall from \cref{prop:discrete_vector_field_eq} the definition of the cross-covariance matrix $\bL_{i}\coloneqq\int x v_i(x)^\intercal d\rho_{i}(x)\in\RR^{d\times d}$. We start by expressing the IGW-gap between any two consecutive steps in terms of these matrices
\begin{align*}
    \IGW(\rho_{i},\rho_{i-1})^2&= \int \left|\llangle x,x'\rrangle-\llangle x-\tau v_i(x), x'-\tau v_i(x')\rrangle\right|^2 d \rho_{i}(x)d\rho_{i}(x')\\
    &= \int \left|\tau\llangle v_i(x),x'\rrangle + \tau\llangle x,v_i(x')\rrangle -\tau^2\llangle  v_i(x),  v_i(x')\rrangle\right|^2 d \rho_{i}(x)d\rho_{i}(x')\\
    & = \int \Big(2\tau^2 \llangle v_i(x),x'\rrangle^2 + 2\tau^2\llangle v_i(x),x'\rrangle \llangle x,v_i(x')\rrangle\\
    &\qquad\qquad\quad-4\tau^3\llangle v_i(x),x'\rrangle\llangle  v_i(x),  v_i(x')\rrangle+ \tau^4\llangle  v_i(x),  v_i(x')\rrangle^2\Big) d \rho_{i}\otimes \rho_{i}(x,x')\\
    & = 2\tau^2 \tr(\bK_i\bsigma_{\rho_i}) + 2\tau^2 \tr(\bL_i^2) - 4\tau^3 \tr(\bL_i\bK_i) + \tau^4\tr(\bK_i^2).
\end{align*}
From \eqref{eq:vector_field_L2_bound}, we have $\sum_{i=1}^n \tau\|v_i\|_{L^2(\rho_i;\RR^d)}^2 \leq \frac{2(\sF(\rho_0)- \sF^\star)}{\lambda_{\min}(\bsigma_{\rho_0})}$, from which it follows that $\|v_i\|_{L^2(\rho_i;\RR^d)}^2=O(1/\tau)$, for all $i=1,\ldots,n$. This further yields the estimates
\begin{align*}
    \sum_{i=1}^n \tau^3 \tr(\bL_i\bK_i)&\leq \sum_{i=1}^n \tau^3 \sqrt{M_2(\rho_i)}\|v_i\|_{L^2(\rho_i;\RR^d)}^3  = O(\tau^{3/2})\\
    \sum_{i=1}^n \tau^4\|\bK_i\|_\F^2&\leq \sum_{i=1}^n \tau^4\|v_i\|_{L^2(\rho_i;\RR^d)}^4 = O(\tau^2),
\end{align*}
using which we conclude that for small $\tau>0$ values
\begin{align*}
    \sum_{i=1}^n \IGW(\rho_{i},\rho_{i-1}) &\approx \sum_{i=1}^n\tau \sqrt{2\tr(\bK_i\bsigma_{\rho_i}) + 2\tr(\bL_i^2)} \\
    & = \int_0^\delta \sqrt{2\tr\big(\bar{\bK}_n(t)\bar{\bsigma}_n(t)\big) + 2\tr\big(\bar{\bL}_n(t)^2\big)} dt\\
    & = \int_0^\delta \sqrt{\llangle \bar{v}_n(t),  \cL_{\bar{\bsigma}_n(t),{\bar{\rho}_n(t)}}[\bar{v}_n(t)]\rrangle_{L^2(\bar{\rho}_n(t);\RR^d)} }dt,
\end{align*}
where the operator $\cL$ on the right-hand side (RHS) is given in \eqref{eq:operator}.

With a slight abuse of the notation $\ell_\IGW$, we may compute the length of the piecewise constant curve via  $\ell_{\IGW}(\bar{\rho}_n) = \sum_{i=1}^n \IGW(\rho_{i},\rho_{i-1})$. Indeed, so long that the partition $P=\{0=t_0<\ldots<t_k=1\}$ has a point inside each of the intervals $((i-1)\tau,i\tau]$, $i=1,\ldots, n$, we have that $\sum_{i=1}^k \IGW\big(\bar\rho_n(t_i),\bar\rho_n(t_{i-1})\big) = \sum_{i=1}^n \IGW(\rho_{i},\rho_{i-1})$; otherwise, the value is smaller by the triangle inequality. We conclude that
\begin{align*}
    \ell_{\IGW}(\bar{\rho}_n) \approx \int_0^\delta \sqrt{\llangle \bar{v}_n(t),  \cL_{\bar{\bsigma}_n(t),{\bar{\rho}_n(t)}}[\bar{v}_n(t)]\rrangle_{L^2(\bar{\rho}_n(t);\RR^d)} }dt.
\end{align*}
We expect this approximating to become an equality in the limit as $n\to \infty$, with $(\rho_t,v_t)$ replacing their piecewise counterparts above. A formal derivation requires resolving some technical details, which we leave for future work.

\end{example}

\subsection{Riemannian structure}

\cref{ex:grad_flow_length} shows that when the gradient flow step size $\tau$ is small, we retrieve a local first-order approximation of IGW as a modified inner product in an appropriate $L^2$ space. This hints at a Riemannian structure arising from this local behavior, which we identify next.

\subsubsection{Metric tensor and IGW}

Consider the Riemannian structure on $\cP_2(\RR^d)$ defined for any $v,w\in L^2(\mu;\RR^d)$ by
\begin{subequations}
\begin{align}
    g_\mu(v,w)&\coloneqq \int \Big(\llangle v(x), x'\rrangle + \llangle x, v(x')\rrangle\Big)\Big( \llangle w(x), x'\rrangle + \llangle x, w(x')\rrangle\Big) d \mu\otimes \mu(x,x')\label{eq:metric_tensor_integral_form}\\
    & = 2\int v(x)^\intercal \bsigma_\mu w(x) d\mu(x) + 2\tr\left(\int xv(x)^\intercal d\mu(x)\int xw(x)^\intercal d\mu(x)\right)\label{eq:metric_tensor_trace_form}\\
    & = \llangle v,  \cL_{\bsigma_\mu,\mu}[w]\rrangle_{L^2(\mu;\RR^d)},\label{eq:metric_tensor}
\end{align}\label{eq:metric_tensor_collect}
\end{subequations}
where $\cL_{\bsigma_\mu,\mu}$ is given in \eqref{eq:operator}. By \cref{rem:property_of_operator}, $g_\mu$ is a positive semi-definite bilinear form. Without ambiguity we will also refer to $g_\mu(v,v)$ for $v\in L^2(\mu;\RR^d)$ as the (instantaneous) kinetic energy derived from IGW. We start from a simple observation through differentiation.
Following \cite[Chapter 7]{villani2008optimal} we will call the integration $\int g_{\rho_t}(v_t,v_t) dt$ over proper interval the \emph{action} of the IGW kinetic energy, though the reader should note that our action does not correspond to structure of the cost function in Wasserstein case, due to the global nature of $g$. 
\begin{lemma}[Action and IGW]\label{lem:IGW_flow_upperbound}
    For any $\mu_0,\mu_1\in\cP_2(\RR^d)$, we have
    \begin{align*}
        \IGW(\mu_0,\mu_1)^2 \leq \inf_{\substack{(\rho_t,v_t):\\\partial_t \rho_t + \nabla\cdot  (\rho_t v_t)=0\\
        \rho_0=\mu_0, \rho_1=\mu_1}} \int_0^1 g_{\rho_t}(v_t,v_t) dt,
    \end{align*}
    where the infimum is over all $(\rho_t,v_t)_{t\in[0,1]}$, such that $\rho$ is a weakly continuous curve that joins $\mu_0,\mu_1$, the velocity field $v_t$ has $\int_0^1\|v_t\|^2_{L^2(\rho_t;\RR^d)}dt<\infty$, and the pair satisfies the continuity equation. 
    
\end{lemma}
The full derivation requires various regularity arguments and is deferred to \cref{appen:lem:IGW_flow_upperbound_proof}. Nevertheless, assuming sufficient regularity, the inequality is quite straightforward. Consider the flow map $X_t$ associated with the velocity field $v_t \in L^2(\rho_t; \RR^d)$, taking an initial position $x_0\in\RR^d$ and mapping it to a new position $x(t)=X_t(x_0)\in\RR^d$. This map is determined by solving the ordinary differential equation (ODE):
\[ \frac{dx(t)}{dt} = v_t\big(x(t)\big). \]
We have
\begin{align*}
    \IGW(\mu_0,\mu_1)^2 &\leq \int \left|\llangle x,x'\rrangle-\llangle X_1(x), X_1(x')\rrangle\right|^2 d \mu_0\otimes\mu_0(x,x')\\
    &= \int \left|\int_0^1 \frac{d}{dt}\llangle X_t(x), X_t(x')\rrangle dt \right|^2 d \mu_0\otimes\mu_0(x,x')\\
    & = \int \left| \int_0^1 \llangle v_t\big(X_t(x)\big), X_t(x')\rrangle + \llangle X_t(x), v_t\big(X_t(x')\big)\rrangle dt \right|^2 d \mu_0\otimes\mu_0(x,x')\\
    & \leq  \int \int_0^1 \left|\llangle v_t\big(X_t(x)\big), X_t(x')\rrangle + \llangle X_t(x), v_t\big(X_t(x')\big)\rrangle\right|^2 dt \, d \mu_0\otimes\mu_0(x,x')\\
    & \leq  \int \int_0^1 \left|\llangle v_t(y), y'\rrangle + \llangle y, v_t(y')\rrangle\right|^2 dt\,  d \rho_t\otimes \rho_t(y,y')\\
    & = \int_0^1 g_{\rho_t}(v_t,v_t) dt,
\end{align*}
where we have used Jensen's inequality, and notice that we are using the first one of the equivalent definitions of $g_\mu$ \eqref{eq:metric_tensor_integral_form}. While Jensen's inequality is not tight (indeed, we generally do not expect equality since IGW geodesics need not be realizable in Euclidean space), this result establishes a simple one-sided connection between the IGW distance and the metric tensor. We next show that with proper modification, equality is achieved for the intrinsic IGW metric, resulting in a formula akin to the celebrated dynamical formulation for the Wasserstein distance by Benamou and Brenier \cite{benamou2000computational}.

\subsubsection{Benamou-Brenier-like formula for IGW}

The Benamou-Brenier formula \cite{benamou2000computational} identifies the 2-Wasserstein distance with the smallest kinetic energy among all connecting velocity fields for mass transportation:
\[
    \W_2(\mu_0,\mu_1)^2 = \inf_{\substack{(\rho_t,v_t):\\\partial_t \rho_t + \nabla\cdot  (\rho_t v_t)=0\\
        \rho_0=\mu_0, \rho_1=\mu_1}}  \int_0^1 \|v_t\|^2_{L^2(\rho_t;\RR^d)} dt,
\]
where the infimum is over the same domain as in \cref{lem:IGW_flow_upperbound}. 
As the Riemannian structure for the Wasserstein space is given by $\langle v,w\rangle_{L^2(\mu;\RR^d)}$, the above shows that $\W_2$ (which also coincides with the intrinsic Wasserstein metric) exactly captures the induced notion of distance. 

\medskip
By the same token, the next theorem shows that the intrinsic IGW metric $\mathsf{d}_{\IGW}$ is the distance induced by the Riemannian structure $g_\mu(v,w)$ from \eqref{eq:metric_tensor_collect}. This yields a Benamou-Brenier-like formula~for~IGW.

\begin{theorem}[IGW Benamou-Brenier formula]\label{thm:BB_IGW}
    Let $\mu_0,\mu_1\in \cP_2(\RR^d)$ be such that there exists a minimizing curve $(\rho_t)_{t\in[0,1]}\subset\cP_2(\RR^d)$ for $\mathsf{d}_{\IGW}(\mu_0,\mu_1)$ with $\inf_t\lambda_{\min}(\bsigma_{\rho_t})>0$. Then
    \begin{equation}
        \mathsf{d}_{\IGW}(\mu_0,\mu_1)^2 = \min_{\mu\in\{\mu_1,\bI^-_\sharp\mu_1\}} \inf_{\substack{(\rho_t,v_t):\\\partial_t \rho_t + \nabla\cdot  (\rho_t v_t)=0\\
        \rho_0=\mu_0, \rho_1=\mu}} \int_0^1  g_{\rho_t}(v_t,v_t) dt,\label{eq:IGW_BB}
    \end{equation}
    where $\bI^-$ is any fixed reflection matrix, and the inner infimum is over the same domain as in \cref{lem:IGW_flow_upperbound}. 
\end{theorem}

\begin{remark}[Fused IGW]
The lower bound on the eigenvalues of $\bsigma_{\rho_t}$ is crucial for our construction of the minimizing flow, as it guarantees compactness to ensure existence of velocity $v$ for the minimizing flow, as was done in \eqref{eq:vector_field_L2_bound}. While removing this condition seems hard in general, one can circumvent it by considering the fused IGW distance \cite{vayer2020fused}:
    \begin{align*}
    \IGW_\lambda(\mu,\nu)^2\coloneqq \inf_{\pi\in\Pi(\mu,\nu)}\int \left(\left|\llangle x,x'\rrangle-\llangle y,y'\rrangle\right|^2 + \lambda\|x-y\|^2\right)d\pi\otimes\pi(x,y,x',y'),
    \end{align*}
    where $\lambda>0$ and $\mu,\nu$ are assumed to be supported in the same Euclidean space. While fused IGW is no longer invariant to orthogonal transformations, it still admits the existence of optimal alignment-transport maps via similar arguments to those from \cref{lem:F2Dual} and \cref{lem:gromov_monge}
    when the optimal dual variable $\bA^\star$ satisfies that $4\bA^\star+\lambda\bI$ is nonsingular. This enables repeating the steps from the proof of \cref{prop:discrete_vector_field_eq} to show that the strong subdifferential arising from construction of the GMM steps with fused IGW is \[\tau^{-1}\left(4\bA^\star_{i+1}T^\star_{i+1} - 2\bsigma_{\rho_{i+1}} x + \lambda (T^\star_{i+1}(x)-x)\right) \in\partial\sF(\rho_{i+1}).\]
    The associated operator is then
    $\cL_{\bA,\mu}^\lambda \coloneqq \cL_{\bA,\mu} + \lambda \id = \cL_{\bA+\lambda\bI/2,\mu}$, which alleviates the invertibility issue of $\bA$ as is required for $\cL^{-1}$ in \cref{rem:property_of_operator}. Furthermore, by repeating the steps in \cref{sec:riemann}, we may define a new Riemannian metric tensor
    \[
        g_\mu^\lambda(v,w)\coloneqq g_\mu(v,w) + \lambda\llangle v,w \rrangle_{L^2(\mu;\RR^d)}.
    \]
    A direct computation verifies that a Benamou-Brenier-type formula holds for the corresponding intrinsic metric and that a minimizing flow exists (follows by existence of an infimizing sequence with bounded $\W_2$ action). Overall, this setting can be viewed as interpolating between $\W_2$ and $\IGW$ geometries. Similar ideas of fusing different metrics and studying its dynamical form and gradient flows were introduced, for instance, for the Wasserstein-Fisher-Rao metric \cite{kondratyev2016new,chizat2018interpolating,liero2018optimal}.
\end{remark}

\begin{remark}[Necessity of reflection]
    
The reflection of $\mu_1$ cannot be dropped in general, as there might not be an IGW-minimizing curve from $\mu_0$ to $\mu_1$ that solves the continuity equation. This is consistent with the invariance of IGW under transformations from $\Od$, which has two connected components corresponding to matrices with determinant $+1$ or $-1$. For instance, for an asymmetric measure $\mu\in\cP_2^{\mathrm{ac}}(\RR^d)$, we have $\mathsf{d}_{\IGW}(\mu, \bI^-_\sharp \mu) = 0$, but there is no flow joining them with IGW length $\ell_{\IGW}=0$. 
    We resolve this issue by considering curves $\rho_t$ from $\mu_0$ to the closer one of $\mu_1$ or $\bI^- \mu_1$ in the IGW sense. An alternative correction for this issue is to consider curves over $\RR^{d+1}$, as $\mu$ and $\bI^-_\sharp\mu$ can be connected by a curve of transformations in $\mathrm{SO}(d+1)$, utilizing the one additional dimension. Also note that depending on the symmetry of $\mu_0,\mu_1$, the minimum could be attained by both $\mu_1$ and $\bI^-_\sharp\mu_1$ (e.g., if $\mu_0=\bI^-_\sharp\mu_0$).

We stress that the restriction to curves that satisfy the continuity equation is inherit to our theory, which, from the outset, instantiated IGW gradient flows in Wasserstein space (see \cref{fig:IGW_steps} and discussion after Eq. \eqref{eq:minimizing_movement}). While IGW flows in the quotient space of $\cP_2(\RR^d)$ is another natural variant to consider, we chose to impose the continuity equation for practical purposes. This enables flowing between specific distributions/shapes/objects rather than equivalent classes thereof, which is desirable since we do not a priori know the underlying rotation. On a related note, reflective symmetry also arise in the counterexample to the optimality of the identity or anti-identity permutations for the one-dimensional GW problem \cite{beinert2022assignment,dumont2024existence}.

\end{remark}

\begin{proof}
The proof of \cref{thm:BB_IGW} strongly relies on the following lemma, which relates IGW curves to the continuity equation, similarly to  \cite[Theorem 8.3.1]{ambrosio2005gradient}. To enable the reparametrization from \cref{lem:metric_derivative}, we introduce the following definition: two curves $\alpha,\beta\in \lip_{\IGW}\big([0,1];\cP_2(\RR^d)\big)$ are said to be IGW equivalent, if there is a nondecreasing function $h:[0,1]\to[0,1]$, such that $\sup_{t\in[0,1]}\IGW(\alpha_t, \beta_{h(t)})=0$.

\begin{lemma}[IGW curves and continuity equation]\label{lem:IGW_curve_length_and_flow}
    Let $\rho:[0,1]\to\cP_2(\RR^d)$ be  an IGW-continuous curve with $\ell_\IGW(\rho)<\infty$. The following statements hold:
    \begin{enumerate}[leftmargin=*]
        \item If $\rho_t\in\cP_2^{\mathrm{ac}}(\RR^d)$ and $\lambda_{\min}(\bsigma_t)>c>0$ for all $t\in[0,1]$, then there exists an IGW equivalent curve $\tilde{\rho}$ that is $\W_2$-Lipschitz, such that $\tilde{\rho}_0= \rho_0$ and either $\tilde{\rho}_1=\rho_1$ or $\tilde{\rho}_1=\bI^-_\sharp\rho_1$, where $\tilde{\rho}$ solves the continuity equation $\partial_t\tilde{\rho}_t + \nabla\cdot \tilde{\rho}_t v_t =0$ for some $v_t\in L^2(\bar\rho_t;\RR^d)$, with
        \[
        \ell_\IGW(\rho)^2 = \ell_\IGW(\tilde{\rho})^2\geq\int_0^1 g_{\rho_t}(v_t,v_t) d t;\]
        \item Conversely, for any weakly continuous curve $\tilde{\rho}$ that is IGW equivalent to $\rho$, satisfying the continuity equation with $v_t\in L^2(\tilde\rho_t;\RR^d)$ such that $\int_0^1\|v_t\|_{L^2(\tilde\rho_t;\RR^d)}^2 dt<\infty$, we have \[
        \ell_\IGW(\rho)^2 = \ell_\IGW(\tilde{\rho})^2\leq \int_0^1 g_{\rho_t}(v_t,v_t) d t.\]
    \end{enumerate}
\end{lemma}

The lemma is proven in \cref{appen:lem:IGW_curve_length_and_flow_proof}. Given this result, the fact that $\mathsf{d}_{\IGW}^2$ is upper bounded by the action (namely, the right-hand side of \eqref{eq:IGW_BB}) is straightforward from the second part of \cref{lem:IGW_curve_length_and_flow}. To prove the opposite inequality, consider the $\mathsf{d}_{\IGW}$ constant-speed geodesic $\rho:[0,1]\to\cP_2(\RR^d)$ connecting $\mu_0,\mu_1$, which is $\mathsf{d}_{\IGW}(\mu_0,\mu_1)$-Lipschitz (see \cref{thm:intrinsic_metric_is_geodesic}). If $\rho\subset\cP_2^{\mathrm{ac}}(\RR^d)$ with $\inf_t\lambda_{\min}(\bsigma_{\rho_t})>0$, then by first part of \cref{lem:IGW_curve_length_and_flow}, we achieves the minimum.

The density condition in Part (1) of \cref{lem:IGW_curve_length_and_flow} guarantees that the IGW equivalent curve has a velocity field that satisfies the continuity equation with it. \cref{thm:BB_IGW}, on the other hand, does not impose the density requirement, which we remove using the following lemma.

\begin{lemma}[Approximation]\label{lem:approximating_curve_with_density}
    For any curve $\rho\in\lip_{\IGW}\big([0,1];\cP_2(\RR^d)\big)$ with $\lambda_{\min}(\bsigma_{\rho_t})\geq c>0$ and any $\epsilon\in(0,1)$, there exists a pair $\curve{\gamma_t^\epsilon,v_t^\epsilon}$ that solves the continuity equation with $\gamma_0^\epsilon=\rho_0,\gamma_1^\epsilon\in\{\rho_1,\bI^-_\sharp\rho_1\}$, such that
    \begin{align*}
        \int_0^1 g_{\gamma_t^\epsilon}(v_t^\epsilon,v_t^\epsilon) dt\leq \ell_{\IGW}(\rho)^2 + O(\epsilon).
    \end{align*}
\end{lemma}

The proof of this lemma, given in \cref{appen:lem:approximating_curve_with_density_proof}, constructs $\curve{\gamma_t^\epsilon,v_t^\epsilon}$ by employing Gaussian smoothing. Namely, the curve is assembled by connecting
    \[
        \rho_0\to \rho_0*\cN_\epsilon \to \rho_1*\cN_\epsilon\to \rho_1,
    \]
and showing that each piece is $\W_2$-Lipschitz with a corresponding velocity field. The intermediate piece, from $\rho_0*\cN_\epsilon$ to $\rho_1*\cN_\epsilon$ is accounted for by Item (1) of \cref{lem:IGW_curve_length_and_flow}. For the external pieces, connecting $\rho_0$ and $\rho_1$ with their Gaussian smoothed versions, one can readily verify $\W_2$-Lipschitz continuity and then invoke \cite[Theorem 8.3.1]{ambrosio2005gradient} to obtain the associated velocity fields. We then combine the curves (as well as their velocity fields) via a time rescaling argument to obtain $\curve{\gamma_t^\epsilon,v_t^\epsilon}$, and bound the overall action as stated in the lemma. In particular, the intermediate piece $\ell_\IGW(\rho)^2+O(\epsilon)$ to the action, while the two external pieces contribute another $O(\epsilon)$ each.

Applying \cref{lem:approximating_curve_with_density} to the $\mathsf{d}_{\IGW}$ constant-speed geodesic $\rho$, we obtain a pair $\curve{\gamma_t^\epsilon,v_t^\epsilon}$ satisfying the continuity equation and connecting $\mu_0,\mu_1$, with
\begin{align*}
    \int_0^1 g_{\gamma_t^\epsilon}(v_t^\epsilon,v_t^\epsilon) dt \leq \ell_\IGW(\rho)^2 + O(\epsilon) = \mathsf{d}_{\IGW}(\mu_0,\mu_1)^2 + O(\epsilon).
\end{align*}
As $\epsilon$ is arbitrary, the proof IGW Benamou-Brenier formula from \eqref{eq:IGW_BB} is concluded.

\end{proof}

\begin{remark}[Tangent space]
    The proof of Part (1) of \cref{lem:IGW_curve_length_and_flow}, found in \cref{appen:lem:IGW_curve_length_and_flow_proof}, provides an interesting insight into the structure of the IGW tangent space. While in the 2-Wasserstein case, the tangent space at $\mu$ is identified as the $L_2(\mu;\RR^d)$ closure of $\{v:\,v=\nabla\phi,\, \phi\in C_c^\infty(\RR^d)\}$, our proof suggests that the only nontrivial elements for the IGW tangent space are those $v$ with $\int xv(x)^\intercal d\mu(x)$ being PSD, or the invariant space $\cI_\mu$ of $\cL_{\bsigma_\mu,\mu}$ as we defined in \cref{rem:property_of_operator}. In fact, for any suitable IGW curve $\rho_t$ that satisfies the continuity equation with $w_t$, $\rho_t$ could be rotated pointwise to $(\bar{\rho}_t,v_t)$, where $v_t\in \cI_{\bar{\rho}_t}$, without changing the IGW length (see proof of Part (1) 1 of \cref{lem:IGW_curve_length_and_flow}), and $g_{\bar{\rho}_t}(v_t,v_t) = g_{\rho_t}(w_t,w_t)$ since the length is unchanged. Hence, any direction $w\in L^2(\mu;\RR^d)$ can be replaced with some $v\in\cI_\mu$, while the associated flow remains IGW equivalent and the length under Riemannian metric tensor $g$ is unchanged. 
    It is known that the 2-Wasserstein tangent space adheres to a variational selection criterion, i.e., tangent vectors are selected to have the minimal $L_2$ norm among all of their divergence-free permutations, see \cite[Lemma 8.4.2]{ambrosio2005gradient}. For IGW, in addition to the minimal selection w.r.t. $g_{\mu}(v,v)$, one also must account for the aforementioned PSD property, which can be viewed as performing minimal selection inside $\cI_\mu$.  
\end{remark}

\subsection{IGW gradient}\label{subsec:IGW_grad}

We conclude this section with a formal derivation of the IGW gradient. Consider a functional $\sF:\cP_2(\RR^d)\to\RR\cup\{\infty\}$ and proper curve $(\rho_t)_{t\in(-\epsilon,\epsilon)}\subset\cP_2(\RR^d)$ that passes through $\rho_0=\mu$, with $\partial_t\rho_t\big|_{t=0}=-\nabla\cdot\mu v$, for some $v\in L^2(\mu;\RR^d)$. Following Otto's formalism \cite{otto2001geometry}, we identify $v$ as a tangent direction at $\mu$ (see also \cite[Chapter 8]{ambrosio2005gradient}). We next identify the gradient of $\sF$ at $\mu$ under the Riemannian structure $g$, denoted by $\grad_{\mathsf{d}_{\IGW}}\sF(\mu)$, as $\mathsf{d}_{\IGW}$ is the metric induced by $g$.
Namely, we look for an element from the cotangent space that satisfies

\[
\partial_t\sF(\rho_t)\big|_{t=0}=g_\mu\big(\grad_{\mathsf{d}_{\IGW}}\sF(\mu),v\big).
\]
Overlooking regularity issues, consider the following steps
\begin{align*}
    \partial_t \sF(\rho_t)\big|_{t=0} &= \int \delta\sF(\mu) \partial_t\rho_t\big|_{t=0}\\
    & = -\int \delta\sF(\mu) \nabla\cdot \mu v \\
    & = \llangle \nabla\delta\sF(\mu), v\rrangle_{L^2(\mu;\RR^d)}\\
    &=g_\mu\big(\cL_{\bsigma_\mu,\mu}^{-1}[\nabla\delta\sF(\mu)],v\big),
\end{align*}
where the first step is the chain rule, the second one uses the fact that $\partial_t\rho_t\big|_{t=0}=-\nabla\cdot\mu v$, the third step comes from integration by parts, while the last step follows by definition of our metric tensor from \eqref{eq:metric_tensor}. We thus conclude that 
\begin{equation}
\grad_{\mathsf{d}_{\IGW}}\sF(\mu)=\cL_{\bsigma_\mu,\mu}^{-1}[\nabla\delta\sF(\mu)].\label{eq:IGW_grad}
\end{equation}
This essentially recovers our main PIDE for the gradient flow from \cref{thm:main}, whereby $v=-\grad_{\mathsf{d}_{\IGW}}\sF(\mu)$ is the direction of the steepest descent of $\sF$ in the IGW intrinsic geometry. Note, however, that while \eqref{eq:IGW_grad} is the gradient w.r.t. the intrinsic IGW metric $\mathsf{d}_{\IGW}$, \cref{thm:main} corresponds to an implicit scheme for the gradient flow w.r.t. the IGW distance itself. These two notions coincide for the Wasserstein distance, but that may not be so for IGW. Nevertheless, we abuse notation and define the IGW gradient as $\grad_\IGW \coloneqq \grad_{\mathsf{d}_{\IGW}}$, despite IGW not possessing a tangent structure. The IGW gradient flow equation from \cref{thm:main} thus reads
\begin{align*}
    v_t=-\grad_{\IGW}\sF(\rho_t).
\end{align*}

\medskip
We conclude this section by fleshing out the relationship between IGW and Wasserstein gradients. Writing $\grad_\W \sF(\mu)$ for the 2-Wasserstein gradient and recalling that $\grad_\W \sF(\mu)=\nabla\delta\sF(\mu)$ \cite{otto2001geometry}, we see that 
\begin{equation}
\grad_\IGW\sF(\mu)=\cL_{\bsigma_\mu,\mu}^{-1}[\grad_\W \sF(\mu)].\label{eq:W_IGW_grad}
\end{equation}
The IGW gradient is thus obtained by transforming the Wasserstein gradient using the inverse of the mobility operator. As discussed after its definition in \eqref{eq:operator}, the action of $\cL^{-1}$ serves to align the velocity field to encourage particles to move along similar directions (see also \cref{fig:L_example}). 

To further compare the induced gradient flows, i.e., IGW versus Wasserstein, recall that by \cref{rem:property_of_operator}, the velocity field from the IGW gradient flow equation can be written as 
\begin{align}
    v_t(x) = -\frac{1}{2}\bsigma_t^{-1} \nabla\delta\sF(\rho_t)(x)+ \frac{1}{2} x^\intercal\otimes \bI (\bI\otimes\bsigma_t^2+\bsigma_t\otimes \bsigma_t)^{-1} \int (y\otimes \bI) \nabla\delta\sF(\rho_t)(y) d \rho_t(y).
    \label{eq:gradient_flow_velocity}
\end{align}
We have seen above that $\partial_t \sF(\rho_t) = \llangle \nabla\delta\sF(\rho_t), v_t\rrangle_{L^2(\rho_t;\RR^d)}$, and by plugging \eqref{eq:gradient_flow_velocity} in, we obtain
\begin{align*}
     \partial_t \sF(\rho_t) 
    &=-g_{\rho_t}\big(\cL_{\bsigma_t,\rho_t}^{-1}[\nabla\delta\sF(\rho_t)],\cL_{\bsigma_t,\rho_t}^{-1}[\nabla\delta\sF(\rho_t)]\big)\\
    &= \underbrace{- \llangle \nabla\delta\sF(\rho_t), \frac{1}{2}\bsigma_t^{-1} \nabla\delta\sF(\rho_t)\rrangle_{L^2(\rho_t;\RR^d)}}_{\text{Descent}} \\
    &\underbrace{+ \frac{1}{2} \left(\int (x\otimes \bI) \nabla\delta\sF(\rho_t)(x) d \rho_t(x) \right)^\intercal \mspace{-3mu}(\bI\otimes\bsigma_t^2+\bsigma_t\otimes \bsigma_t)^{-1}\mspace{-3mu} \int (y\otimes \bI) \nabla\delta\sF(\rho_t)(y) d \rho_t(y)}_{\text{Damping}}.\numberthis\label{eq:gradient_flow_functional_derivative_on_trajectory}
\end{align*}
This decomposes the IGW gradient flow into two components: a linear transformation of the Wasserstein gradient, termed the \emph{descent term}, and an integral transformation, termed \emph{damping}. The names arise from the characteristics of these terms, as the descent term is negative and aligned with the Wasserstein flow structure, while the damping term is positive and slows down the flow to encourage interparticle alignment. This decomposition, along with the distinct effects of descent and damping, is further explored and illustrated in the numerical experiment in the next section.

\begin{example}[IGW gradient computation]
    As a simple example, we compute the IGW gradient for potential functional $\sV(\mu) = \frac{1}{2}\int \|x\|^2 d\mu(x)$. Clearly $\grad_\W \sV(\mu)=\nabla\delta\sV(\mu) = x$, and by invoking the inverse formula from \cref{prop:solution_of_integral_system}, we obtain
    \[
        \grad_\IGW \sV(\mu)=\cL_{\bsigma_\mu,\mu}^{-1}[\grad_\W \sV(\mu)] = \cL_{\bsigma_\mu,\mu}^{-1}[x]= \frac{1}{2}\bsigma_\mu^{-1} x - \bsigma_\mu^{-1} \bB x,
    \]
    where $\bB$ is the solution to the Sylvester equation $\bsigma_\mu\bB + \bB\bsigma_\mu = \frac{1}{2}\bsigma_\mu$. Assuming that $\bsigma_\mu$ is nonsingular, the latter can be obtained via $\vectorize(\bB)=\bM^{-1}\bL$, where $\bM=\bI\otimes\bsigma_\mu+\bsigma_\mu\otimes\bI$ and $\bL=\vectorize(\frac{1}{2}\bsigma_\mu)$. Note that when $\mu$ is isotropic, i.e., $\bsigma_\mu=\bI$, we have $\bB=\frac{1}{4}\bI$ and therefore $\grad_\IGW\sV(\mu)=\frac{1}{4}x$. The latter coincides with the Wasserstein gradient, up to a constant factor. This is a consequence of the symmetry of both $\sV$ and $\mu$, as rotations are no longer useful for preserving the shape. Consequently, the particles along both the IGW and Wasserstein flows follow the same trajectory, as the gradient directions coincide. To compute $\grad_\IGW\sF$ for other functionals, one should first evaluate the Wasserstein gradient and then apply the inverse mobility operator to it to obtain $\grad_\IGW\sF(\mu)=\cL_{\bsigma_\mu,\mu}^{-1}[\grad_\W \sF(\mu)]$, as per \cref{rem:property_of_operator}.
\end{example}

\section{Numerical experiments}\label{sec:experiment}

We present numerical experiments for the IGW gradient flow and dynamical formulation from Theorems \ref{thm:main} and \ref{thm:BB_IGW}, respectively.\footnote{Demo code for all experiments, as well as additional ones not featuring in the text, is available at \url{https://github.com/ZhengxinZh/IGW}} For the gradient flow, we directly compute the IGW gradient of some functionals of interest, and apply the forward Euler scheme. For the dynamical formula, we parametrize the velocity field by a neural network and utilize the neural ODE framework \cite{chen2018neural} to obtain the flow trajectory between source and target distributions. The distributions considered throughout this section are given as point clouds, i.e., discrete uniform distributions over points in Euclidean space. Additional numerical results are provided in \cref{appen:additional_experiments}. 

\subsection{Gradient flow}
We compute the forward Euler scheme for the IGW gradient flow initiated at different point clouds. Consider three functionals: the potential energy, two-dimensional Coulomb interaction energy, and entropy. These functionals are respectively defined as:
\begin{align*}
    \sV(\mu)&\coloneqq \int \|x\|^2/2 d\mu(x)\\
    \sC(\mu)&\coloneqq -\int \log(\|x-x'\|) d\mu\otimes\mu(x,x')\\
    \sH(\mu)&\coloneqq\int_{\RR^d} \frac{d\mu}{dx} \log\left(\frac{d\mu}{dx}\right) dx.
\end{align*}
We maintain the notation $\sF$ for a generic functional, when describing the computational pipeline.

\medskip
We start by evaluating the Wasserstein gradient of these functionals, and then obtain the IGW gradient by applying the inverse mobility operator, as per \eqref{eq:W_IGW_grad}. For the potential energy, we have $\grad_\W\sV(\mu)(x)=\nabla\delta\sV(\mu)(x)=x$. For the other two functionals, write $\mu=\frac{1}{n}\sum_{i=1}^n\delta_{x_i}$ for the point cloud supported on $\{x_i\}_{i=1}^n$, and consider the following surrogates. We approximate the Coulomb interaction by adding a smoothing parameter $\epsilon=0.2$ and ignoring the diagonal of the distance matrix:
\begin{align*}
    \widetilde{\sC}(\mu)\coloneqq  -\frac{\sum_{i\neq j}\log(\epsilon +\|x_i-x_j\|^2)}{2n(n-1)}.
\end{align*}
To approximate the Wasserstein gradient of $\sC$, we consider the derivative of $\widetilde{\sC}$ w.r.t. the norm of the point difference, yielding
\begin{align*}
    \widetilde\grad_\W\sC(\mu)(x_i)\coloneqq -\frac{1}{n-1}\sum_{1\leq j\leq n, j\neq i}\frac{x_i-x_j}{\epsilon+\|x_i-x_j\|^2}.
\end{align*}

For the entropy, inspired by the classical Kozachenko-Leonenko $k$-nearest neighbor estimator \cite{kozachenko1987sample}, we approximate the functional and its Wasserstein gradient via 
\begin{align*}
    \widetilde{\sH}(\mu)&\coloneqq \frac{1}{n} \sum_{i=1}^n \log(\epsilon + \min_{j\neq i}\|x_i-x_j\|^2)/2,\\
    \widetilde\grad_\W\sH(\mu)(x_i) &\coloneqq \frac{x_i-x_{\argmin_{j\neq i}\|x_i-x_j\|}}{\epsilon + \min_{j\neq i}\|x_i-x_j\|^2},
\end{align*}
where the latter expression is obtained like in the Coulomb interaction case. With these estimates of the Wasserstein gradients, we obtain the IGW gradient via $\grad_\IGW\sF(\mu) = \cL_{\bsigma_\mu,\mu}^{-1}[\grad_\W\sF(\mu)]$ using the expression for the inverse mobility operator from \cref{rem:property_of_operator}.

\medskip
For initialization, we consider several shapes of different geometric features and symmetries (see also \cref{appen:additional_experiments}). \cref{fig:gradient_flow_ellipse} illustrates the gradient flow trajectories initiated at a ellipse, and the associated velocity fields, both under IGW and the  2-Wasserstein distance. To remain compliant with the continuity equation $\partial_t\rho_t+\nabla\cdot\rho_t v_t=0$, we use the flow ODE $    \frac{d}{dt} x_t = v_t(x_t)$, for $v_t=-\grad_\IGW\sF(\rho_t)$ or $v_t=-\grad_\W\sF(\rho_t)$. For $\tau=\frac{T}{k}$ and $j=0,\ldots,k-1$, the forward explicit Euler scheme reads
\begin{align*}
    x_i^{t_{j+1}} = x_i^{t_{j}} + \tau v_{t_j}(x_i^{t_{j}}),\quad t_j=j\tau.
\end{align*}
We set $\tau=0.01$ and choose $T$ to ensure numerical stability and prevent the algorithm from diverging. The trajectories for $\sV$, $\sC$, and $\sH$ are illustrated in \cref{fig:gradient_flow_ellipse}, while the functional decay is shown in \cref{fig:functional_decay}.

\begin{figure}[h!]
    \centering
    \begin{subfigure}[t]{1\textwidth}
    \centering
    \includegraphics[width=1\textwidth]{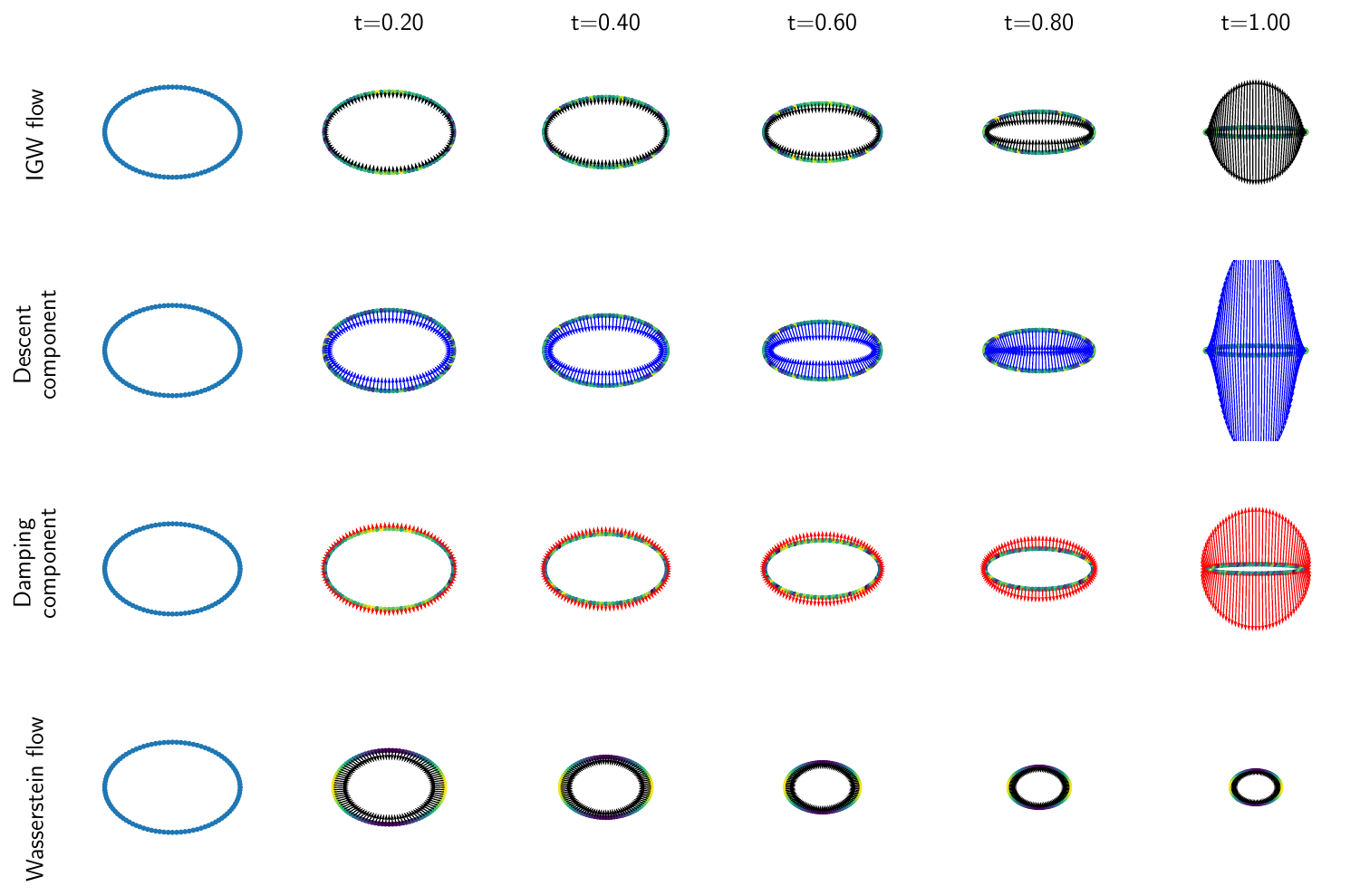}
    \vspace{-2em}
    \caption{Potential}
    \end{subfigure}
\end{figure}

\begin{figure}[H]\ContinuedFloat
    \begin{subfigure}[t]{1\textwidth}
    \centering
    \includegraphics[width=\textwidth]{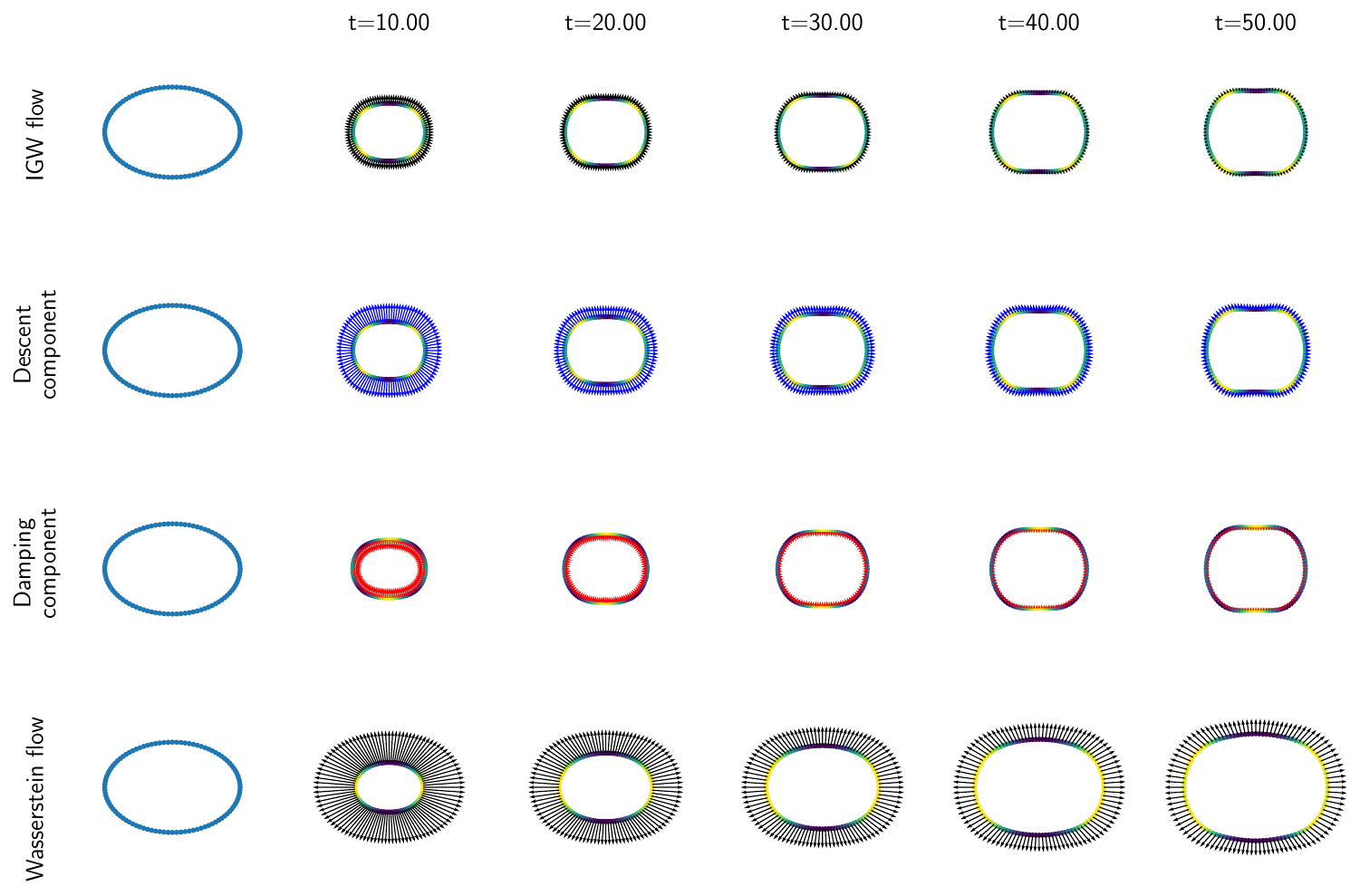}
    \vspace{-0.8em}
    \caption{Interaction}
    \end{subfigure}
    \end{figure}

     \begin{figure}[H]\ContinuedFloat
    \begin{subfigure}[t]{1\textwidth}
    \centering
    \includegraphics[width=\textwidth]{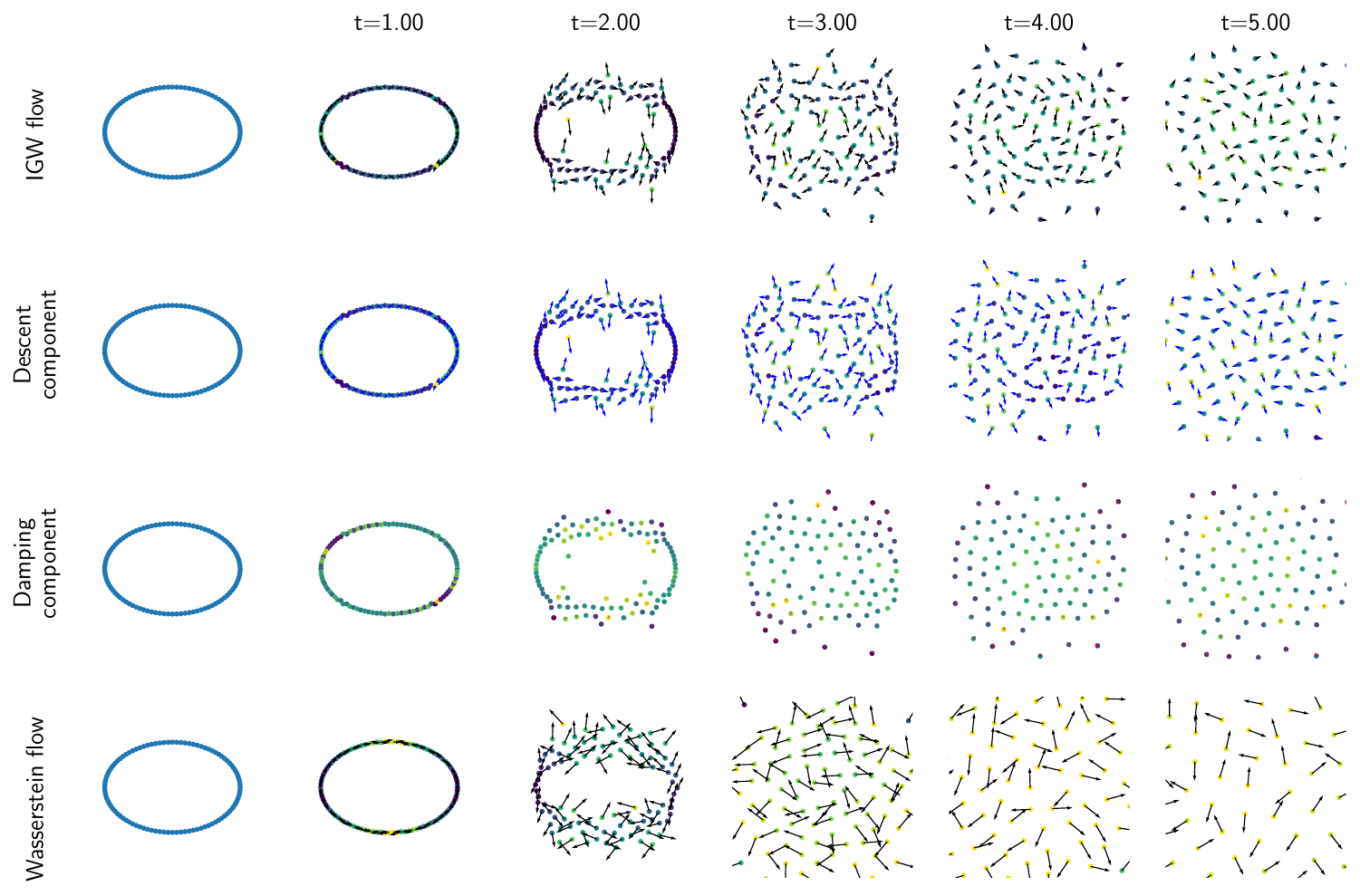} 
    \caption{Entropy}
    \end{subfigure}
     \caption{Gradient flows, starting from an initial distribution given by an ellipse point cloud. For each functional, the first three rows plot 5 snapshots of the vector fields of the overall gradient descent trajectory $\grad_\IGW\sF(\rho)$ (first row), as well as the descent part and the damping components (second and third rows, respectively). For comparison, the Wasserstein gradient flow is shown in the fourth row.}\label{fig:gradient_flow_ellipse}

\end{figure}

\begin{figure}[H]
    \begin{subfigure}[b]{0.33\textwidth}
   \includegraphics[height=1.4in]{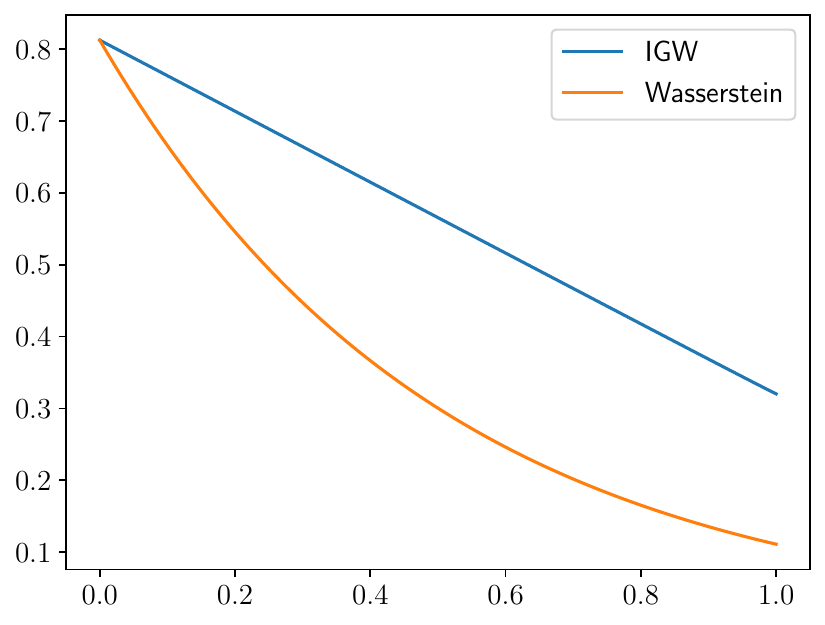}
    \caption{Potential}
    \end{subfigure}
    ~
    \begin{subfigure}[b]{0.33\textwidth}
    \centering
    \includegraphics[height=1.4in]{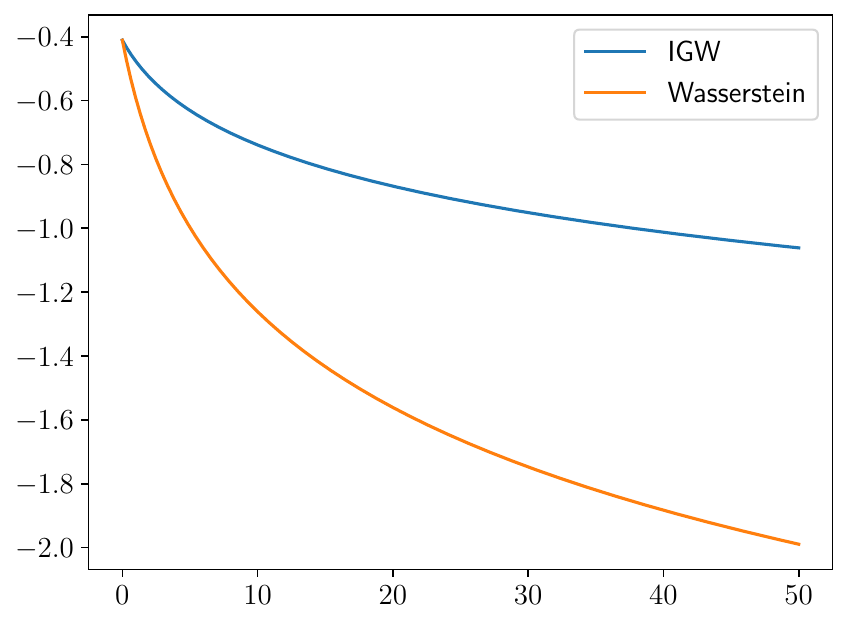}
    \caption{Interaction}
    \end{subfigure}
    ~
    \begin{subfigure}[b]{0.33\textwidth}
    \centering
    \includegraphics[height=1.4in]{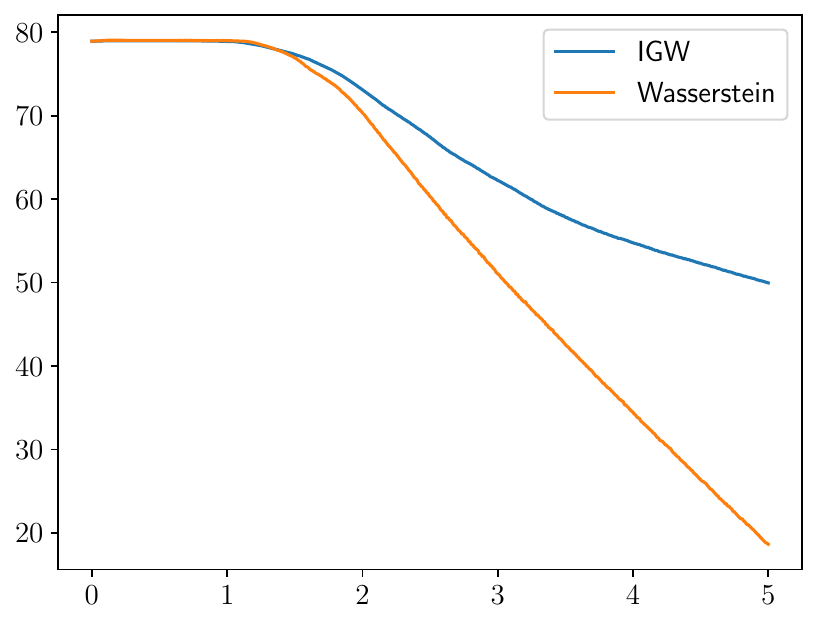}
    \caption{Entropy}
    \end{subfigure}
    \caption{Decay of the functional value. The Wasserstein gradient flow exhibits faster decay, in particular, at an exponential rate for strongly convex functional like the potential energy. The IGW flow usually decays slower, as was discussed in \cref{subsec:IGW_grad}, due to the damping term.}\label{fig:functional_decay}
\end{figure}

Each subfigure in \cref{fig:gradient_flow_ellipse} corresponds to a different functional out of $\sV$, $\sW$, and $\sH$, where the first three rows illustrate different components of the IGW gradient flow, while the fourth row is the Wasserstein flow. For IGW, we present the trajectory of the full gradient flow (which comprises both descent and damping; see \eqref{eq:gradient_flow_functional_derivative_on_trajectory}), as well as those of the descent and damping components separately. The particles themselves are color-coded according to the value of $\llangle \nabla\delta\sF(\rho_t), v_t\rrangle_{L^2(\rho_t;\RR^d)}$. The corresponding velocity fields are shown using arrows pointing out of the particle, colored in black, blue, and red for the full flow, descent, and damping, respectively.

Notice that the damping component of $v_t$ typically points to the opposite direction of the descent components. This echoes the observation from \eqref{eq:gradient_flow_functional_derivative_on_trajectory} that the damping part slows down the decay of the functional in favor of preserving the shape and ensuring alignment of the particles. For potential energy $\sV$ in Subfigure (A), the nonuniform deformation of the shape leads to near-singular covariance matrix, which is likely to cause the PIDE formulation to diverge. This is in line with the observation from \cref{rem:infinite_time_extension}, that preserving nonsingularity of the covariance is necessary for the gradient flow equation from \cref{thm:main} to be defined. For the Coulomb interaction, we note the nonuniform deformation that leads to an initially convex boundary evolving into a nonconvex one. Finally, for entropy we again observe the slower diffusion of particles, with the final (max-entropy) configuration maintaining relative proximity between particles, compared to the Wasserstein gradient flow. 

\cref{fig:functional_decay} shows that the Wasserstein gradient flow shrinks the functional value faster. This is expected since the IGW flow can be viewed as a constrained version, that seeks not only to minimize the functional, but also to avoid distorting the shape. The figure also shows the well-known exponential decay of the Wasserstein gradient flow \cite{ambrosio2005gradient}. The functional decay under the IGW gradient flow presents different profiles and its rate of convergence is left as an appealing question for future research.

\subsection{Flow matching examples}

We next visualize the IGW dynamical formulation from \cref{thm:BB_IGW}:
\begin{equation*}
    \mathsf{d}_{\IGW}(\mu_0,\mu_1)^2 = \min_{\mu\in\{\mu_1,\bI^-_\sharp\mu_1\}} \inf_{\substack{(\rho_t,v_t):\\\partial_t \rho_t + \nabla\cdot  (\rho_t v_t)=0\\
    \rho_0=\mu_0, \rho_1=\mu}} \int_0^1  g_{\rho_t}(v_t,v_t) dt.
\end{equation*}
To identify an IGW minimal particle flow between two point clouds $\mu_0$ and $\mu_1$, we optimize the action over connecting flow fields. For simplicity, we either use symmetric shapes or ones that have a clear minimal path not requiring a reflection. We solve the variational problem above by parametrizing the flow field $v:[0,1]\times\RR^d\to\RR^d$ by a neural network $v^{\mathrm{NN}}$, which has two hidden layers with 50 neurons each and $\tanh$ activations. The flow of $v^{\mathrm{NN}}$ is then computed using the neural ODE framework \cite{chen2018neural}, with $k+1$ steps $\mu_0=\rho_0,\rho_{t_1},\ldots,\rho_{t_k}$. To enforce the boundary condition, we add a maximum mean discrepancy (MMD) \cite{gretton2012kernel} penalty between $\rho_{t_k},\mu_1$ to the cost. The overall cost reads:
\begin{align*}
    \frac{1}{k} \sum_{j=1}^k g_{\rho_{t_j}}(v^{\mathrm{NN}}_{t_j},v^{\mathrm{NN}}_{t_j}) + \lambda \MMD(\rho_{t_k},\mu_1),
\end{align*}
where $\MMD(\mu,\nu)\coloneqq\int k(x,x')d(\mu-\nu)\otimes(\mu-\nu)(x,x')$ and $k$ is sum of Gaussian kernels of bandwidths $\sigma\times\{.0001,.001,.01,.05,.25,1,4,20,100,1000\}$ with $\sigma=0.03$. We set the default regularization parameter to $\lambda=100$, and take $k=10$, epoch number $2000$, and learning rate 0.01, where $\lambda$, epoch and learning rate are then adjusted accordingly for different cases. When $v^{\mathrm{NN}}$ is trained to achieve a small MMD value between $\rho_{t_k},\mu_1$, we view it as a connecting flow field. For comparison, in addition to the IGW action, we also train for the Wasserstein action $\int_0^1\|v_t\|_{L^2(\rho_t;\RR^d)}^2dt$ and present the obtained trajectory in the second row of each example.

\cref{fig:flow_matching}  shows that the IGW flow captures and preserves the shape information along the deforming trajectory, while the Wasserstein flow only seeks to minimize the transportation cost. In particular, for the first experiment, between a cat shape and its rotation, the IGW flow identifies a continuous curve of rotations. Furthermore, the IGW distance value between any point along the trajectory and the target shape remains small, suggesting that the shape is preserved throughout (recall the invariance of IGW to orthogonal transformation). In contrast, the Wasserstein flow significantly distorts the shape; for example, in the first example, part of the cat's tail is transported to its head.

\begin{figure}[H]
    \centering
    \includegraphics[width=0.87\textwidth]{fig/flow_matching/cat.pdf}

\includegraphics[width=0.87\textwidth]{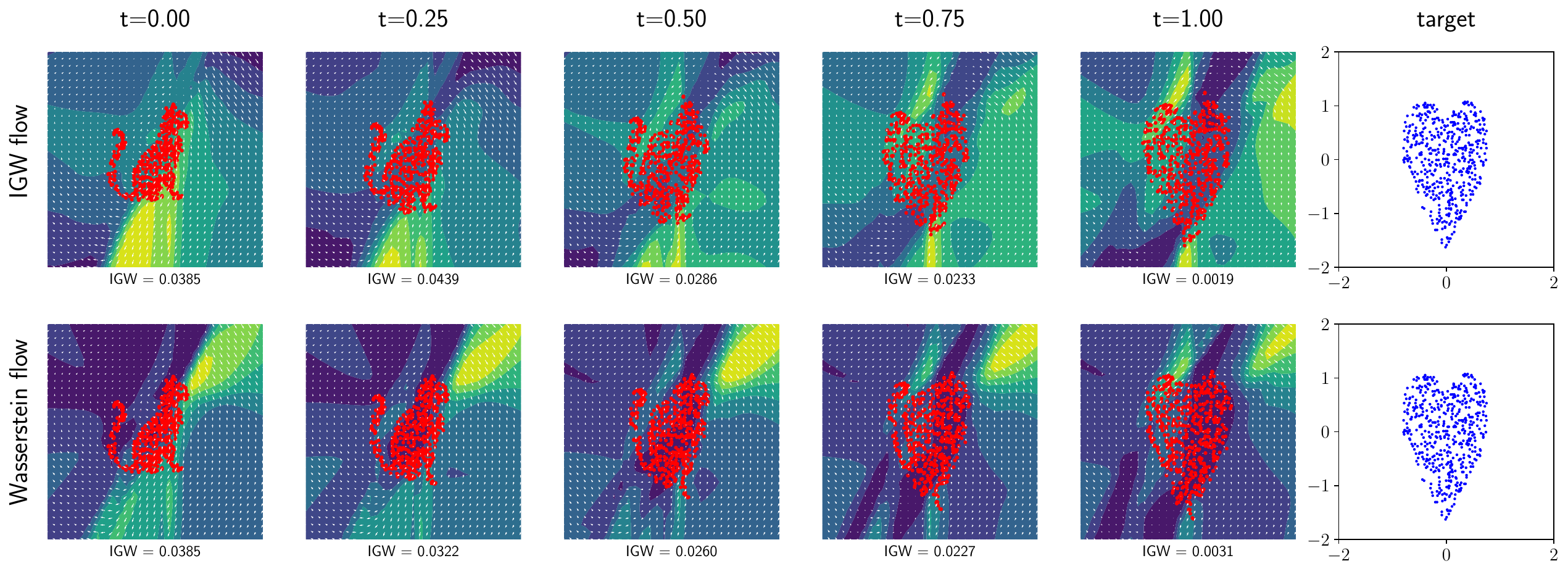}

\includegraphics[width=0.87\textwidth]{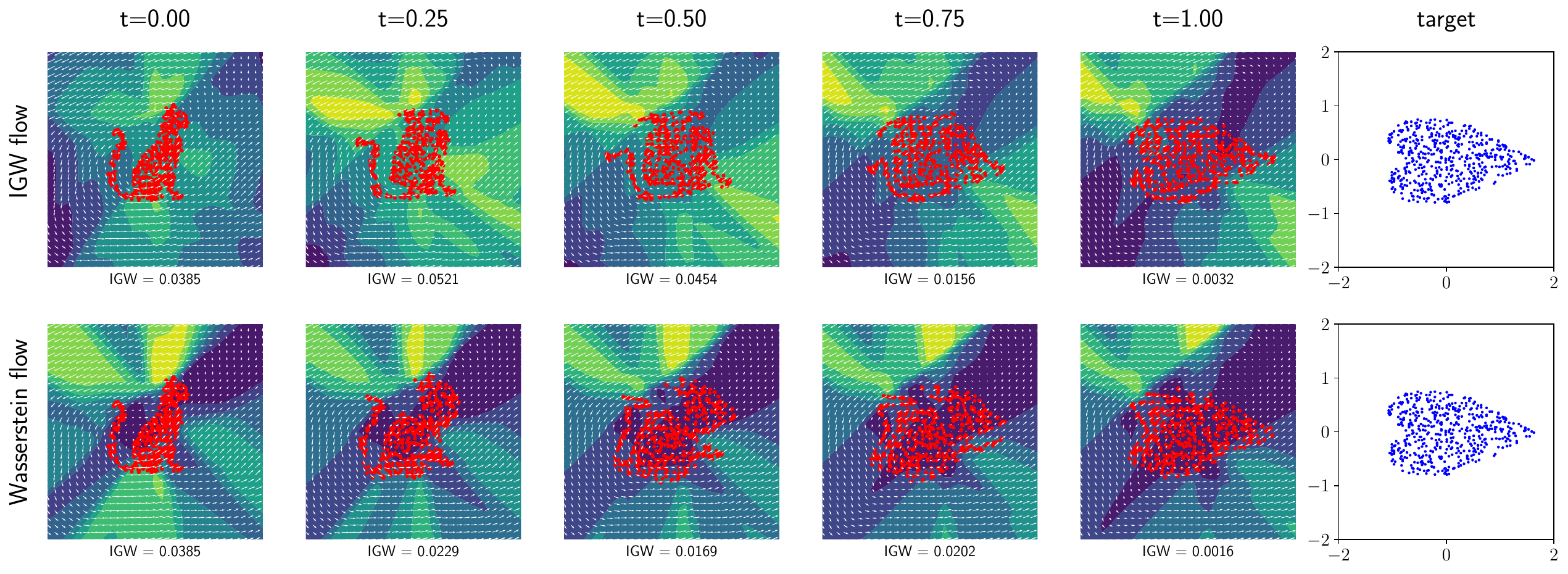}
    \caption{Flow matching between: (i) a cat shape and its rotation, (ii) a cat and a heart, (iii) a cat and a rotated heart. Each example comprises two rows, presenting the IGW and Wasserstein flows, respectively. In Example (i), the IGW flow identifies the rotation while the Wasserstein flow distorts the shape. Both flows present a similar behavior in Example (ii). In Example (iii), the IGW flow seems to follow a similar path to that of Example (ii), where the heart is not rotated, plus an additional rotation operation. The flow fields are plotted as vector field (white arrows), and the contour plot is the extrapolated local cost, i.e., we evaluate $v^{\mathrm{NN}}$ on a surrounding grid, and compute $\llangle v^{\mathrm{NN}}_{t_j}(x)), \cL_{\bsigma_{\rho_{t_j}},\rho_{t_j}}[v^{\mathrm{NN}}](x)\rrangle$. From the contour plot we see that the flow tends to transport the distribution to the low energy part.}\label{fig:flow_matching}
\end{figure}

To further illustrate the ability of the IGW flow to identify rotation paths, we present in \cref{fig:flow_matching_planes} a more complex matching between 3D shapes. We consider two different biplane models from the Princeton Shape benchmark \cite{shilane2004princeton}, represented by point clouds with $n=800$ samples. The IGW flow maintains low distortion along the trajectory and identifies a near-rotation path, as the source and target shapes are similar. The Wasserstein flow, on the other hand, again significantly distorts the shape. This suggests that the IGW flow can find low distortion paths between shapes that are similar up to orthogonal transformations, even when the orthogonal group is not amenable for simple parametrization as in 2D.

\begin{figure}[H]
    \centering
    \includegraphics[width=0.99\textwidth]{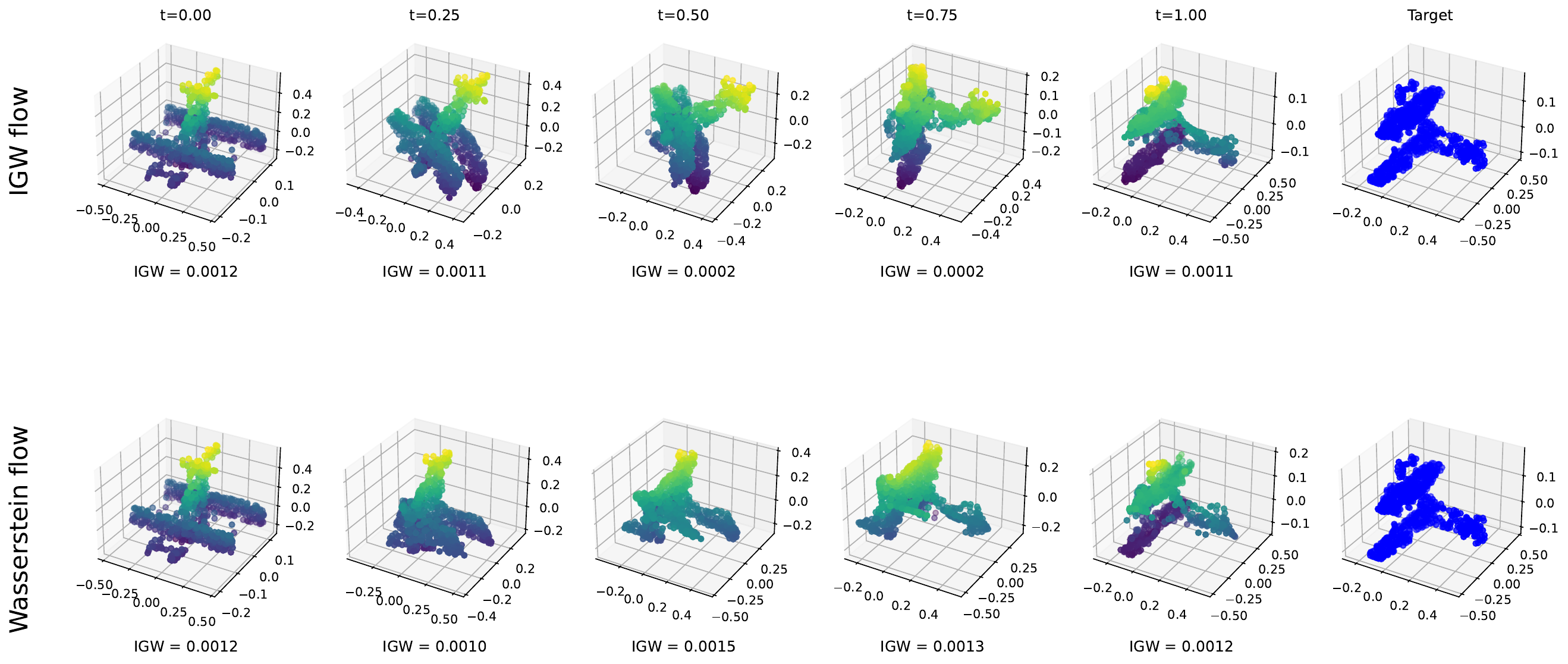}
    \caption{Flow matching between 3D biplanes of similar (yet nonidentical) shape but different orientation.}\label{fig:flow_matching_planes}
\end{figure}

\section{Concluding Remarks}
This work presented a detailed study of gradient flows and Riemannian structure in the IGW geometry. Our main result for the gradient flow addressed the convergence and limiting characterization of the IGW GMM. The analytical challenge arises from the lack of convexity of IGW along geodesics, compared to classical OT. To overcome this, we utilize the duality formula to linearize the distance and calculate the Fr\'echet subdifferential, which enables an explicit construction of the GMM sequence. Our construction results in a continuous-time limit that follows the continuity equation, thereby instantiating the IGW flow in $\cP_2(\RR^d)$ (as opposed to a quotient space). The flow follows a velocity field obtained from a transformation of the Wasserstein gradient via the inverse mobility operator. The transformed velocity introduces a damping component to the descent movement, whose role is to preserve the global structure of the distribution. 
We then leverage the identified differential structure to derive the Riemannian metric tensor that gives rise to the intrinsic IGW geometry. Specifically, the proposed Riemannian structure induces the (globally defined) intrinsic IGW metric, resulting in a Benamou-Brenier-like dynamical formulation. Overall, our results open the door to viewing IGW as a transport problem from the flow perspective. We believe that this new geometry on $\cP_2(\RR^d)$, which exhibits marked differences from the Wasserstein case, may be of wider interest.

Future research directions stemming from this work are abundant. Firstly, deriving convergence rates of the functional along the IGW gradient flow is an appealing endeavor. It would also be interesting to consider other similarity functions $c_\cX(x,x')$, in addition to the $\llangle x,x'\rrangle$ used for IGW. In particular, treating the full $(p,q)$-GW distance (see \eqref{eq:GW_intro}) from the flow and Riemannian geometry perspective is an important next step, which seems to be beyond the reach of our current analysis techniques. Even for the quadratic GW distance, where $p=q=2$, our analysis does not directly apply due to the absence of Gromov-Monge maps. We also ask what would be the geometry that other GW discrepancies will induce? For similarity functions with different invariance properties to the inner product cost considered herein, e.g., if $c_\cX(f(x),f(x'))=c_\cX(x,x')$ for any $f$ in an $\cX$ automorphism group $G$, deriving the differential and Riemannian structure seems like a fascinating task. We highlight that the linearity of the inner product kernel enabled many aspects of our analysis, with properties of its invariance group $\Od$, such as connectivity and smoothness, playing a crucial role (e.g., $\bI^-$ in \cref{thm:BB_IGW}). New analysis techniques will be needed to treat more general cases. 

Another interesting research direction involves the mobility operator and a better understanding of its properties. We have observed that the inverse mobility serves to align the particles during the movement and preserve the global structure. A deeper understanding involves studying the operator for its fixed points, eigenvalues, eigenfunctions, etc. Indeed, an eigen-decomposition of the mobility operator can shed new light on its dominant components and reinforce our understanding of the discovered descent and damping terms in the IGW gradient. It would also be interesting to see how the mobility operator generalizes to other GW distances, as discussed above.

\section{Proofs for Section \ref{sec:background}}\label{appen:background_proofs}

\subsection{Proof of \cref{lem:F2Dual}}\label{appen:F2Dual_proof}
    The variational form \eqref{eq:F2Decomp} was established in \cite[Lemma 2]{rioux2023entropic}, so we only prove the correspondence between $\bA^\star$ and $\pi^\star$. The sufficient part is straightforward. For necessary part, we first put the $\inf$ over $\Pi(\mu,\nu)$ in $\mathsf{IOT}$ outside, which gives a joint minimizing problem
    \begin{align*}\numberthis\label{eq:igw_joint_minimization}
        \inf_{\bA\in\RR^{d_x}\times \RR^{d_y}} \inf_{\pi\in\Pi(\mu,\nu)}8\|\mathbf{A}\|_\F^2+ \int -8x^{\intercal}\mathbf{A}y d\pi(x,y),
    \end{align*}
    where the functional is strongly convex in $\bA$. Thus for any optimal pair $(\bA^\star,\pi^\star_{\bA^\star})$, $\bA^\star$ minimizes $8\|\mathbf{A}\|_\F^2+ \int -8x^{\intercal}\mathbf{A}y d\pi^\star_{\bA^\star}(x,y)$, which by convexity is achieved at $\bA^\star=\frac{1}{2}\int xy^\intercal\pi^\star_{\bA^\star}(x,y)$. Plugging this form back, we obtain the minimum of the joint minimization problem as
    \begin{align*}
        &8\left\|\frac{1}{2}\int xy^\intercal\pi^\star_{\bA^\star}(x,y)\right\|_\F^2+ \int -8x^{\intercal}\left(\frac{1}{2}\int xy^\intercal\pi^\star_{\bA^\star}(x,y)\right)y d\pi^\star_{\bA^\star}(x,y)\\
        &=-2 \left\|\int xy^\intercal\pi^\star_{\bA^\star}(x,y)\right\|_\F^2\\
        &=-2 \int \llangle x,x'\rrangle\llangle y,y'\rrangle d\pi^\star_{\bA^\star}\otimes \pi^\star_{\bA^\star}(x,y,x',y').
    \end{align*}
    Since the minimum equals $\sF_2(\mu,\nu)$, comparing the form of the functional we conclude that $\pi^\star_{\bA^\star}$ is optimal for $\IGW(\mu,\nu)$.\qed

\subsection{Proof of \cref{lem:gromov_monge}}\label{appen:gromov_monge_proof}
    By assumption, we have an optimizer $(\bA^\star,\pi^\star)$ of the joint optimization
    \[
        \sF_2(\mu,\nu)=\inf_{\bA\in\RR^{d_x}\times \RR^{d_y}} \inf_{\pi\in\Pi(\mu,\nu)}8\|\mathbf{A}\|_\F^2+ \int -8x^{\intercal}\mathbf{A}y d\pi(x,y),
    \]
    and by \cref{lem:F2Dual}, $\pi^\star$ is optimal for $\mathsf{IOT}_{\bA^\star}(\mu,\nu)$. We have
    \begin{align*}
        \mathsf{IOT}_{\bA^\star}(\mu,\nu) &  = \inf_{\pi\in\Pi(\mu,\nu)}\int -8x^\intercal\bA^\star y d\pi(x,y)\\
        &\stackrel{(a)}= \inf_{\pi'\in\Pi(\mu,(8\bA^\star)_\sharp\nu)}\int -x^\intercal z d\pi'(x,z)\\
        &\stackrel{(b)}= \int -x^\intercal T^{\mu\to{(8\bA^\star)_\sharp\nu}}(x) d\mu(x)\\
        &\stackrel{(c)}= \int -8x^\intercal \bA^\star T^\star(x) d\mu(x)
    \end{align*}
    where (a) uses the bijection of the coupling set, (b) follows directly from Brenier's theorem (\cref{thm:brenier}), and (c) is by the definition of $T^\star$, while noting that $T^\star_\sharp\mu=\nu$. Observe that $(\id,8\bA^\star)_\sharp\pi^\star$ minimizes the RHS of (a), which by \cref{thm:brenier} also equals $(\id,T^{\mu\to(8\bA^\star)_\sharp\nu})_\sharp\mu$. Thus, we conclude that $\pi^\star = (\id,T^\star)_\sharp\mu$ and $\bA^\star = \frac{1}{2}\int x\,T^\star(x)^\intercal d\mu(x)$.\qed

\section{Proofs for Section \ref{sec:IGW_structure}}\label{appen:IGW_structure_proofs}

\subsection{Proof of \cref{prop:IGW_pmetric}}\label{appen:prop:IGW_pmetric_proof}
 Symmetry, non-negativity, and the sufficient condition for nullification are straightforward. For the triangle inequality, take $\mu,\nu,\rho\in\cP_2(\cH)$, let $\pi_1$ and $\pi_2$ be optimal for $\IGW(\mu,\rho)$ and $\IGW(\nu,\rho)$, respectively, and invoke the glueing lemma \cite[Lemma 7.6]{villani2003topics} to obtain $\pi\in\cP(\cH^3)$ with $\pi(\cdot,\cdot,\cH)=\pi_1$ and $\pi(\cH,\cdot,\cdot)=\pi_2$. Identifying the IGW distance as an $L^2$ norm, we use Minkowski inequality to deduce:
    \[
    \IGW(\mu,\nu)\leq \left(\int |\llangle x,x'\rrangle-\llangle z,z'\rrangle+\llangle z,z'\rrangle-\llangle y,y'\rrangle|^2 d\pi\otimes\pi\right)^{\frac{1}{2}}\leq \IGW(\mu,\rho)+\IGW(\nu,\rho).
    \]
    Lastly, for the necessary condition for nullification, suppose $\IGW(\mu,\nu)=0$ and let $\pi\in\Pi(\mu,\nu)$ be an optimal coupling. Our argument relies on viewing $\supp(\pi)$ as a correspondence set---see Definition 2.1 in \cite{memoli2011gromov}. For any point $x\in\supp(\mu)$, we show that there is a unique point  $y\in\supp(\nu)$ such that $(x,y)\in\supp(\pi)$, which concludes the proof, as it directly implies that $\pi$ is supported on a graph of a unitary isomorphism. Suppose there is more than one such point, and denote them by $y_1\neq y_2$. Since IGW cost nullifies, $\llangle x,x'\rrangle = \llangle y,y'\rrangle$ for $\pi\otimes\pi$-a.s. $(x,y),(x',y')$, and hence by continuity of distortion function $|\llangle x,x'\rrangle-\llangle y,y'\rrangle|$, we may extend this to all $(x,y),(x',y')\in\supp(\pi)$.
    Consequently, 
    \[|\llangle x,x\rrangle-\llangle y_1,y_1\rrangle|^2 = |\llangle x,x\rrangle-\llangle y_2,y_2\rrangle|^2 = |\llangle x,x\rrangle-\llangle y_1,y_2\rrangle|^2 =0,\]
    and by the Cauchy-Schwarz inequality, we obtain $y_1=y_2$. Hence $\pi$ is supported on the graph of a mapping, denoted by $T$, and $\llangle x,x'\rrangle = \llangle T(x),T(x')\rrangle$. \qed
    
\subsection{Proof of \cref{lem:symmetrization}}\label{appen:lem:symmetrization_proof}

    By definition, every $\bO\in\cO_{\mu,\nu}$ corresponds to an optimal IGW coupling for $(\mu,\nu)$. For $\bO$ and the corresponding coupling $\tilde\pi$, denote its SVD as $\int xy^\intercal d\tilde\pi=\bP\blambda\bQ^\intercal$. Note that $\pi^\star\coloneqq(\id,\bO)_\sharp\tilde\pi$ is an optimal IGW coupling for $(\mu,\bO_\sharp\nu)$, and $\int xy^\intercal d\pi^\star = \int xy^\intercal \bO^\intercal d\tilde\pi = \bP\blambda\bQ^\intercal\bQ\bP^\intercal=\bP\blambda\bP^\intercal$, which is symmetric, and the eigenvalues are the same as the singular values of $\int xy^\intercal d\pi^\star$, hence nonnegative. By proof of \eqref{eq:F2Decomp}, the matrix $\bA^\star= \frac{1}{2}\big(\int xy^\intercal d\pi^\star\big) $ achieves optimality in the dual form in \eqref{eq:F2Decomp}. \qed

\subsection{Proof of \cref{lem:equivalence_igw_w}}\label{appen:lem:equivalence_igw_w_proof}
    The first inequality follows from a straightforward application of Cauchy-Schwarz inequality. For any coupling $\pi\in\Pi(\mu,\nu)$, we have
    \begin{align*}
        &\int \left|\llangle x,x'\rrangle-\llangle y,y'\rrangle\right|^2d\pi\otimes\pi(x,y,x',y')\\
        &= \int \left|\llangle x,x'-y'\rrangle+\llangle x-y,y'\rrangle\right|^2 d\pi\otimes\pi(x,y,x',y')\\
        &\leq \int \left(2\llangle x,x'-y'\rrangle^2+2\llangle x-y,y'\rrangle^2\right) d\pi\otimes\pi(x,y,x',y')\\
        &\leq \int \left(2 \|x\|^2\|x'-y'\|^2+2\|x-y\|^2\|y'\|^2 \right)d\pi\otimes\pi(x,y,x',y')\\
        & = \int \left(2 \|x\|^2+2\|y\|^2\right)d\pi(x,y)\int\|x-y\|^2 d\pi(x,y).
    \end{align*}

\medskip

    We now move to establish the second inequality. By \cref{lem:symmetrization} and the invariance of IGW to orthogonal transformations, it is enough to prove the claim for $\mu,\nu$, such that there is an optimal IGW coupling $\pi^\star$ with $\int xy^\intercal d\pi^\star$ being PSD. The argument largely mirrors the steps from the proof of \cite[Lemma 4.2]{zhang2024gromov}, but is presented in full for completeness. Let $\int xy^\intercal d\pi^\star(x,y)=\bP\blambda^\star\bP^\intercal$ is the diagonalization of the covariance matrix, with eigenvalues $\lambda_1,\ldots,\lambda_d\geq0$. Set $\tilde{\pi} = (\bP^\intercal,\bP^\intercal)\pi^\star$, and denote by $a_1,\ldots,a_d$ and $b_1,\ldots,d_b$ the diagonal entries of $\bP^\intercal\bsigma_{\mu}\bP$ and $\bP^\intercal\bsigma_\nu\bP$, respectively, all of which are positive as $\bsigma_{\mu}$ and $\bsigma_\nu$ are both full rank. Note that the IGW problem can be reformulated and lower bounded as
    \begin{align*}
        \IGW(\mu,\nu)^2 &= \inf_{\pi\in \Pi(\mu,\nu)} \|\bsigma_{\mu}\|_\F^2+\|\bsigma_\nu\|_\F^2-2\left\|\int xy^\intercal d\pi(x,y)\right\|_\F^2\\
        & = \|\bP^\intercal\bsigma_{\mu}\bP\|_\F^2+\|\bP^\intercal\bsigma_\nu\bP\|_\F^2-2\|\blambda^\star\|_\F^2\\
        &\geq \sum_i a_i^2 +b_i^2-2\lambda_i^2.\numberthis\label{eq:IGW_primal_decomp_lower_bound}
    \end{align*}
    Observing that $a_i+b_i-2\lambda_i\geq 0$, as $\int\left( x_i^2+y_i^2-2x_iy_i \right)d \tilde\pi \geq0$, we further have
    \begin{align*}
    \sqrt{\frac{a_i^2+b_i^2}{2}}\geq \frac{a_i+b_i}{2}\geq \lambda_i,\quad\forall i=1,\ldots, d,
    \end{align*}
    which implies
    \begin{align*}
    \IGW(\mu,\nu)^2&\geq  2\sum_{i=1 }^d\left(\sqrt{\frac{a_i^2+b_i^2}{2}}+\lambda_i\right)\left(\sqrt{\frac{a_i^2+b_i^2}{2}}-\lambda_i\right)\\
    &\geq  2\min_{i=1,\ldots,d} \sqrt{\frac{a_i^2+b_i^2}{2}} \cdot \sum_{i=1}^d \left(\sqrt{\frac{a_i^2+b_i^2}{2}}-\lambda_i\right).
    \end{align*} 
    Having that, we compute
    \begin{align*}
    \W_2(\mu,\nu)^2 &= \inf_{\pi\in\Pi(\mu,\nu)} \int \|x- y\|^2 d\pi(x,y)\\
    &\leq \int \| x-y\|^2 d\pi^\star(x,y)\\
    & = \int \| x-y\|^2 d\tilde{\pi}(x,y)\\
    & = \sum_{i=1 }^d a_i + b_i -2\lambda_i\\
    &\leq 2\sum_{i=1 }^d\left(\sqrt{\frac{a_i^2+b_i^2}{2}}-\lambda_i\right)\\
    &\leq \frac{\IGW^2(\mu,\nu)}{\min_i \sqrt{\frac{a_i^2+b_i^2}{2}}}.
    \end{align*}
    Noting that $\lambda_{\mathrm{min}}(\bsigma_\mu)\leq a_i$ and $\lambda_{\mathrm{min}}(\bsigma_\nu)\leq b_i$, for all $i=1,\ldots,d$, concludes the proof.\qed

\subsection{Proof of \cref{prop:IGW_curve_realization}}\label{appen:prop:IGW_curve_realization_proof}
    For simplicity we suppose $L=1$. Fix $\bar{\rho}_n(0)=\rho_0$. We construct a sequence of curves $\bar{\rho}_n$, $n=0,1,\ldots$, as follows. For $t^n_k=k/2^n$, where $k=1,\cdots,2^n$, recursively set $\bar{\rho}_n(t^n_k)\coloneqq\bO_\sharp\rho_{t^n_k}$ for some $\bO\in\cO_{\bar{\rho}_n(t^n_{k-1}),\rho_{t^n_k}}$, so that by \cref{lem:equivalence_igw_w}
    \[\W_2\big(\bar{\rho}_n(t^n_{k-1}),\bar{\rho}_n(t^n_k)\big)\leq \frac{1}{\sqrt{c}} \IGW\big(\bar{\rho}_n(t^n_{k-1}),\bar{\rho}_n(t^n_k)\big).
    \]
    Having that, we set $\bar{\rho}_n(t)$, for all $t\in (t^n_{k-1},t^n_k)$ , as the $\W_2$-displacement interpolation between the edge points $\bar{\rho}_n(t^n_{k-1})$ and $\bar{\rho}_n(t^n_k)$. As the above bound implies $\W_2\big(\bar{\rho}_n(t^n_{k-1}),\bar{\rho}_n(t^n_k)\big)\leq \frac{1}{\sqrt{c}2^n}$, we see that the piecewise geodesic $\bar\rho_n$ is $\frac{1}{\sqrt{c}}$-Lipschitz.
    
    Now we use a version of Arzel\`a–Ascoli theorem for possibly noncompact spaces, which is derived as part of the following argument, for completeness. First note that for any dyadic number $t$, $\{\bar{\rho}_n(t)\}_{n=0}^\infty$ belongs to a compact (w.r.t. $\W_2$) family $\Od_\sharp\rho_t$, up to finitely many $n$ values. Number all dyadic numbers into a sequence $\{s_m\}_{m\in\NN}$. Since for any $s_m$, $\{\bar{\rho}_n(s_m)\}_n$ all belong to $\Od_\sharp\rho_{s_m}$, we may pick a subsequence of $\bar{\rho}_n$ that is convergent at $s_1$, and from this subsequence we draw a further subsequence that converges at $s_2$, and so on. Using the usual diagonal trick, i.e. picking the $m$-th term from the subsequence of $s_m$ as above, we find a subsequence $n_k$ that converges at all dyadic numbers.
    Denote by $\bar{\rho}$ the limit, which is currently defined only over dyadic numbers, and satisfies 
    \[\W_2\big(\bar{\rho}(s_{m_1}), \bar{\rho}(s_{m_2})\big) = \lim_{n_k}\W_2\big(\bar{\rho}_{n_k}(s_{m_1}), \bar{\rho}_{n_k}(s_{m_2})\big)\leq \frac{|s_{m_1}-s_{m_2}|}{\sqrt{c}}.\]
    To extend it to $[0,1]$, for each $t\in[0,1]$ we fix a sequence $\{s^t_m\}_m$ of dyadic numbers with $s^t_m\to t$. Clearly $\bar{\rho}(s^t_m)$ is a $\W_2$-Cauchy sequence, hence we have a $\W_2$ limit, denoted as $\bar{\rho}(t)$, with 
    \[\W_2\big(\bar{\rho}(t_1),\bar{\rho}(t_2)\big) = \lim_{m} \W_2\big(\bar{\rho}(s^{t_1}_m),\bar{\rho}(s^{t_2}_m)\big) \leq \lim_{m} \frac{|s^{t_1}_m-s^{t_2}_m|}{\sqrt{c}} = \frac{|t_1-t_2|}{\sqrt{c}},\quad \forall t_1,t_2\in[0,1].\]
    This implies that $\bar{\rho}$ is $\frac{1}{\sqrt{c}}$-Lipschitz. We next show that $\bar{\rho}_{n_k}$ converges uniformly to $\bar{\rho}$. Fix $\epsilon>0$. For any dyadic number $s$, there is positive integer $K_s$ such that $\W_2\big(\bar{\rho}(s),\bar{\rho}_{n_k}(s)\big)<\epsilon/3$ for all $k\geq K_s$. Since $\bar{\rho}$ and $\bar{\rho}_{n_k}$ are $\frac{1}{\sqrt{c}}$-equicontinuous, for $t\in(s-\epsilon\sqrt{c}/3,s+\epsilon\sqrt{c}/3)\cap[0,1]$, $\W_2\big(\bar{\rho}(t),\bar{\rho}_{n_k}(t)\big)\leq \epsilon$. Clearly, $\{(s-\epsilon\sqrt{c}/3,s+\epsilon\sqrt{c}/3):\,s\mbox{ is dyadic}\}$ is an open cover of $[0,1]$, which has a finite subcover. Since each open interval is associated with an integer $K_s$, we may pick the largest among the finite covers, denoted as $K$, and conclude that $\W_2\big(\bar{\rho}(t),\bar{\rho}_{n_k}(t)\big)\leq \epsilon$ for all $t\in[0,1]$ when $k\geq K$, which yields uniform convergence.
    With this, it only remains show that $\bar{\rho}$ and $\rho$ are equivalent in $\IGW$.

    For a fixed $t\in[0,1]$, denote $\mu=\rho_t$, and let $t_k$ be a sequence of dyadic numbers that converges to $t$. Consider $\mu_k\coloneqq\bar{\rho}_{t_k}$, which is Cauchy in $\W_2$ and converges to $\bar{\rho}_{t}$. Take $\bO_k\in\cO_{\mu_k,\mu}$. Since $\Od$ is compact under the operator norm, we pick a subsequence $t_{k_\ell}$ such that $\bO_{k_\ell}$ converges in the operator norm to a limit, which we denote by $\bar{\bO}$. Now we show that $\bar{\bO}_\sharp\mu$ is limit of this subsequence $\mu_{k_\ell}$ in $\W_2$, which implies that $\bar{\rho}_{t}=\bar{\bO}_\sharp\rho_{t}$. Compute
    \begin{align*}
        \W_2(\bar{\bO}_\sharp\mu,\mu_{k_\ell}) &\leq \W_2\big(\bar{\bO}_\sharp\mu, (\bO_{k_\ell})_\sharp\mu\big)  +\W_2\big((\bO_{k_\ell})_\sharp\mu,\mu_{k_\ell}\big) \\
        &\leq \|\bar{\bO}- \bO_{k_\ell}\|_\op \sqrt{M_2(\mu)}  + \frac{1}{\sqrt{c}}\IGW(\mu,\mu_{k_\ell}) \\
        & = \|\bar{\bO}- \bO_{k_\ell}\|_\op \sqrt{M_2(\mu)}  + \frac{1}{\sqrt{c}}\IGW\big(\rho_t,\bar{\rho}_{t_{k_\ell}}\big)\\
        & = \|\bar{\bO}- \bO_{k_\ell}\|_\op \sqrt{M_2(\mu)}  + \frac{1}{\sqrt{c}}\IGW\big(\rho_t,\rho_{t_{k_\ell}}\big)
    \end{align*}
    which converges to 0. Thus $\bar{\bO}_\sharp \mu = \bar{\rho}_t$, so $\bar{\rho}$ and $\rho$ are equivalent in $\IGW$.\qed

\subsection{Proof of \cref{lem:local_convexity_igw}}\label{appen:lem:local_convexity_igw_proof}
We first glue the optimal couplings for $(\mu_0,\mu_1)$ and $(\mu_0,\mu_2)$ into a 3-coupling $\pi\in\Pi(\mu_0,\mu_1,\mu_2)$.
Now consider the generalized geodesic $\nu_t$ from $\mu_1$ to $\mu_2$ w.r.t. $\mu_0$, we have
    \begin{align*}
        &\IGW(\nu_t,\mu_0)^2 \\
        &\leq\int \left|\llangle (1-t)y+tz,(1-t)y'+tz'\rrangle-\llangle x,x'\rrangle\right|^2d\pi\otimes\pi(x,y,z,x',y',z')\\
        & = \int \left|(1-t)\llangle y,y'\rrangle
        +t\llangle z,z'\rrangle + t(t-1)\llangle y-z,y'-z'\rrangle
         -\llangle x,x'\rrangle \right|^2d\pi\otimes\pi(x,y,z,x',y',z')\\
         & = \int \left|(1-t)\llangle y,y'\rrangle
        +t\llangle z,z'\rrangle-\llangle x,x'\rrangle\right|^2 \\
        & \qquad+ 2t\llangle y-z,y'-z'\rrangle (\llangle x,x'\rrangle
        -\llangle y,y'\rrangle) d\pi\otimes\pi(x,y,z,x',y',z') +O(t^2)\\
        & = \int (1-t)\left|\llangle y,y'\rrangle-\llangle x,x'\rrangle\right|^2
        +t\left|\llangle z,z'\rrangle-\llangle x,x'\rrangle\right|^2 - t(1-t)\left|\llangle y,y'\rrangle-\llangle z,z'\rrangle\right|^2 \\
        &\qquad+ 2t\llangle y-z,y'-z'\rrangle (\llangle x,x'\rrangle
        -\llangle y,y'\rrangle) d\pi\otimes\pi(x,y,z,x',y',z') +O(t^2)\\
        & = (1-t)\IGW(\mu_1,\mu_0)^2 + t\IGW(\mu_2,\mu_0)^2 + \int -t \left|\llangle y,y'\rrangle-\llangle z,z'\rrangle\right|^2 \\
        &\qquad+ 2t\llangle y-z,y'-z'\rrangle (\llangle x,x'\rrangle
        -\llangle y,y'\rrangle) d\pi\otimes\pi(x,y,z,x',y',z') +O(t^2)\\
        & \leq (1-t)\IGW(\mu_1,\mu_0)^2 + t\IGW(\mu_2,\mu_0)^2 - t\int \left|\llangle y,y'\rrangle-\llangle z,z'\rrangle\right|^2d\pi\otimes\pi(x,y,z,x',y',z')  \\
        &\qquad+ 2t\sqrt{\int (\llangle y-z,y'-z'\rrangle)^2d\pi\otimes\pi}\sqrt{\int (\llangle x,x'\rrangle
        -\llangle y,y'\rrangle)^2 d\pi\otimes\pi} +O(t^2)\\
        &\leq (1-t)\IGW(\mu_1,\mu_0)^2 + t\IGW(\mu_2,\mu_0)^2 - t\int \left|\llangle y,y'\rrangle-\llangle z,z'\rrangle\right|^2d\pi\otimes\pi(x,y,z,x',y',z')\\
        &\qquad+ 2t \IGW(\mu_0,\mu_1)\int \|y-z\|^2 d\pi(x,y,z) + O(t^2).
    \end{align*}
    Now we expand $\int \|y-z\|^2 d\pi(x,y,z)\leq \int 2\|x-y\|^2 + 2\|x-z\|^2 d\pi(x,y,z)$. Note that by the same argument from proof of \cref{lem:equivalence_igw_w}, since $\int xy^\intercal d\pi(x,y,z)$, $\int xz^\intercal d\pi(x,y,z)$ are PSD and $\lambda_{\min}(\bsigma_{\mu_0})\geq c$, we have $\frac{c}{\sqrt{2}}\int \|x-y\|^2 d\pi(x,y,z) \leq \int \left|\llangle x,x'\rrangle-\llangle y,y'\rrangle\right|^2d\pi\otimes\pi(x,y,z,x',y',z')$, as well as $\frac{c}{\sqrt{2}}\int \|x-z\|^2 d\pi(x,y,z) \leq \int \left|\llangle x,x'\rrangle-\llangle z,z'\rrangle\right|^2 d\pi\otimes\pi(x,y,z,x',y',z')$. Plugging back we have 
    \begin{align*}
        &\IGW(\nu_t,\mu_0)^2 \\
        &\leq (1-t)\IGW(\mu_1,\mu_0)^2 + t\IGW(\mu_2,\mu_0)^2 - t\int \left|\llangle y,y'\rrangle-\llangle z,z'\rrangle\right|^2d\pi\otimes\pi(x,y,z,x',y',z')\\
        &\qquad+ \frac{4\sqrt{2}}{c}t \IGW(\mu_0,\mu_1)\int   \left|\llangle x,x'\rrangle-\llangle z,z'\rrangle\right|^2 + \left|\llangle x,x'\rrangle-\llangle y,y'\rrangle\right|^2 d\pi\otimes\pi(x,y,z,x',y',z')\\
        &\qquad+ O(t^2)\\
        &\leq (1-t)\IGW(\mu_1,\mu_0)^2 + t\IGW(\mu_2,\mu_0)^2 - t\int \left|\llangle y,y'\rrangle-\llangle z,z'\rrangle\right|^2d\pi\otimes\pi(x,y,z,x',y',z')\\
        &\qquad+ \frac{4\sqrt{2}}{c}t \IGW(\mu_0,\mu_1)\left(\int   2\left|\llangle y,y'\rrangle-\llangle z,z'\rrangle\right|^2 d\pi\otimes\pi(x,y,z,x',y',z') + 3\IGW(\mu_0,\mu_1)^2\right)\\
        &\qquad+ O(t^2)\\
        &\leq (1-t)\IGW(\mu_1,\mu_0)^2 + t\IGW(\mu_2,\mu_0)^2 \\
        &\qquad- t\left(1- \frac{8\sqrt{2}}{c}\IGW(\mu_0,\mu_1)\right) \int \left|\llangle y,y'\rrangle-\llangle z,z'\rrangle\right|^2d\pi\otimes\pi(x,y,z,x',y',z')\\
        &\qquad + \frac{12\sqrt{2}}{c}t\IGW(\mu_0,\mu_1)^3 + O(t^2)\\
        &\leq (1-t)\IGW(\mu_1,\mu_0)^2 + t\IGW(\mu_2,\mu_0)^2 - t\left(1- \frac{8\sqrt{2}}{c}\IGW(\mu_0,\mu_1)\right) \IGW(\mu_1,\mu_2)^2 \\
        &\qquad + \frac{12\sqrt{2}}{c}t\IGW(\mu_0,\mu_1)^3 + O(t^2)\numberthis\label{eq:distance_on_generalized_geodesic_intermediate_step}
    \end{align*}
    where we have used that $\IGW(\mu_0,\mu_1)\leq \frac{c}{16\sqrt{2}}$ such that $1- \frac{8\sqrt{2}}{c}\IGW(\mu_0,\mu_1)\geq1/2>0$.\qed

\section{Auxiliary Proofs for \cref{sec:IGW_gradient_flow}}\label{appen:IGW_gradientflow_aux_proofs}

\subsection{Local convexity of the entropy functional from \cref{rem:func_examples}}\label{appen:entropy}

Convexity (i.e., with $\lambda=0$) of the potential and interaction energy functionals $\sV$ and $\sW$ is straightforward whenever the functions $V$ and $W$, respectively, are convex; see \cite[Example 9.3.1, 9.3.4]{ambrosio2005gradient}. We next present the proof of the generalized geodesic convexity (again, with $\lambda=0$) for the entropy functional
\begin{align*}
    \sH(\mu)\coloneqq 
    \begin{cases}
        \int_{\RR^d} \frac{d\mu}{dx} \log\left(\frac{d\mu}{dx}\right) dx &, \mathrm{ if }\mu\in\cP_2^{\mathrm{ac}}(\RR^d)\\
        +\infty&, \mathrm{ otherwise}
    \end{cases}.
\end{align*}
However, we can only establish this convexity with a slightly modified definition of generalized geodesics, as described below. Namely, the modified geodesics have a multiplicative structure that lends itself better for analysis under the entropy functional. Throughout this derivation, we slightly abuse notation and denote a probability measure and its Lebesgue density by the same symbol, e.g., identifying $\mu\in\cP_2^{\mathrm{ac}}(\RR^d)$ with $\frac{d\mu}{dx}$.

Recall the definition of generalized geodesic $\nu_t$ of $\mu_1,\mu_2$ w.r.t. $\mu_0$ (\cref{def:gen_geo_IGW}), where we instantiate $\bA_1\coloneqq\frac{1}{2}\int xy^\intercal d\pi_{01}(x,y)$ and $\bA_2\coloneqq\frac{1}{2}\int xz^\intercal d\pi_{02}(x,z)$ as PSD and nonsingular. Supposing that $\mu_0,\mu_1,\mu_2\in\dom(\sH)\subset\cP_2^{\mathrm{ac}}(\RR^d)$, we modify the generalized geodesics as follows. By Lemmas \ref{thm:brenier} and \ref{lem:gromov_monge}, there exist two Gromov-Monge maps $T_i\coloneqq (8\bA_i)^{-1}\nabla\varphi_i$ from $\mu_0$ to $\mu_i$, $i=1,2$, where $\varphi_1,\varphi_2$ are convex functions. Define 
\[T_t\coloneqq \big((1-t)(8\bA_1)^{-1} + t (8\bA_2)^{-1}\big)\big((1-t)\nabla\varphi_1 + t\nabla\varphi_2\big),\]
and consider the new generalized geodesic $\nu_t\coloneqq (T_t)_\sharp \mu_0$. Note that by \cite[Proposition 6.2.12]{ambrosio2005gradient} $(1-t)\nabla\varphi_1 + t \nabla\varphi_2$ is $\mu$-essentially injective, and therefore, so it $T_t$. Following the same idea as in \cite[Proposition 9.3.9]{ambrosio2005gradient}, we note that $\nabla \varphi_i$, for $i=0,1$, is (approximately) differentiable in the sense of \cite[Definition 5.5.1]{ambrosio2005gradient}, and by \cite[ Lemma 5.5.3]{ambrosio2005gradient}, the Hessian matrices $\bH\varphi_i \coloneqq (\partial_{x_j x_k}\varphi_i)_{j,k=1}^d$ exist and are nonsingular for $\mu_0$-a.e. $x\in\RR^d$. Thus we invoke \cite[Lemma 5.5.3]{ambrosio2005gradient} and compute the entropy along $\nu_t$ as
\begin{align*}
    \sH(\nu_t) = \int \mu_0(x) \log\frac{\mu_0(x)}{\det\big(\nabla T_t(x)\big)} dx,
\end{align*}
where we have used positivity of determinants of $\bA_1,\bA_2$ and the Hessians $\bH\varphi_1,\bH\varphi_2$. The convexity over $t$ now follows from the concavity of 
\[\log\det\big(\nabla T_t(x)\big) = \log\det\big((1-t)(8\bA_1)^{-1} + t (8\bA_2)^{-1}\big)+ \log\det\big((1-t)\bH\varphi_1 + t \bH\varphi_2\big),\]
since $(8\bA_1)^{-1}, (8\bA_2)^{-1},\bH\varphi_1,\bH\varphi_2$ are all PSD; c.f. \cite[Proposition 9.3.9]{ambrosio2005gradient}. We obtain
\begin{align*}
    \sH(\nu_t) \leq (1-t)\sH(\mu_1) + t \sH(\mu_2),
\end{align*}
which establishes convexity of $\sH$ along the modified generalized geodesics.

\medskip

As we have modified the definition of generalized geodesics, a priori, it is unclear whether our gradient flow theory (which is derived under the notion from \cref{def:gen_geo_IGW}) still holds under the new definition. We next show that this is indeed the case by rederiving the key intermediate results from the proof of \cref{thm:main}, specifically, Lemmas \ref{lem:local_convexity_igw}  \ref{lem:convexity_of_proximal_functional}, under the new definition.

We start from the local convexity of IGW from \cref{lem:local_convexity_igw}, under the additional assumption that there exists $c_1>0$ with $\lambda_{\min(\bsigma_{\mu_0})}\geq c_1$, and $c_2>0$ with $1/c_2\leq \lambda_{\min}(\bA_i)\leq \lambda_{\max}(\bA_i)\leq c_2$. We write $\pi\in\Pi(\mu_1,\mu_1,\mu_2)$ for the joint distribution obtained by gluing $\pi_{01}$ and $\pi_{02}$, and note that $T_t = \big((1-t)\bA_1^{-1} + t \bA_2^{-1}\big)\big((1-t)\bA_1 T_1 + t\bA_2 T_2\big)$. Thus we may expand 
\[\big((1-t)\bA_1^{-1} + t \bA_2^{-1}\big)\big((1-t)\bA_1 y+t\bA_2 z\big) = y + t(\bA_1^{-1}\bA_2 z + \bA_2^{-1}\bA_1 y - 2y) +O(t^2),\]
and for simplicity denote $\bX\coloneqq \bA_1^{-1}\bA_2$. Using the shorthand $d\pi\otimes\pi$ for $d\pi\otimes\pi(x,y,z,x',y',z')$ under the integral sign below, consider
\begin{align*}
    &\IGW(\nu_t,\mu_0)^2 \\
        &\leq\int\mspace{-2mu} \big|\mspace{-2mu}\llangle \big((1\mspace{-1mu}-\mspace{-1mu}t)\bA_1^{-1} \mspace{-1mu}+\mspace{-1mu} t \bA_2^{-1}\big)\big((1\mspace{-1mu}-\mspace{-1mu}t)\bA_1 y\mspace{-1mu}+\mspace{-1mu}t\bA_2 z\big),\big((1\mspace{-1mu}-\mspace{-1mu}t)\bA_1^{-1} \mspace{-1mu}+\mspace{-1mu} t \bA_2^{-1}\big)\big((1\mspace{-1mu}-\mspace{-1mu}t)\bA_1 y'\mspace{-1mu}+\mspace{-1mu}t\bA_2 z'\big)\rrangle\\
        &\qquad\qquad\qquad\qquad\qquad\qquad\qquad\qquad\qquad\qquad\qquad\qquad\qquad\qquad\qquad\qquad-\llangle x,x'\rrangle\big|^2d\pi\otimes\pi\\
        & = \int \big|\llangle y + t(\bX z + \bX^{-1} y - 2y) , y' + t(\bX z' + \bX^{-1} y' - 2y')\rrangle-\llangle x,x'\rrangle\big|^2d\pi\otimes\pi+O(t^2)\\
        & = \int \big|\llangle y  , y' \rrangle+ t \left(\llangle \bX z + \bX^{-1} y - 2y, y' \rrangle+\llangle y, \bX z' + \bX^{-1} y' - 2y' \rrangle\right) -\llangle x,x'\rrangle\big|^2d\pi\otimes\pi\mspace{-4mu}+\mspace{-4mu}O(t^2)\\
        & = \int \Big(\big|\llangle y  , y' \rrangle-\llangle x,x'\rrangle\big|^2+ 2t \left(\llangle y  , y' \rrangle-\llangle x,x'\rrangle\right)\big(\llangle \bX z + \bX^{-1} y - 2y, y' \rrangle\\
        &\qquad\qquad\qquad\qquad\qquad\qquad\qquad\qquad\qquad\qquad+\llangle y, \bX z' + \bX^{-1} y' - 2y' \rrangle\Big)  d\pi\otimes\pi+O(t^2)\\
        & = \int \Big(\big|\llangle y  , y' \rrangle-\llangle x,x'\rrangle\big|^2+ 2t (\llangle y  , y' \rrangle-\llangle x,x'\rrangle)\left(\llangle z-y, y' \rrangle + \llangle y, z'-y' \rrangle\right)\Big)  d\pi\otimes\pi\\
        &\quad  +  \int 2t \big(\llangle y  , y' \rrangle-\llangle x,x'\rrangle\big)\big(\llangle \bX y + \bX^{-1} y - 2y, y' \rrangle+\llangle y, \bX y' + \bX^{-1} y' - 2y' \rrangle\big)d\pi\otimes\pi\\
        &\qquad - \int 2t \big(\llangle y  , y' \rrangle-\llangle x,x'\rrangle\big)\big(\llangle (\bI-\bX)(z-y), y' \rrangle+ \llangle y, (\bI-\bX)(z'-y') \rrangle\big) d\pi\otimes\pi+O(t^2).\numberthis\label{eq:entropy_rederive1}
\end{align*}
From the proof of \cref{lem:local_convexity_igw}, we have that the first line in the last expression is bounded by 
\begin{align*}
    (1-t)\IGW(\mu_1,\mu_0)^2 + t\IGW(\mu_2,\mu_0)^2 - t&\left(1- \frac{8\sqrt{2}}{c_1}\IGW(\mu_0,\mu_1)\right) \int \left|\llangle y,y'\rrangle-\llangle z,z'\rrangle\right|^2d\pi\otimes\pi\\
    &\qquad\qquad\qquad+ \frac{12\sqrt{2}}{c_1}t\IGW(\mu_0,\mu_1)^3 + O(t^2),
\end{align*}
for $\lambda_{\min(\bsigma_{\mu_0})}\geq c_1>0$. We now bound the next two terms. Noting that $\bX  + \bX^{-1}  - 2\bI = \bX^{-1}(\bX-\bI)^2$, for the second line we obtain
\begin{align*}
        \int \big(\llangle y  , y' \rrangle&-\llangle x,x'\rrangle\big)\left(\llangle \bX y + \bX^{-1} y - 2y, y' \rrangle+\llangle y, \bX y' + \bX^{-1} y' - 2y' \rrangle\right)d\pi\otimes\pi\\
        & = \int \big(\llangle y  , y' \rrangle-\llangle x,x'\rrangle\big)\left(\llangle \bX^{-1}(\bX-\bI)^2 y, y' \rrangle+\llangle y, \bX^{-1}(\bX-\bI)^2 y' \rrangle\right)d\pi\otimes\pi\\
        & \leq 2\IGW(\mu_1,\mu_0)\left\|\bX^{-1} (\bX-\bI)^2\right\|_\op M_2(\mu_1)\\
        & \leq 2\IGW(\mu_1,\mu_0)\left\|\bX^{-1} \right\|_\op\left\|\bX-\bI\right\|_\op^2 M_2(\mu_1),
\end{align*}
where the last step uses sub-multiplicativity of the operator norm. For the third line in \eqref{eq:entropy_rederive1}, we similarly have
\begin{align*}
    &\int (\llangle y  , y' \rrangle-\llangle x,x'\rrangle)\left(\llangle (\bI-\bX)(z-y), y' \rrangle + \llangle y, (\bI-\bX)(z'-y') \rrangle\right) d\pi\otimes\pi(x,y,z,x',y',z')\\
    &\leq 2\sqrt{M_2(\mu_1)}\IGW(\mu_1,\mu_0)\|\bI-\bX\|_\op \sqrt{\int \|z-y\|^2 d\pi(x,y,z)}.\numberthis\label{eq:rederive1}
\end{align*}
To control the right-hand sides of the two display equations above, first note that $\|\bX^{-1}\|_\op,\|\bX\|_\op$ are both upper bounded so long as there is a $c_2>0$ with $1/c_2\leq \lambda_{\min}(\bA_i)\leq \lambda_{\max}(\bA_i)\leq c_2$, whence \[\|\bA_1\|_\op\vee\|\bA_1^{-1}\|_\op\vee\|\bA_2\|_\op\vee\|\bA_2^{-1}\|_\op\leq c_2 \quad \mathrm{and}\quad  \|\bX^{-1}\|_\op\vee\|\bX\|_\op\leq c_2^2.\]
Then, write $\bX-\bI = \bA_1^{-1}(\bA_2-\bA_1) = \frac{1}{2}\bA_1^{-1}\int x(z-y)^\intercal d\pi(x,y,z)$ and bound
\begin{equation}
\|\bX-\bI\|_\op\leq \frac{c_2}{2}\left\|\int x(z-y)^\intercal d\pi(x,y,z)\right\|_\op \leq \frac{c_2}{2}\sqrt{M_2(\mu_0)\int \|z-y\|^2 d\pi(x,y,z)},\label{eq:rederive2}
\end{equation}
using the sub-multiplicative property and Cauchy-Schwarz inequality. Lastly, note that 
\begin{align*}
    \int \|z-y\|^2 d\pi(x,y,z)\leq \frac{2\sqrt{2}}{c_1} \left( 2\int \left|\llangle y,y'\rrangle-\llangle z,z'\rrangle\right|^2d\pi\otimes\pi + 3\IGW(\mu_1,\mu_0)^2 \right),\numberthis\label{eq:rederive3}
\end{align*}
for $\lambda_{\min(\bsigma_{\mu_0})}\geq c_1>0$, where the last inequality follows similarly to step (a) in \eqref{eq:distance_on_generalized_geodesic_intermediate_step} from the proof of \cref{lem:local_convexity_igw}.
The latter is inserted to the right-hand sides of \eqref{eq:rederive1} and \eqref{eq:rederive2}.

Combining above, we obtain 
\begin{align*}
    &\IGW(\nu_t,\mu_0)^2 \\
    &\leq (1-t)\IGW(\mu_1,\mu_0)^2 \mspace{-4mu}+\mspace{-4mu} t\IGW(\mu_2,\mu_0)^2\mspace{-4mu} -\mspace{-4mu} t\mspace{-2mu}\left(\mspace{-2mu}1- \frac{8\sqrt{2}}{c_1}\IGW(\mu_0,\mu_1)\mspace{-2mu}\right)\mspace{-5mu} \int \mspace{-5mu}\left|\llangle y,y'\rrangle-\llangle z,z'\rrangle\right|^2\mspace{-2mu}d\pi\otimes\pi \\
    & +\mspace{-3mu} 4t \IGW(\mu_1,\mu_0)\Bigg(\mspace{-2mu} M_2(\mu_1)\|\bX^{-1}\|_\op \|\bX\mspace{-2mu}-\mspace{-2mu}\bI\|_\op^2\mspace{-3mu}+\mspace{-3mu}   \sqrt{M_2(\mu_1)} \|\bI\mspace{-2mu}-\mspace{-2mu}\bX\|_\op \sqrt{\mspace{-4mu}\int \|z\mspace{-2mu}-\mspace{-2mu}y\|^2 d\pi(x,y,z)}\Bigg)\\
    &+ t\frac{12\sqrt{2}}{c_1}\IGW(\mu_0,\mu_1)^3 + O(t^2)\\
    &\stackrel{\mathrm{(a)}}{\leq} (1-t)\IGW(\mu_1,\mu_0)^2 \mspace{-4mu}+\mspace{-4mu} t\IGW(\mu_2,\mu_0)^2\mspace{-4mu} -\mspace{-4mu} t\mspace{-2mu}\left(\mspace{-2mu}1- \frac{8\sqrt{2}}{c_1}\IGW(\mu_0,\mu_1)\mspace{-2mu}\right)\mspace{-5mu} \int \mspace{-5mu}\left|\llangle y,y'\rrangle-\llangle z,z'\rrangle\right|^2\mspace{-2mu}d\pi\otimes\pi \\
    & + 4t \IGW(\mu_1,\mu_0)\left(\frac{c_2^4}{4} M_2(\mu_1)M_2(\mu_0)+\frac{c_2}{2} M_2(\mu_1) \right)\int \|z-y\|^2 d\pi(x,y,z)\\
    &+ t\frac{12\sqrt{2}}{c_1}\IGW(\mu_0,\mu_1)^3 + O(t^2)\\
    &\stackrel{\mathrm{(b)}}{\leq} (1-t)\IGW(\mu_1,\mu_0)^2 \mspace{-4mu}+\mspace{-4mu} t\IGW(\mu_2,\mu_0)^2\mspace{-4mu} -\mspace{-4mu} t\mspace{-2mu}\left(\mspace{-2mu}1- \frac{8\sqrt{2}}{c_1}\IGW(\mu_0,\mu_1)\mspace{-2mu}\right)\mspace{-5mu} \int \mspace{-5mu}\left|\llangle y,y'\rrangle-\llangle z,z'\rrangle\right|^2\mspace{-2mu}d\pi\otimes\pi \\
    &+ t\frac{8\sqrt{2}}{c_1} \IGW(\mu_1,\mu_0)\left(\frac{c_2^4}{4} M_2(\mu_1)M_2(\mu_0)+\frac{c_2}{2} M_2(\mu_1) \right) \\
    &\qquad\qquad\qquad\qquad\qquad\qquad\qquad\qquad\times\Bigg( 2\int \left|\llangle y,y'\rrangle-\llangle z,z'\rrangle\right|^2d\pi\otimes\pi + 3\IGW(\mu_1,\mu_0)^2 \Bigg)\\
    &+ t\frac{12\sqrt{2}}{c_1}\IGW(\mu_0,\mu_1)^3 + O(t^2)\\
    &\stackrel{\mathrm{(c)}}{=} (1-t)\IGW(\mu_1,\mu_0)^2 + t\IGW(\mu_2,\mu_0)^2  \\
    &- t\left(1- \frac{8\sqrt{2}}{c_1}\IGW(\mu_0,\mu_1) \left(1+ \frac{c_2^4}{2} M_2(\mu_1)M_2(\mu_0)+ c_2 M_2(\mu_1) \right) \right) \IGW(\mu_1,\mu_2)^2\\
    &+  t\frac{12\sqrt{2}}{c_1}\left(1+ \frac{c_2^4}{2} M_2(\mu_1)M_2(\mu_0)+ c_2 M_2(\mu_1) \right) \IGW(\mu_0,\mu_1)^3  + O(t^2),\numberthis\label{eq:entropy_modified_generalized_geodesic}
\end{align*}
where for step (a) we used \eqref{eq:rederive1} \eqref{eq:rederive2} and merge into a single term of $\int\|z-y\|^2d\pi(x,y,z)$, and for step (b) we used \eqref{eq:rederive3}. Also for step (c) we further require $\tau$ to be small enough so that $1- \frac{8\sqrt{2}}{c_1}\IGW(\mu_0,\mu_1) \left(1+ \frac{c_2^4}{2} M_2(\mu_1)M_2(\mu_0)+ c_2 M_2(\mu_1) \right)>0$, which, as we will see later in the application to variational inequality, only depends on $\lambda_{\min}(\bsigma_{\rho_0})$, $M_2(\rho_0)$, $\sH(\rho_0),\sH^\star$. This recovers the generalized geodesic convexity of IGW fro \cref{lem:local_convexity_igw} w.r.t. the modified definition of geodesics, with slightly different coefficients. 

\medskip

We proceed to derive the variational inequality from \cref{lem:convexity_of_proximal_functional}, recalling that $\lambda=0$ for the entropy functional. Under the same conditions as \cref{lem:convexity_of_proximal_functional}, we now have
\begin{align*}
&\sH(\mu_1)+\frac{\IGW(\mu_1,\mu_0)^2}{2\tau}\\
        &\leq \sH(\nu_t)+\frac{\IGW(\nu_t,\mu_0)^2}{2\tau}\\
        &\leq (1-t)\sH(\mu_1)+t\sH(\mu_2) + (1-t)\frac{\IGW(\mu_1,\mu_0)^2}{2\tau} + t\frac{\IGW(\mu_2,\mu_0)^2}{2\tau} \\
        &\qquad - t\frac{1}{2\tau}\left(1- \frac{8\sqrt{2}}{c_1}\IGW(\mu_0,\mu_1) \left(1+ \frac{c_2^4}{2} M_2(\mu_1)M_2(\mu_0)+ c_2 M_2(\mu_1) \right) \right) \IGW(\mu_1,\mu_2)^2\\
        &\qquad\qquad + t\frac{1}{2\tau}\frac{12\sqrt{2}}{c_1}\left(1+ \frac{c_2^4}{2} M_2(\mu_1)M_2(\mu_0)+ c_2 M_2(\mu_1) \right) \IGW(\mu_0,\mu_1)^3  + O(t^2).
\end{align*}
Cancelling the same terms on both sides and letting $t\to 0$, we conclude that
\begin{align*}
        &\sH(\mu_1)+\frac{\IGW(\mu_1,\mu_0)^2}{2\tau}\\
        &\leq  \sH(\mu_2) + \frac{\IGW(\mu_2,\mu_0)^2}{2\tau} \\
        &\qquad - \frac{1}{2\tau}\left(1- \frac{8\sqrt{2}}{c_1}\IGW(\mu_0,\mu_1) \left(1+ \frac{c_2^4}{2} M_2(\mu_1)M_2(\mu_0)+ c_2 M_2(\mu_1) \right) \right) \IGW(\mu_1,\mu_2)^2\\
        &\qquad\qquad + \frac{1}{2\tau}\frac{12\sqrt{2}}{c_1}\left(1+ \frac{c_2^4}{2} M_2(\mu_1)M_2(\mu_0)+ c_2 M_2(\mu_1) \right) \IGW(\mu_0,\mu_1)^3 .
\end{align*}
This is in parallel to \cref{lem:convexity_of_proximal_functional}, where we note that the last two terms now have different coefficients, though the order of each term in terms of $\tau$ remains the same. By \cref{prop:existence_of_scheme} and \cref{lem:bar_delta}, we we can pick $c_1=\frac{\lambda_{\min}(\bsigma_{\rho_0})}{2}$ and $c_2>0$ that only depends on $M_2(\rho_0)$ and $\lambda_{\min}(\bsigma_{\rho_0})$. Now take $t\in((i-1)\tau,i\tau]$, $(\mu_0,\mu_1,\mu_2)=(\rho_{i-1},\rho_{i},\nu)$, where $\nu\in\cB_{\IGW}(\rho_0,\bar{\delta})$ is arbitrary, and the existence of generalized geodesic is guaranteed by choice of $\bar{\delta}$ from \cref{lem:bar_delta}. Using the notation from \cref{prop:discrete_solution_comparison}, with the only exception being a revised definition of
\begin{align*}
    \bar{\sigma}_\tau(t)\coloneqq - \frac{1}{2\tau} \frac{8\sqrt{2}}{c_1}\sD_\tau(t) \left(1+ \frac{c_2^4}{2} M_2(\rho_i)M_2(\rho_{i-1})+ c_2 M_2(\rho_i) \right),
\end{align*}
we have
\begin{align*}
    \sH(\rho_i)+\frac{\IGW(\rho_i,\rho_{i-1})^2}{2\tau} &\leq  \sH(\nu) + \frac{\IGW(\nu,\rho_{i-1})^2}{2\tau} - \left(\frac{1}{2\tau} + \bar{\sigma}_\tau(t) \right) \IGW(\rho_i,\nu)^2\\
        &\qquad + \frac{1}{2\tau}\frac{12\sqrt{2}}{c_1}\left(1+ \frac{c_2^4}{2} M_2(\rho_i)M_2(\rho_{i-1})+ c_2 M_2(\rho_i) \right) \sD_\tau(t)^3.
\end{align*}
For simplicity we further let $\tau$ to be small enough such that $$\frac{12\sqrt{2}}{c_1}\left(1+ \frac{c_2^4}{2} M_2(\rho_i)M_2(\rho_{i-1})+ c_2 M_2(\rho_i) \right) \sD_\tau(t)\leq \frac{1}{2},$$ which only depends on $\lambda_{\min}(\bsigma_{\rho_0}), M_2(\rho_0), \sH(\rho_0),\sH^\star$. Note that as we don't have a global lower bound for $\sH$, here we use instead a lower bound of $\sH(\mu)$ in the designated IGW ball $\cB_{\IGW}(\rho_0,\bar{\delta})$ from \cref{lem:bar_delta}.

Rearranging we have
\begin{align*}
    \frac{1}{2} \frac{d}{dt} d_\tau(t;\nu) +  \bar{\sigma}_\tau(t) \IGW(\bar{\rho}_n(t),\nu)^2 &\leq \sH(\nu) - \sH_\tau(t) +  \sR_\tau(t) - \frac{\sD_\tau(t)^2}{4\tau},
\end{align*}
which is essentially \eqref{eq:variational_differential_inequality}. To invoke Gr\"onwall type bound in \cref{lem:gronwall}, we note that 
\begin{align*}
    \|\bar{\sigma}_\tau(t)\|\lesssim_{\lambda_{\min}(\bsigma_{\rho_0}), M_2(\rho_0)} \frac{\sD_\tau(t)}{\tau},
\end{align*}
which enables the application of integral bounds in proof of \cref{prop:discrete_solution_comparison}, which lead to the same result $d_{\tau\eta}(t,t)\lesssim_{\sH(\rho_0),\sH^\star,\lambda_{\min}(\bsigma_{\rho_0}), M_2(\rho_0)} \sqrt{\tau}+\sqrt{\eta}$.

\subsection{Derivations for \cref{rem:property_of_operator}}\label{appen:rem:property_of_operator}

We solve equation $\cL_{\bA,\mu}[v]=w$, and characterize its inverse $\cL^{-1}_{\bA,\mu}$, which appears in the IGW gradient flow PIDE from \cref{thm:main}. 
Leveraging symmetry, we derive the inverse over the invariant space by solving the Sylvester equation, as shown below. We note that the general solution of the inverse over $L^2(\mu;\RR^d)$ would require a detailed study of the T-Sylvester equation \cite{ikramov2012matrix}, which we leave for future work.

\begin{proposition}[Inverse operator]\label{prop:solution_of_integral_system}
    Let $\bA\in\RR^{d\times d}$ be nonsingular PSD, $\mu\in\cP(\RR^d)$ have a nonsingular covariance $\bsigma_\mu$, $v\in\cI_\mu\coloneqq\{v\in L^2(\mu;\RR^d): \int xv(x)^\intercal d\mu(x) \text{ is symmetric}\}$, and $w\in L^2(\mu;\RR^d)$. Then the integral system 
    \begin{align*}
        \cL_{\bA,\mu}[v](x) = w(x),\quad x\in\RR^d,
    \end{align*}
    has a unique solution $v\in L^2(\mu;\RR^d)$, given by
    \begin{align*}
        v(x)=\cL_{\bA,\mu}^{-1}[w](x)
        = \frac{1}{2}\bA^{-1}w(x) - \frac{1}{2}(x^\intercal\otimes \bI) (\bI\otimes\bA^2+\bsigma_\mu\otimes \bA)^{-1} \int (y\otimes \bI) w(y) d \mu(y),
    \end{align*}
    where $\bI\in\RR^{d\times d}$ is the identity matrix and $\otimes$ denotes the Kronecker product between matrices. If $\bA=\bsigma_\mu$, then $\cL_{\bA,\mu}^{-1}[\cI_\mu]\subset\cI_\mu$.
\end{proposition}
\begin{proof}
    Recall that $\cL_{\bA,\mu}[v](x) = 2\bA v(x) + 2\int  y v(y)^\intercal x  d \mu(y) =  2\bA v(x) + 2\int v(y) y^\intercal x  d \mu(y)$ by assumption.
    Clearly if the solution exists, it has to have form $v(x)=\frac{1}{2}\bA^{-1}w(x)-\bA^{-1}\bB x$, where $\bB\coloneqq \int v(y) y^\intercal d \mu(y)$. Inserting the expression for $v$ into $\bB$, we obtain
    \[\bB = \int \left(\frac{1}{2}\bA^{-1}w(y)-\bA^{-1}\bB y\right) y^\intercal d \mu(y),\]
    from which it follows that $\bA\bB +\bB\bsigma_\mu  =  \frac{1}{2}\int  w(y)y^\intercal d \mu(y)$. The latter is known as the Sylvester equation, whose \emph{unique} and symmetric solution is given in terms of Kronecker products as
    \[
        (\bI\otimes\bA+\bsigma_\mu^\intercal\otimes \bI)\vectorize(\bB) = \frac{1}{2}\vectorize\left(\int  w(y)y^\intercal d \mu(y)\right),\]
    where $\vectorize(\bM)$, for a matrix $\bM\in\RR^{d\times d}$, is the vector of length $d^2$ composed by listing the elements of $\bM$ in column-major order. Solving the above, we obtain 
    \[\vectorize(\bB) = \frac{1}{2}(\bI\otimes\bA+\bsigma_\mu\otimes \bI)^{-1}\int (y\otimes \bI) w(y) d \mu(y),\]
    where we have used the nonsingularity of $\bA,\bsigma_\mu$, as well as the symmetry of the latter. Inserting this back into our expression for $v$, we have
    \begin{align*}
        v(x) & =\frac{1}{2}\bA^{-1}w(x)-\bA^{-1}\bB x\\
        & = \frac{1}{2}\bA^{-1}w(x) - x^\intercal \otimes \bI(\bI\otimes\bA^{-1})\vectorize(\bB)\\
        & = \frac{1}{2}\bA^{-1}w(x) - \frac{1}{2}x^\intercal\otimes \bI(\bI\otimes\bA^{-1})(\bI\otimes\bA+\bsigma_\mu\otimes \bI)^{-1} \int (y\otimes \bI) w(y) d \mu(y)\\
        & = \frac{1}{2}\bA^{-1}w(x) - \frac{1}{2}x^\intercal\otimes \bI (\bI\otimes\bA^2+\bsigma_\mu\otimes \bA)^{-1} \int (y\otimes \bI) w(y) d \mu(y),
    \end{align*}
    where the last step uses the inverse and the mixed-product properties of Kronecker products. By the construction of $\bB$, we conclude the existence and uniqueness of the solution $v$. Lastly, notice that if $\bsigma_\mu\bB +\bB\bsigma_\mu  =  \frac{1}{2}\int  w(y)y^\intercal d \mu(y)$, then $\bsigma_\mu(\bB-\bB^\intercal) + (\bB-\bB^\intercal)\bsigma_\mu  =  0$, and by uniqueness of the solution to Sylvester equation, $\bB$ has to be symmetric, i.e. $\cL_{\bsigma_\mu,\mu}^{-1}[\cI_\mu]\subset\cI_\mu$.
\end{proof}

Next, we address the spectrum of the mobility operator $\cL_{\bsigma_\mu,\mu}$, which was commented on in Item (5) of \cref{rem:property_of_operator}. Considering the spectrum over the whole space $L^2(\mu;\RR^d)$ would, again, require studying the T-Sylvester equation, but matters are much simpler when restricting to the invariant space $\cI_{\mu}$. Writing $\spec(\cL)$ for the spectrum of an operator $\cL$, we have the following proposition.

\begin{proposition}[Spectrum]
Suppose that $\bsigma_\mu$ is nonsingular with eigenvalues $\Lambda_\mu=\{\lambda_1,\ldots,\lambda_d\}$, and set $\Gamma_\mu \coloneqq \Lambda_\mu+\Lambda_\mu=\{\lambda_i+\lambda_j\}_{i,j=1}^d$. We have $\spec\big(\cL_{\bsigma_\mu,\mu}\big|_{\cI_\mu}\big) = 2(\Lambda_\mu \cup\Gamma_\mu)$. Furthermore, for each $\lambda\in\spec\big(\cL_{\bsigma_\mu,\mu}\big|_{\cI_\mu}\big)$, each nontrivial solution to $(\cL_{\bsigma_\mu,\mu}-\lambda)[v]=0$ in $\cI_\mu$ belongs to one of the following cases, all of which are necessary and sufficient:
\begin{enumerate}[leftmargin=*]
    \item if $\lambda/2=\lambda_i\in \Lambda_\mu\setminus \Gamma_\mu$, then $v$ takes values in the eigenspace of $\bsigma_\mu$ corresponding to $\lambda_i$ and $\int v(x)x^\intercal d\mu(x)=0$;
    \item if $\lambda/2\in \Gamma_\mu\setminus \Lambda_\mu$, then $v=-(\bsigma_\mu-\lambda/2)^{-1}\bB \id$ for a nontrivial solution $\bB$ to the Sylvester equation $(\bsigma_\mu-\lambda/2)\bB + \bB\bsigma_\mu = 0$;
    \item if $\lambda/2=\lambda_i\in\Lambda_\mu\cap \Gamma_\mu$, then $v = -(\bsigma_\mu-\lambda/2)^+ \bB \id + e$, where $\bB$ is a solution to the Sylvester equation $(\bsigma_\mu-\lambda/2)\bB + \bB\bsigma_\mu = 0$, $(\bsigma_\mu-\lambda/2)^+$ is the matrix pseudoinverse, and $e$ is a vector field taking values in eigenspace of $\lambda_i$, with $\int e(x)x^\intercal d\mu(x)=0$.
\end{enumerate}
\end{proposition}

\begin{proof}
For brevity, denote $\cL\coloneqq\cL_{\bsigma_\mu,\mu}|_{\cI_\mu}$ and consider the operator $\cL-\lambda\id$. Suppose that $\lambda/2\in \RR\setminus (\Lambda_\mu \cup\Gamma_\mu)$ and consider the equation $(\cL-\lambda\id)[v]=w$. Following the same approach as in the above proof, we arrive at 
\[(\bsigma_\mu-\lambda/2)\bB + \bB\bsigma_\mu = \frac{1}{2}\int w(y)y^\intercal d\mu(y).\]
Since $\lambda/2\not\in \Gamma_\mu$, we see that $(\bsigma_\mu-\lambda/2)$ and $-\bsigma_\mu$ do not have the same eigenvalues. Consequently, the Sylvester equation has a unique solution $\bB$, which yields a unique inverse, as $v(x) = (\bsigma_\mu-\lambda/2)^{-1}(w(x)-\bB x)$. We conclude that $\lambda/2$ cannot belong to $\spec(\cL)$.

It remains to  characterize the kernel of $\cL-\lambda\id$ for each $\lambda\in2(\Lambda_\mu \cup\Gamma_\mu)$, i.e., nontrivial solutions $v$ to the equation $(\cL-\lambda\id)[v]=0$. Note that we again arrive at the reduced system
\begin{align*}
    (\bsigma_\mu-\lambda/2)\bB + \bB\bsigma_\mu &= 0\\
    (\bsigma_\mu-\lambda/2)v(x)+\bB x&=0,\quad \forall x\in\RR^d.
\end{align*}
When $\lambda/2=\lambda_i\in \Lambda_\mu\setminus \Gamma_\mu$, the Sylvester equation has a unique solution $\bB=0$, hence $(\bsigma_\mu-\lambda/2) v(x)=0$ and $v(x)$ takes values in the eigenspace of $\lambda_i$, which we note could be \emph{nonlinear}. Plugging back we further obtain a sufficient condition of $\int v(x) x^\intercal d\mu(x)=0$. On the other hand, if $\lambda/2\in \Gamma_\mu\setminus \Lambda_\mu$, then the Sylvester equation has nontrivial solutions, and the eigenfunction corresponding to each nontrivial $\bB$ must satisfy $v(x)=-(\bsigma_\mu-\lambda/2)^{-1}\bB x$, which has to be \emph{linear}. Lastly, when $\lambda/2\in\Lambda_\mu\cap \Gamma_\mu$, the Sylvester equation again has nontrivial solutions~$\bB$, and any $v$ in the kernel has to satisfy $(\bsigma_\mu-\lambda/2) v(x) = -\bB x$ for such a $\bB$. Each such $v$ has to abide the form
\begin{align*}
    v = -(\bsigma_\mu-\lambda/2)^+ \bB \id + e,
\end{align*}
where $(\bsigma_\mu-\lambda/2)^+$ is the matrix pseudoinverse, and $e$ is a vector field taking values in eigenspace of $\lambda/2=\lambda_i$. Plugging back, we obtain the sufficient condition for this form to solve $(\cL-\lambda\id)[v]=0$:
\begin{align*}
    \int e(x)x^\intercal d\mu(x) = ((\bsigma_\mu-\lambda/2)(\bsigma_\mu-\lambda/2)^+-(\bsigma_\mu-\lambda/2)^+(\bsigma_\mu-\lambda/2))\bB = 0,
\end{align*}
and $\bB$ solves the Sylvester equation.
\end{proof}

\subsection{Extension of the GMM curve}\label{appen:rem:infinite_time_extension}
Despite the lack of long time control and global convergence for our minimizing movement scheme, we may still prove the follow argument of long time 'existence' of the GMM curve. We note here that this proof uses several arguments and observations from later sections, although placed here to adhere to the order of the main text. Suppose WLOG that the minimum $\sF^\star$ is not attained in finite time.

Suppose again the Assumptions \ref{assumption:main}, \ref{assumption:differentiability}. For $\mu_0\in\dom(\sF)$, by \cref{thm:main} we know that there is an interval $\sI_0=[0,\delta_0]$ with $\delta_0$ depending only on $\mu_0$ and $\sF$ (as is given in \cref{prop:existence_of_scheme}), where there is a GMM curve $\rho^0_t$ starting from $\mu_0$ existing on $\cJ_0 $. Furthermore $\rho^0_t$ satisfies the continuity equation with a velocity field $v^0_t$, and $v^0_t$ satisfies the gradient flow equation $$\cL_{\bsigma_{\rho^0_t},\rho^0_t}[v^0_t]\in\partial\sF(\rho^0_t)$$ for a.e. $t$. Denote $\mu_1\coloneqq\rho^0_{\delta_0}$, and note that since $\mu_1\in\dom(\sF)$, thus having a density, we again satisfies the assumptions, and can thus find a GMM curve  starting from $\mu_1$, defined on a nonempty interval $\cJ_1=[\delta_0,\delta_0+\delta_1]$. Now define
\begin{align*}
    (\check{\rho}^n_t,\check{v}^n_t) = (\rho^i_t, v^i_t), t\in\sI_i, i=0,\ldots,n.
\end{align*}
Since $\check{\rho}^n$, defined on $\cup_{i=0}^n \sI_i$, is Wasserstein continuous, by \cite[Lemma 8.1.2]{ambrosio2005gradient} we conclude that $\check{\rho}^n_t,\check{v}^n_t$ satisfies the continuity equation, and the gradient flow equation a.e. on $\cup_{i=0}^n \sI_i$. Now we show indefinite extension of the piecewise GMM curve, i.e. we show that $\lim_{n\to\infty}\sum_{i=0}^n \delta_i = +\infty$. Suppose by contradiction that it is bounded, i.e. there is $\lim_{n\to\infty}\sum_{i=0}^n \delta_i = A>0$, which implies that $\delta_n\to0$. By \cref{prop:existence_of_scheme}, since the minimum $\sF^\star$ is not attained, the choice of $\delta_n$ implies that $\lambda_{\min}(\bsigma_{\mu_n})\to0$. To derive contradiction, we will show that $\lambda_{\min}(\bsigma_{\mu_n})$ is uniformly bounded away from 0.

By construction, for any $t\in[0,A)$, there is a uniquely defined $\check{\rho}_t$ with $v_t$ satisfying the continuity equation and gradient flow equation. Now we show that $\IGW(\check{\rho}_t,\rho_0)$ is uniformly bounded for all $t\in[0,A)$: recall from \cref{ex:grad_flow_length} that on $\sI_0$, the discrete solution $\bar{\rho}^0_k,\bar{v}^0_k$ with time step $\tau=\delta_0/k$ satisfies that 
\begin{align*}
    &\sum_{i=1}^k \IGW(\rho_i,\rho_{i-1})^2/\tau\\
    &= \sum 2\tau \tr(\bK_i\bsigma_{\rho_i}) + 2\tau \tr(\bL_i^2) + O(\sqrt{\tau}) \\
    & = \int_0^{\delta_0} g_{\bar{\rho}^0_k(t)}(\bar{v}^0_k(t),\bar{v}^0_k(t)) dt + O(\sqrt{\tau})\\
    & = \int_0^ {\delta_0} 2\int \bar{v}^0_k(t,x)^\intercal \bsigma_{\bar{\rho}^0_k(t)} \bar{v}^0_k(t,x) d\bar{\rho}^0_k(t,x) + 2\left\|\int x\bar{v}^0_k(t,x)^\intercal d\bar{\rho}^0_k(t,x)\right\|_\F^2 dt  + O(\sqrt{\tau}).
\end{align*}
Taking $\liminf$ (along proper subsequence) and by Jensen's inequality, we have
\begin{align*}
    &\liminf_k \int_0^{\delta_0} 2\int \bar{v}^0_k(t,x)^\intercal \bsigma_{\bar{\rho}^0_k(t)} \bar{v}^0_k(t,x) d\bar{\rho}^0_k(t,x) + 2\left\|\int x\bar{v}^0_k(t,x)^\intercal d\bar{\rho}^0_k(t,x)\right\|_\F^2 dt\\
    &\geq \int_0^{\delta_0} 2\int v^0_t(x)^\intercal \bsigma_{\rho^0_t} v^0_t(x) d\rho^0_t(x) + 2\left\|\int xv^0_t(x)^\intercal d\rho^0_t(x)\right\|_\F^2 dt\\
    & = \int_0^{\delta_0} g_{\rho^0_t}(v^0_t,v^0_t) dt,
\end{align*}
where note that we need to also invoke similar approach as of \eqref{eq:part1}. Also recall that by definition of the discrete solution, $\IGW (\rho_{i+1},\rho_i)^2\leq 2\tau\big(\sF(\rho_i) - \sF(\rho_{i+1})\big)$, and note that by lower semi-continuity, $\liminf_k \sF(\bar{\rho}^0_k(\delta_0))\geq\sF(\rho^0_{\delta_0})$, hence
\begin{align*}
    &2(\sF(\rho^0_0)-\sF(\rho^0_{\delta_0}) )  \\
    & \geq \liminf_k 2(\sF(\bar{\rho}^0_k(0))-\sF(\bar{\rho}^0_k(\delta_0)) \\
    & \geq \liminf_k \sum_{i=1}^k \IGW (\rho_{i+1},\rho_i)^2/\tau \\
    & =  \liminf_k \int_0^{\delta_0} 2\int \bar{v}^0_k(t,x)^\intercal \bsigma_{\bar{\rho}^0_k(t)} \bar{v}^0_k(t,x) d\bar{\rho}^0_k(t,x) + 2\left\|\int x\bar{v}^0_k(t,x)^\intercal d\bar{\rho}^0_k(t,x)\right\|_\F^2 dt\\
    &\geq \int_0^{\delta_0} g_{\rho^0_t}(v^0_t,v^0_t) dt.
\end{align*}
Clearly this holds true for all intervals $\sI_0,\ldots,\sI_n,\ldots$, and we thus have that for any $T\in[0,A)$,
\begin{align*}
    \int_0^{T} g_{\check{\rho}_t}(\check{v}_t,\check{v}_t) dt \leq \limsup_n 2(\sF(\rho^0_0)-\sF(\rho^0_{\delta_n}) )\leq 2(\sF(\rho^0_0)-\sF^\star)<\infty.
\end{align*}
Note that this is essentially the energy identity \cite[Theorem 2.3.3, Eq (11.2.4)]{ambrosio2005gradient}, though for simplicity we only show an inequality here. 

Next by \cref{lem:IGW_flow_upperbound}, we conclude that for all $t\in[0,A)$, $\check{\rho}_t\in\cB_{\IGW}(\mu_0, \sqrt{2(\sF(\rho^0_0)-\sF^\star)})$. In fact, by the stronger upper bound from \cref{lem:IGW_curve_length_and_flow}, the curve is absolute continuous and of finite length, hence by compactness from \cref{lem:weak_topo_lsc}, for a sequence of numbers $t_n\to A$, $\check{\rho}_{t_n}$ is an IGW Cauchy sequence and has a weak limit $\nu$. Now consider $\lambda_{\min}(\bsigma_{\check{\rho}_{t_n}})$. We next show that this sequence is uniformly lower bounded by $c>0$, otherwise suppose there's a sequence of unit vectors $v_n$ s.t. $\int \llangle v_n, x\rrangle^2 d\check{\rho}_{t_n}(x) \to 0$, and a subsequence, not relabeled for simplicity, $v_n\to v$. Note that
\begin{align*}
    \int \llangle v,x\rrangle^2 d\nu(x) &\leq \liminf \int \llangle v,x\rrangle^2 d\check{\rho}_{t_n}(x)\\
    & = \liminf \int \llangle v_n,x\rrangle^2 d\check{\rho}_{t_n}(x)\\
    &=0
\end{align*}
by the uniform boundedness of the second moments, see \eqref{eq:second_moment_igw_only}. We thus conclude that $\nu\not\in\cP^{\mathrm{ac}}(\RR^d)$, and by \cref{assumption:main} $\infty = \sF(\nu)\leq \liminf \sF(\check{\rho}_{t_n}) \leq \sF(\mu_0)$, a contradiction!

Now suppose $\lambda_{\min}(\bsigma_{\check{\rho}_{t_n}})$ is uniformly lower bounded by $c>0$. Clearly this lower bound has to hold for sequence $t_n\coloneqq\sum_{i=0}^n\delta_n$. Thus for a sufficiently large $n$, $\mu_n=\check{\rho}_{\sum_{i=1}^{n-1}\delta_n}$ has $\lambda_{\min}(\bsigma_{\mu_n})\geq c>0$, which gives the desired contradiction.

Combining above, we conclude that the piecewise GMM curve can be extended to unbounded interval.\qed

\subsection{Proof of \cref{lem:weak_topo_lsc}}\label{appen::weak_topo_lsc_proof}

For the first fact, notice that for any $\nu\in\cB_{\IGW}(\mu,r)$ and the optimal $\pi^\star$ for $\IGW(\nu,\mu)$,
    \begin{align*}
        \|\bsigma_\nu\|_\F^2 &= \int \llangle x,x'\rrangle^2 d\nu\otimes\nu(x,x')\\
        &\leq \int 2\llangle y,y'\rrangle^2 + 2 |\llangle x,x'\rrangle - \llangle y,y'\rrangle |^2d\pi^\star\otimes\pi^\star(x,y,x',y') \\
        &\leq \int 2\llangle y,y'\rrangle^2 d\mu\otimes\mu(x,x') + 2\IGW(\nu,\mu)^2\\
        &\leq 2M_2(\mu)^2 + 2r^2,
    \end{align*}
    by which we conclude that
    \begin{align*}
        &M_2(\nu) = \tr(\bsigma_\nu)\leq \sqrt{d}\|\bsigma_\mu\|_\F \leq \sqrt{d(2M_2(\mu)^2 + 2r^2)}\numberthis\label{eq:second_moment_igw_only}\\
        &\W_2(\nu,\mu)^2\leq 2M_2(\nu)+2M_2(\mu)\leq 2\sqrt{d(2M_2(\mu)^2 + 2r^2)} + 2M_2(\mu).
    \end{align*}
    The weak compactness of $\IGW$ ball now follows from the weak compactness of $\W_2$ balls. 
    
    For the second statement, suppose a sequence $(\mu_,\nu_n)\stackrel{w}{\to}(\mu,\nu)$ in $\cP_2(\RR^d)$, and let $\pi^\star_n\in\Pi(\mu_n,\nu_n)$ be an optimal IGW coupling for the said pair. Denote by $\{(\mu_{n_k},\nu_{n_k})\}_{k\in\NN}$ a subsequence that converges to infimum of $\inf_{n\in\NN}\IGW(\mu_n,\nu_n)$. The sequence $\{\pi^\star_{n_k}\}_k$ is tight by \cite[Lemma 4.4]{villani2008optimal}, thus we have a further subsequence that converges weakly, and for simplicity we still denote the new subsequence $\{\pi^\star_{n_k}\}_k$. Clearly $\pi^\star_{n_k}$ has a weak limit, denoted by $\pi$, and $\pi\in\Pi(\mu,\nu)$. As $\pi_{n_k}\otimes\pi_{n_k}$ also converges weakly to $\pi\otimes\pi$, we have \begin{align*}\liminf_{n\to\infty} \IGW^2(\mu_n,\nu_n) &=  \lim_{k\to\infty} \int |\llangle x,x'\rrangle-\llangle y,y'\rrangle|^2 d\pi_{n_k}\otimes\pi_{n_k}(x,y,x',y')\\
    &\geq \int |\llangle x,x'\rrangle-\llangle y,y'\rrangle|^2 d\pi\otimes\pi(x,y,x',y')\\
    &\geq\IGW^2(\mu,\nu).
    \end{align*}
    A similar derivation for the Wasserstein distance can be found in \cite[Section 5.1.1]{ambrosio2005gradient}.\qed

\subsection{Proof of \cref{lem:bar_delta}}\label{appen:bar_delta_proof}
For item (i) we bound the eigenvalue and moment through \cref{lem:equivalence_igw_w}.
For any $\mu\in\cP_2(\RR^d)$ and unit vector $v\in\RR^d$, we have
\begin{equation}
    \left(\int (v^\intercal y)^2 d \mu(y)\right)^{1/2} \geq \left(\int (v^\intercal x)^2 d \bar\rho_0(y) \right)^{1/2} - \left(\int \big(v^\intercal (x-y)\big)^2 d \pi(y)\right)^{1/2},\label{eq:eigenval_lower_bound_for_convergent_sequence}
\end{equation}
where $\bar{\rho}_0=\bO_\sharp\rho_0$, for $\bO\in\cO_{\mu,\rho_0}$ (namely, the rotated version of $\rho_0$ w.r.t. $\mu$, as in \cref{lem:symmetrization} ), and $\pi$ is taken as the optimal $\W_2$ coupling between $(\mu,\bar{\rho}_0)$. Notice that $\int (v^\intercal x)^2 d \bar{\rho}_0(y)\geq \lambda_{\min}(\bsigma_{\rho_0})$, and \[\int \big(v^\intercal (x-y)\big)^2 d \pi(y) \leq \int \|x-y\|^2 d \pi(y) = \W_2^2(\mu,\bar{\rho}_0) \leq \frac{\sqrt{2}}{\lambda_{\min}(\bsigma_{\rho_0})}\IGW (\mu,\rho_0)^2,\]
where the last inequality follows by \cref{lem:equivalence_igw_w}. Note that the right-hand side (RHS) above only depends on the smallest eigenvalue of the covariance matrix of $\rho_0$. Such bounds are used repeatedly in our analysis since under the proximal mapping, we have access to $\rho_i$ but not $\rho_{i+1}$. This implies
\begin{equation}            \left(\sqrt{\lambda_{\max}(\bsigma_{\rho_0})} +  \frac{2^{1/4}\IGW (\mu,\rho_0)}{\sqrt{\lambda_{\min}(\bsigma_{\rho_0})}} \right)^2 \geq \int (v^\intercal y)^2 d \mu(y) \geq \left(\sqrt{\lambda_{\min}(\bsigma_{\rho_0})} -  \frac{2^{1/4}\IGW (\mu,\rho_0)}{\sqrt{\lambda_{\min}(\bsigma_{\rho_0})}} \right)^2.\label{eq:eigenval_ULB}
\end{equation}
Setting $\bar{\delta} \coloneqq \frac{(1-1/\sqrt{2})\lambda_{\min}(\bsigma_{\rho_0})}{2^{1/4}}$, any $\mu\in\cP_2(\RR^d)\cap \cB_{\IGW}(\rho_0,\bar\delta)$ satisfies $\lambda_{\min}(\bsigma_\mu)\geq \frac{\lambda_{\min}(\bsigma_{\rho_0})}{2}$. Also~observe
\begin{equation}
M_2(\mu) = \int \|y\|^2 d \mu(y)\leq \int \big( 2\|y-x\|^2+2\|x\|^2  d\big) \pi(x,y) \leq \frac{2\sqrt{2}\,\IGW^2 (\mu,\rho_0)}{\lambda_{\min}(\bsigma_{\rho_0})} + 2M_2(\rho_0).\label{eq:M2_bound}
    \end{equation}

Now we proceed to item (ii) and seek to ensure the nonsingularity of $\bA^\star$, which concludes the existence of Gromov-Monge map. Recall that from \eqref{eq:eigenval_bound_in_IGW_ball}, it suffices to have 
\begin{align*}
    \lambda_{\min}(\bsigma_{\mu})^2 + \lambda_{\min}(\bsigma_{\nu})^2 - 4\bar{\delta}^2 \geq \frac{\lambda_{\min}(\bsigma_{\rho_0})^2}{4}>0,
\end{align*}
which follows directly as $\bar{\delta}\leq \frac{\lambda_{\min}(\bsigma_{\rho_0})}{4}$ and $\lambda_{\min}(\bsigma_{\mu})\wedge\lambda_{\min}(\bsigma_{\nu})\geq \frac{\lambda_{\min}(\bsigma_{\rho_0})}{2}$.

\qed

\subsection{Proof of \cref{lem:convexity_of_proximal_functional}}\label{appen:convexity_of_proximal_functional_proof}
    By definition,
    \begin{align*}
        &\sF(\mu_1)+\frac{\IGW(\mu_1,\mu_0)^2}{2\tau}\\
        &\leq \sF(\nu_t)+\frac{\IGW(\nu_t,\mu_0)^2}{2\tau}\\
        &\leq (1-t)\sF(\mu_1)+t\sF(\mu_2) \\
        &\qquad- t(\lambda + \frac{1-8\sqrt{2}\IGW(\mu_0,\mu_1)/c}{2\tau}) \int \left|\llangle y,y'\rrangle-\llangle z,z'\rrangle\right|^2 d\pi\otimes\pi(x,y,z,x',y',z')\\
        &\qquad + (1-t)\frac{\IGW(\mu_1,\mu_0)^2}{2\tau} + t\frac{\IGW(\mu_2,\mu_0)^2}{2\tau} + \frac{6\sqrt{2}}{c\tau}t\IGW(\mu_0,\mu_1)^3 + O(t^2)\\
        &\leq (1-t)\sF(\mu_1)+t\sF(\mu_2) - t\left(\lambda + \frac{1-8\sqrt{2}\IGW(\mu_0,\mu_1)/c}{2\tau}\right) \IGW(\mu_1,\mu_2)^2\\
        &\qquad + (1-t)\frac{\IGW(\mu_1,\mu_0)^2}{2\tau} + t\frac{\IGW(\mu_2,\mu_0)^2}{2\tau} + \frac{6\sqrt{2}}{c\tau}t\IGW(\mu_0,\mu_1)^3 + O(t^2)
    \end{align*}
    where we have used that $\lambda + \frac{1-8\sqrt{2}\IGW(\mu_0,\mu_1)/c}{2\tau}\geq0$ since $1-8\sqrt{2}\IGW(\mu_0,\mu_1)/c\geq1/2$ and $\tau<\frac{1}{4|\lambda|}$. Cancelling same terms on both sides and letting $t\to0$, we have
    \begin{align*}
        \sF(\mu_1)+\frac{\IGW(\mu_1,\mu_0)^2}{2\tau} \leq & \sF(\mu_2) + \frac{\IGW(\mu_2,\mu_0)^2}{2\tau} - \left(\lambda + \frac{1-8\sqrt{2}\IGW(\mu_0,\mu_1)/c}{2\tau}\right) \IGW(\mu_1,\mu_2)^2\\
        & + \frac{6\sqrt{2}}{c\tau} \IGW(\mu_0,\mu_1)^3.
    \end{align*}
\qed

\subsection{Proof of \cref{lem:gronwall}}\label{appen:gronwall_proof}

    Starting from the assumed inequality and multiplying both sides by $e^{A(t)}$, we have
    \begin{align*}
        \frac{d}{dt} \big(e^{A(t)/2}x(t)\big)^2\leq e^{A(t)}c(t) + b(t) e^{A(t)}x(t).
    \end{align*}
    Now denote $y(t) \coloneqq e^{A(t)/2}x(t)$. By Lemma 4.1.8 from \cite{ambrosio2005gradient}, we obtain
    \begin{align*}
        |y(T)|\leq \left(y^2(0) + \sup_{t\in[0,T]} \int_0^t  e^{A(s)}c(s) ds \right)^{1/2} + 2\int_0^T\big|b(t) e^{A(t)/2}\big|dt.
    \end{align*}\qed

\subsection{Proof of \cref{lem:velocity_distributional_limits_equations}} \label{appen:lem:velocity_distributional_limits_equations_proof}

    We now proceed to show \eqref{eq:technical_eq_1} and \eqref{eq:technical_eq_2}. Note that $\bsigma_t$ is the uniform limit of $\bar{\bsigma}_{n_k}(t)$ in $\|\cdot\|_\F$ by \cref{prop:strong_convergence}, which, combined with \cref{prop:existence_of_scheme}, implies that $2\bar{\bA}_{n_k}(t)$ converges to $\bsigma_t$ uniformly. Also note $M_2(\nu_n)$ is bounded, thus \eqref{eq:technical_eq_1} follows directly. For \eqref{eq:technical_eq_2}, we first denote $\phi_n(t)\coloneqq \int xg(t,x)^\intercal d\bar{\rho}_n(t,x)$ and $\phi(t)\coloneqq \int xg(t,x)^\intercal d\rho_t(x)$, and note that by uniform convergence of $\bar{\rho}_{n_k}$ to $\rho$, we have uniform convergence of $\phi_{n_k}$ to $\phi$ in $\|\cdot\|_\F$. Also we have a constant $C>0$ depending on upper bound of $g$ and $\sup_{k}M_2(\bar{\rho}_{n_k}(t))$ (finite by \cref{prop:existence_of_scheme}), such that $\sup_{k}\|\phi_{n_k}\|_\F\vee\|\phi\|_\F\leq C$. For any $\epsilon>0$, compute
    \begin{align*}
        &\liminf_{k\to\infty} \int_0^\delta \int   \bar{v}_{n_k}(t,x)^\intercal \phi_{n_k}(t)x   d\bar{\rho}_{n_k}(t,x) dt + \epsilon \frac{4(\sF(\rho_0)- \sF^\star)}{\lambda_{\min}(\bsigma_{\rho_0})} +  \frac{C^2}{4\epsilon}\int_0^\delta M_2(\bar{\rho}_{n_k}(t,x)) dt \\
        &\geq \liminf_{k\to\infty} \int_0^\delta \int   \bar{v}_{n_k}(t,x)^\intercal \phi_{n_k}(t)x + \epsilon\|\bar{v}_{n_k}(t,x)\|^2 + \frac{C^2}{4\epsilon}\|x\|^2   d\bar{\rho}_{n_k}(t,x) dt\\
        &=\liminf_{k\to\infty} \delta\int_0^\delta  \int   y^\intercal \phi_{n_k}(t)x + \epsilon\|y\|^2 + \frac{C^2}{4\epsilon}\|x\|^2   d\nu_{n_k}(t,x,y) \\
        &= \liminf_{k\to\infty} \delta\int_0^\delta \int   y^\intercal \phi(t)x + \epsilon\|y\|^2 + \frac{C^2}{4\epsilon}\|x\|^2   d\nu_{n_k}(t,x,y) \\
        &\stackrel{\mathrm{(a)}}{\geq} \delta\int_0^\delta \int   y^\intercal \phi(t)x + \epsilon\|y\|^2 + \frac{C^2}{4\epsilon}\|x\|^2   d\nu(t,x,y)  \\
        & \geq \int_0^\delta \int   v_t(x)^\intercal \phi(t)x  d\rho_t(x)dt + \frac{C^2}{4\epsilon} \int_0^\delta M_2(\rho_t) dt,
    \end{align*}
    where in step (a) we have used weakly convergent sequence integrated over lower bounded function, see \cite[Lemma 5.1.7]{ambrosio2005gradient}. Again by uniform convergence of $\bar{\rho}_{n_k}(t,x)$, we have 
    \begin{align*}
        \lim_{k\to\infty}\int_0^\delta M_2(\bar{\rho}_{n_k}(t,x)) dt = \int_0^\delta M_2(\rho_t) dt,
    \end{align*}
    hence we may cancel the convergent term. Driving $\epsilon\to0$ we obtain
    \begin{align*}
        \liminf_{k\to\infty} \int_0^\delta \int   \bar{v}_{n_k}(t,x)^\intercal \phi_{n_k}(t)x   d\bar{\rho}_{n_k}(t,x) dt \geq \int_0^\delta  \int   v_t(x)^\intercal \phi(t)x  d\rho_t(x)dt,
    \end{align*}
    or equivalently, 
    \begin{align*}
        &\liminf_{k\to\infty} \int_0^\delta \int \llangle g(t,x),\int y\bar{v}_{n_k}(t,y)^\intercal d\bar{\rho}_{n_k}(t,y) x \rrangle d\bar{\rho}_{n_k}(t,x) dt \\
        &\quad\quad\quad\quad\quad\quad\quad\quad\quad\quad\quad\quad\quad\quad \geq \int_0^\delta \int\llangle g(t,x),\int yv_t(y)^\intercal d\rho_t(y) x \rrangle d\rho_t(x) dt.
    \end{align*}
    Plugging in $-\xi$, we obtain \eqref{eq:technical_eq_2}, which concludes the proof.\qed

\section{Auxiliary Proofs for \cref{sec:riemann}}\label{appen:riemann_aux_proofs}

\subsection{Proof of \cref{thm:intrinsic_metric_is_geodesic}}\label{appen:thm:intrinsic_metric_is_geodesic}

    To prove Items (1) and (2) it suffices to show the existence of minimizing curve. First note that by \cref{lem:equivalence_igw_w}, the $\W_2$ geodesic is a valid $\IGW$-Lipschitz curve connecting $\mu_0,\mu_1$, hence $\lip_{\IGW}([0,1];\cP_2(\RR^d))\neq\emptyset$ and $\mathsf{d}_{\IGW}<\infty$ always. Consider now a sequence of curves $\{\rho_n\}_{n\in\NN}$, reparametrized to be uniformly Lipschitz with constant $\lip(\rho_n)=\ell_\IGW(\rho_n)<2\mathsf{d}_{\IGW}(\mu_0,\mu_1)$, such that $\ell_\IGW(\rho_n)\to \mathsf{d}_{\IGW}(\mu_0,\mu_1)$. By weak compactness of the IGW ball, we may pick a dense subset $\{q_m\}_{m\in\NN}$ of $[0,1]$ and find a subsequence $\{\rho_{n_k}\}_{k\in\NN}$ that converges weakly at each $q_n$. Denoting the limit by $\rho$, we have $\rho_{n_k}(q_m)\stackrel{w}{\to}\rho_{q_m}$, for each $m\in\NN$ as $k\to\infty$. By \cref{lem:weak_topo_lsc}, IGW is l.s.c., and thus 
    \begin{align*}
        \IGW\big(\rho_{q_m},\rho_{q_l}\big)&\leq \liminf_k \IGW\big(\rho_{n_k}(q_m),\rho_{n_k}(q_l)\big)\\
        &\leq \liminf_k \lip(\rho_{n_k})|q_m-q_l|\\
        &= \mathsf{d}_{\IGW}(\mu_0,\mu_1)|q_m-q_l|.
    \end{align*}
    For any $t\in[0,1]$, fix a subsequence of $\{q_m\}_{m\in\NN}$ with $q_{m_k}\to t$, and since $\rho_{q_{m_k}}$ is a Cauchy sequence in IGW, by \cref{lem:weak_topo_lsc} we may find a weak limit, which we assign to $\rho_t$. Again, by the l.s.c. property, $\rho$ is $\mathsf{d}_{\IGW}(\rho_0,\rho_1)$-Lipschitz, and hence $\ell_\IGW(\rho)\leq \lip(\rho) \leq \mathsf{d}_{\IGW}(\mu_0,\mu_1)$, proving Items (1) and (2).
    
    For the continuity claim from Item (3), suppose, without loss of generality, that $\mu$ is already rotated by $\bO\in\cO_{\mu_0,\mu}$. Denote a connecting curve from $\mu_0$ to $\mu$ by 
    $\rho_t\coloneqq(g_t)_\sharp \pi^\star$, for $g_t(x,y)=(1-t)x+ty$ and an IGW optimal $\pi^\star\in\Pi(\mu_0,\mu)$. We now have 
    \begin{align*}
        \IGW(\rho_t,\rho_s)^2&\leq \int \left| \llangle g_t(x,y), g_t(x',y')\rrangle - \llangle g_s(x,y),g_s(x',y')\rrangle \right|^2d\pi^\star\otimes\pi^\star\\
        &= \int \left|  (t-s)\Big(\llangle y-x  ,  x' \rrangle +  \llangle x ,  y'-x'  \rrangle + (t+s) \llangle  y-x , y' - x'\rrangle \Big)\right|^2d\pi^\star\otimes\pi^\star \\
        &\lesssim  (t-s)^2 \int \| y-x\|^2  \|x'\|^2 d\pi^\star\otimes\pi^\star + (t^2-s^2)^2 \int   \|y-x\|^2 \|y' - x'\|^2 d\pi^\star\otimes\pi^\star \\
        &\lesssim_{M_2(\mu_0)} |t-s|^2\frac{\IGW(\mu_0,\mu)^2}{\lambda_{\min}(\bsigma_{\mu_0})}\left(1 + \frac{\IGW(\mu_0,\mu)^2}{\lambda_{\min}(\bsigma_{\mu_0})}\right).\numberthis\label{eq:proof:thm:IGW_Lipshitz_on_displacement}
    \end{align*}
    Now use the fact that $M_2(\mu)$ is bounded for $\mu\to\mu_0$, whereby \[\mathsf{d}_{\IGW}(\mu_0,\mu)\leq \ell_\IGW(\rho)\lesssim_{\lambda_{\min}(\bsigma_{\mu_0}),M_2(\mu_0)} \IGW(\mu,\mu_0)\to0.\]

\subsection{Proof of \cref{lem:IGW_flow_upperbound}}\label{appen:lem:IGW_flow_upperbound_proof}

    The result essentially follows from Jensen's inequality directly, up to a few regularity arguments. Denote by $\cN(x)$ the standard normal density and, for $\epsilon\in(0,1)$, set $\cN_\epsilon(x)\coloneqq \epsilon^{-d}\cN(x/\epsilon)$. Define $\rho^\epsilon_t\coloneqq \rho_t*\cN_\epsilon$ and $v^\epsilon\coloneqq  (v_t\rho_t)*\cN_\epsilon/\rho^\epsilon_t$, where, slightly abusing notation, we identify a measure with its Lebesgue density. Note that $(\rho^\epsilon,v^\epsilon)$ also solves the continuity equation, and by \cite[Proposition 8.1.8]{ambrosio2005gradient} we have a flow map $X_t:\RR^d\to\RR^d$ that solves the ODE
    \begin{align*}
        \begin{cases}
            \frac{d}{dt} X_t(x) = v^\epsilon_t(X_t(x))\\
            X_0(x) = x
        \end{cases}
    \end{align*}
    $\rho^\epsilon_0$-a.s. such that $\rho^\epsilon_t = (X_t)_\sharp\rho^\epsilon_0$ for all $t\in[0,1]$. With that, consider
    \begin{align*}
    &\IGW(\rho^\epsilon_0,\rho^\epsilon_1)^2 \\
    &\leq \int \left|\llangle x,x'\rrangle-\llangle X_1(x), X_1(x')\rrangle\right|^2 d \rho^\epsilon_0\otimes \rho^\epsilon_0(x,x')\\
    &= \int \left|\int_0^1 \frac{d}{dt}\llangle X_t(x), X_t(x')\rrangle dt \right|^2  d \rho^\epsilon_0\otimes \rho^\epsilon_0(x,x')\\
    & = \int \left| \int_0^1 \llangle v^\epsilon_t(X_t(x)), X_t(x')\rrangle + \llangle X_t(x), v^\epsilon_t(X_t(x'))\rrangle dt \right|^2 d \rho^\epsilon_0\otimes \rho^\epsilon_0(x,x')\\
    & \leq  \int \int_0^1 \left|\llangle v^\epsilon_t(X_t(x)), X_t(x')\rrangle + \llangle X_t(x), v^\epsilon_t(X_t(x'))\rrangle\right|^2 dt   d \rho^\epsilon_0\otimes \rho^\epsilon_0(x,x')\\
    & = \int_0^1 \int  2\Big(\llangle v^\epsilon_t(X_t(x)), X_t(x')\rrangle^2 + \llangle v^\epsilon_t(X_t(x)), X_t(x')\rrangle\llangle X_t(x), v^\epsilon_t(X_t(x'))\rrangle\Big)  d \rho^\epsilon_0\otimes \rho^\epsilon_0(x,x') dt\\
    & = \int_0^1 \int  2\Big(\llangle v^\epsilon_t(y), y'\rrangle^2 + \llangle v^\epsilon_t(y), y'\rrangle \llangle y, v^\epsilon_t(y')\rrangle \Big) d \rho^\epsilon_t\otimes\rho^\epsilon_t(y,y') dt\\
    & = \int_0^1 g_{\rho^\epsilon_t}(v^\epsilon_t,v^\epsilon_t) dt,\numberthis\label{eq:Jensen_lemma_proof1}
    \end{align*}
    where the first inequality is by specifying a coupling, while the latter follows by Jensen's. To conclude the proof, we next show that $\limsup_{\epsilon\to0} \int_0^1 g_{\rho^\epsilon_t}(v^\epsilon_t,v^\epsilon_t) dt \leq \int_0^1 g_{\rho_t}(v_t,v_t) dt$. Combined with l.s.c. of IGW, this will yield the desired result, since $\rho^\epsilon_i\stackrel{w}{\to},\rho_i$, $i=0,1$, as $\epsilon\to0$. 
    
    \medskip
    Suppose $\int_0^1 g_{\rho_t}(v_t,v_t) dt<\infty$. Note that by \cite[Theorem 8.3.1]{ambrosio2005gradient}, the curve $(\rho_t,v_t)_{t\in[0,1]}$ is $\W_2$-absolutely continuous and $\sup_{t\in[0,1]}\W_2(\rho^\epsilon_t,\rho_t)\lesssim_{d}\epsilon$ (see also \cite[Lemma 17.5]{ambrosio2021lectures}), whereby $\sup_{t\in[0,1]}\|\bsigma_{\rho^\epsilon_t}-\bsigma_{\rho_t}\|_\F\lesssim_{d,\sup_t M_2(\rho_t)} \epsilon$. Compute
    \begin{align*}
        \limsup_{\epsilon\to0}&\int_0^1 \mspace{-5mu}g_{\rho^\epsilon_t}(v^\epsilon_t ,v^\epsilon_t ) dt\\
        &=  \limsup_{\epsilon\to0}\mspace{2mu} 2\int_0^1 \mspace{-6mu}\int \llangle v^\epsilon_t(x), \bsigma_{\rho^\epsilon_t } v^\epsilon_t(x) \rrangle d\rho^\epsilon_t(x) \mspace{-1mu}+\mspace{-1mu} 2\tr\left(\int xv^\epsilon_t(x)^\intercal d\rho^\epsilon_t(x)\right)^2\mspace{-5mu} dt\\
        & \stackrel{\mathrm{(a)}}{=} \limsup_{\epsilon\to0} \mspace{2mu}2\int_0^1 \mspace{-6mu}\int \llangle v^\epsilon_t(x), \bsigma_{\rho_t} v^\epsilon_t(x) \rrangle d\rho^\epsilon_t(x) \mspace{-1mu}+\mspace{-1mu} 2\tr\left(\int xv^\epsilon_t(x)^\intercal d\rho^\epsilon_t(x)\right)^2\mspace{-5mu} dt\\
        &\stackrel{\mathrm{(b)}}{\leq} \mspace{2mu}2\int_0^1 \int \llangle v_t(x), \bsigma_{\rho_t} v_t(x) \rrangle d\rho_t(x) + 2\tr\left(\int xv_t(x)^\intercal d\rho_t(x)\right)^2 dt.\numberthis\label{eq:Jensen_lemma_proof2}
    \end{align*}
    First note that for any fixed PSD matrix $\bsigma$, $\int \llangle v^\epsilon_t(x), \bsigma v^\epsilon_t(x) \rrangle d\rho^\epsilon_t(x) \leq \int \llangle v_t(x), \bsigma v_t(x) \rrangle d\rho_t(x)$ by Jensen's inequality (c.f. \cite[Lemma 8.1.10]{ambrosio2005gradient}). For Step (a), we have used the fact that 
    \begin{align*}
        \left|\int_0^1 \int \llangle v^\epsilon_t(x), (\bsigma_{\rho^\epsilon_t }-\bsigma_{\rho_t}) v^\epsilon_t(x) \rrangle d\rho^\epsilon_t(x) dt\right| &\lesssim_{d,\sup_t M_2(\rho_t)} \epsilon \int_0^1 \int\|v^\epsilon_t(x)\|^2 d\rho^\epsilon_t(x) dt \\
        &\leq \epsilon \int_0^1 \int\|v_t(x)\|^2 d\rho_t(x) dt.
    \end{align*}
    For Step (b), in addition to Jensen's inequality, we note that
    \begin{align*}
        \int xv^\epsilon_t(x)^\intercal d\rho^\epsilon_t(x) = \int x  v_t(y)^\intercal \cN_\epsilon(x-y) d\rho_t(y) dx = \int y  v_t(y)^\intercal d\rho_t(y)
    \end{align*}
    as $\int x\cN_\epsilon(x-y)dx=y$. Combining \eqref{eq:Jensen_lemma_proof1} and \eqref{eq:Jensen_lemma_proof2}, we arrive at
    \begin{align*}
        \IGW(\mu_0,\mu_1)^2 &\leq \liminf_{\epsilon\to0} \IGW(\rho^\epsilon_0,\rho^\epsilon_1)^2\\
        & \leq \liminf_{\epsilon\to0}  \int_0^1 g_{\rho^\epsilon_t }(v^\epsilon_t ,v^\epsilon_t ) dt\\
        & \leq \limsup_{\epsilon\to0}\int_0^1 g_{\rho^\epsilon_t }(v^\epsilon_t ,v^\epsilon_t ) dt\\
        &\leq \int_0^1 2\int \llangle v_t(x), \bsigma_{\rho_t} v_t(x) \rrangle d\rho_t(x) + 2\tr\left(\int xv_t(x)^\intercal d\rho_t(x)\right)^2 dt\\
        & = \int_0^1 g_{\rho_t}(v_t,v_t) dt,
    \end{align*}
    which concludes the proof.\qed

\subsection{Proof of \cref{lem:IGW_curve_length_and_flow}}\label{appen:lem:IGW_curve_length_and_flow_proof}

    Without loss of generality we suppose that the curve is parametrized to be $L$-Lipschitz with $L = \ell_\IGW(\rho)$. 
    
    \medskip
    \noindent{\underline{Item (1) -- Lower bound:}} We start from the lower bound from Item (1), which requires most of the work. Given the IGW-continuous curve $(\rho_t)_{t\in[0,1]}$, we will construct an IGW-equivalent curve $(\tilde\rho_t)_{t\in[0,1]}$, i.e., such that $\ell_\IGW(\rho)=\ell_\IGW(\tilde\rho)$, which is also $\W_2$-continuous. Consequently, the latter satisfies the continuity equation together with an appropriate velocity field, and its IGW length will be lower bounded by the action, as desired. 

    The construction employs an auxiliary curve $(\gamma_t)_{t\in[0,1]}$ as an intermediate step, which we describe next. Consider the uniform partition $0=t_0<\ldots<t_n=1$ with step size $\tau=1/n$, and for each $t_i = i/n$, define $\gamma_i\coloneqq \bO_\sharp\rho_{t_i}$ for $\bO\in\cO(\gamma_{i-1},\rho_{t_i})$ for $i=1,\ldots,n$ and $\gamma_0=\rho_0$. Also define $\bar{\gamma}_n(t)\coloneqq \gamma_i$ for $t\in((i-1)\tau,i\tau]$. Thanks to the rotations, the piecewise constant curve $(\bar\gamma_n(t))_{t\in[0,1]}$ has bounded $\W_2$ gap, namely, for $s\leq t$ we have (see derivation of \cref{prop:limit_of_minimizing_movement} for a similar bound) 
    \begin{align*}
    \W_2\big(\bar{\gamma}_n(s),\bar{\gamma}_n(t)\big) &= \W_2\big(\gamma_{\lceil s/\tau \rceil}, \gamma_{\lceil t/\tau \rceil}\big)\\
        &\leq \frac{\sum_{i=\lceil s/\tau \rceil}^{\lceil t/\tau \rceil-1} \IGW(\rho_{t_{i+1}},\rho_{t_i})}{\sqrt{c}}\\
        &\leq \frac{L|\lceil s/\tau \rceil - \lceil t/\tau \rceil|\tau}{\sqrt{c}}
    \end{align*}
    Having that, we invoke \cite[Proposition 3.3.1]{ambrosio2005gradient} to conclude that $\bar{\gamma}_n$ converges weakly to a $\W_2$-Lipschitz limit $(\gamma_t)_{t\in[0,1]}$, along a subsequence $n_k$, with $\gamma_0=\rho_0$ and $\gamma_1=\bO_\sharp\rho_1$ for some $\bO\in\Od$. 
    It's immediate to see that $\bar{\rho}_n$ is an IGW-Cauchy sequence, thus by a same argument as \cref{prop:strong_convergence}, we lift this to uniform $\W_2$ convergence. Notably, while $\gamma$ initiates at the right distribution $\rho_0$, its endpoint is a (possibly) rotated version of $\rho_1$. As we seek a curve that interpolates between $\rho_0$ and $\rho_1$ exactly, we next correct for that rotation in a manner that maintains $\W_2$-Lipschitzness.

    We treat the cases of whether $\bO\in\mathrm{SO}(d)$ or not separately. If $\bO\in\SOd$, then we may find a smooth curve in $\SOd$ that joins $\bI$ and $\bO^{-1}$; otherwise we find a curve joining $\bI$ and $\bI^-\bO^{-1}$. Denoting this curve by $(\bO(t))_{t\in[0,1]}$ and define $\tilde{\rho}_t = \bO(t)_\sharp\gamma_t$, which clearly satisfies the boundary conditions at $t=0,1$. 

\medskip
    To lower bound $\ell_\IGW(\tilde\rho)$ to the action, our next step is to construct the velocity field for $\curve{\tilde\rho_t}$. We start from the standard construction of the velocity field for the $\W_2$-continuous curve $\curve{\gamma_t}$ using transport maps, and then extract the velocity for $\curve{\tilde\rho_t}$ from it, as the curve are related through rotations. Let $T^\star_i$ be the Gromov-Monge map from $\gamma_i$ to $\gamma_{i-1}$, and set $w_i\coloneqq \frac{x-T^\star_i(x)}{\tau}$ and let $\bar{w}_n(t)\coloneqq  w_i$, for $t\in((i-1)\tau,i\tau]$ and $i=1,\ldots,n$, be the corresponding piecewise constant interpolation. Recall that $c>0$ lower bounds the smallest eigenvalue of the covariance matrix along the trajectory, whereby 
    \begin{align*}
        \int\|w_i(x)\|^2 d\gamma_i &= \int \|x-T^\star_i(x)\|^2/\tau^2 d\gamma_i \\
        &\leq \frac{\IGW(\gamma_i,\gamma_{i-1})^2}{c\tau^2}\\
        &\leq \frac{L^2}{c}.
    \end{align*}
    Hence the sequence of joint distributions $\nu_n \coloneqq \upsilon_1[(\id,\bar{w}_n)_\sharp\bar{\gamma}_n]$ is tight and has a weak limit $\nu$ along the further subsequence (again not relabeled). Define $w_t\coloneqq \int y d\nu_{t,x}(y)$ where $\nu_{t,x}$ is the disintegration of $\nu$ w.r.t. its first two marginal $\upsilon_1\gamma$. Similar to \cref{prop:limit_continuity_equation} we conclude that $(\gamma_t,w_t)_{t\in[0,1]}$ solves the continuity equation, i.e.,
    \begin{align*}
        \iint \partial_t g(t,x) d\gamma_t(x)dt = -\iint \llangle\nabla g(t,x),w_t(x)\rrangle d\gamma_t(x)dt,\quad \forall g\in C^\infty_c((0,1)\times \RR^d).
    \end{align*}
    To identify the appropriate velocity field for $\curve{\tilde\rho_t}$, compute
    \begin{align*}
        \iint &\partial_t g(t,x) d\tilde{\rho}_t(x)dt\\
        &\stackrel{\mathrm{(a)}}{=} \iint \Big(\partial_t (g(t,\bO(t)x)) - \llangle(\nabla g)(t,\bO(t)x) , \bO(t)'x \rrangle \Big)d\gamma_t(x)dt  \\
        & \stackrel{\mathrm{(b)}}{=} -\iint \llangle \bO(t)^\intercal(\nabla g)(t,\bO(t)x),w_t(x)\rrangle d\gamma_t(x)dt  - \iint \llangle(\nabla g)(t,\bO(t)x) , \bO(t)'x \rrangle d\gamma_t(x)dt \\
        & = -\iint \llangle \nabla g(t,\bO(t)x),\bO(t) w_t(x) + \bO(t)'x\rrangle d\gamma_t(x)dt \\
        & = -\iint \llangle \nabla g(t,x),\bO(t) w_t(\bO(t)^\intercal x) + \bO(t)'\bO(t)^\intercal x\rrangle d\tilde{\rho}_t(x)dt ,
    \end{align*}
    where note that in step (a) we treat $g(t,\bO(t)x)$ as a function of $t,x$ and takes its partial derivative, and in step (b) we write $(\nabla g)$ as the gradient function w.r.t. the space slot. We conclude that $\curve{\tilde{\rho}_t, v_t}$ satisfies the continuity equation for $v_t(x)\coloneqq \bO(t) w_t(\bO(t)^\intercal x) + \bO(t)'\bO(t)^\intercal x$. We also observe that $g_{\tilde{\rho}_t}(v_t,v_t)=g_{\gamma_t}(w_t,w_t)$, for all $t\in[0,1]$. Indeed:
    \begin{align*}
        &g_{\tilde{\rho}_t}(v_t,v_t)\\
        &=  \int \Big(\llangle v_t(x), x'\rrangle + \llangle x, v_t(x')\rrangle\Big)^2 d \tilde{\rho}_t\otimes \tilde{\rho}_t(x,x')\\
        & \stackrel{\mathrm{(a)}}{=} \int \Big(\llangle \bO(t) w_t(\bO(t)^\intercal x) + \bO(t)'\bO(t)^\intercal x, x'\rrangle \\
        &\qquad\qquad\qquad\qquad\qquad\qquad+ \llangle x, \bO(t) w_t(\bO(t)^\intercal x') + \bO(t)'\bO(t)^\intercal x'\rrangle\Big)^2 d \tilde{\rho}_t\otimes \tilde{\rho}_t(x,x')\\
        & = \int \Big(\llangle \bO(t) w_t(x) + \bO(t)'x, \bO(t)x'\rrangle + \llangle \bO(t)x, \bO(t) w_t(x') + \bO(t)' x'\rrangle\Big)^2 d \gamma_t\otimes\gamma_t(x,x')\\
        & = \int \Big(\llangle w_t(x) , x'\rrangle + \llangle \bO(t)' x , \bO(t)x'\rrangle + \llangle  x, w_t(x') \rrangle + \llangle \bO(t)x, \bO(t)' x'\rrangle\Big)^2 d \gamma_t\otimes\gamma_t(x,x')\\
        & = \int \Big(\llangle w_t(x) , x'\rrangle + \llangle  x, w_t(x') \rrangle + \partial_t\llangle \bO(t)x, \bO(t) x'\rrangle\Big)^2 d \gamma_t\otimes\gamma_t(x,x')\\
        & \stackrel{\mathrm{(b)}}{=} \int \Big(\llangle w_t(x) , x'\rrangle + \llangle  x, w_t(x') \rrangle \Big)^2 d \gamma_t\otimes\gamma_t(x,x')\\
        & = g_{\gamma_t}(w_t,w_t),
    \end{align*}
    where we plugged in the definition of $v_t$ in step (a), and in step (b) we have used that $\llangle \bO(t)x, \bO(t) x'\rrangle = \llangle x, x'\rrangle$ since $\bO(t)\in\Od$.

    \medskip
    With this equality at hand, to conclude the proof of the lower bound it suffices to show that 
    \begin{align*}
        \int_0^1 g_{\gamma_t}(w_t,w_t) dt \leq \ell_\IGW(\gamma)^2. 
    \end{align*}
    First expand
    \begin{align*}
    &\IGW(\gamma_{i},\gamma_{i-1})^2\\
    &= \int \left|\llangle x,x'\rrangle-\llangle x-\tau w_i(x), x'-\tau w_i(x')\rrangle\right|^2 d \gamma_{i}\otimes \gamma_{i}(x,x')\\
    &= \int \left|\tau\llangle w_i(x),x'\rrangle + \tau\llangle x,w_i(x')\rrangle -\tau^2\llangle w_i(x),  w_i(x')\rrangle\right|^2 d \gamma_{i}\otimes \gamma_{i}(x,x')\\
    & = \int \Big(2\tau^2 \llangle w_i(x),x'\rrangle^2 + 2\tau^2\llangle w_i(x),x'\rrangle \llangle x,w_i(x')\rrangle -4\tau^3\llangle w_i(x),x'\rrangle\llangle  w_i(x),  w_i(x')\rrangle\\
    &\qquad\qquad\qquad\qquad\qquad\qquad\qquad\qquad\qquad\qquad\qquad\qquad+ \tau^4\llangle  w_i(x),  w_i(x')\rrangle^2 \Big)d \gamma_{i}\otimes \gamma_{i}(x,x')\\
    & = \tau^2 g_{\gamma_i}(w_i,w_i) + R_i. 
    \end{align*}
    where $R_i\coloneqq\int \big(- 4\tau^3 \llangle w_i(x),x'\rrangle\llangle  w_i(x),  w_i(x')\rrangle + \tau^4\llangle  w_i(x),  w_i(x')\rrangle^2 \big)d \gamma_{i}\otimes \gamma_{i}(x,x') = O(\tau^3)$. Recalling that $\IGW(\gamma_i,\gamma_{i-1})=\IGW(\rho_i,\rho_{i-1})\leq L\tau$ with $L=\ell_\IGW(\rho)=\ell_\IGW(\gamma)$, the above implies
    \begin{align*}
         \int_0^1 g_{\bar{\gamma}_n(t)}(\bar{w}_n(t),\bar{w}_n(t)) dt &= \tau\sum_{i=1}^n g_{\gamma_i}(w_i,w_i)\\
         & = \frac{1}{\tau}\sum_{i=1}^n \IGW(\gamma_{i},\gamma_{i-1})^2 -\frac{R_i}{\tau^2}\\
         & \leq \ell_\IGW(\gamma) + O(\tau).\numberthis\label{eq:UB1}
    \end{align*}
    Thus, it remains to establish the continuous-time flow associated with $\curve{\gamma_t,w_t}$ as a lower bound on that of the piecewise constant interpolation, from the left-hand side of \eqref{eq:UB1}. Specifically, we will show that
    \[
             \liminf_{k\to\infty} \int_0^1 g_{\bar{\gamma}_{n_k}(t)}(\bar{w}_{n_k}(t),\bar{w}_{n_k}(t)) dt 
         \geq \int_0^1 g_{\gamma_t}(w_t,w_t)dt,
    \]
    where recall that $n_k$ is the subsequence along which $\bar{\gamma}_n$ converges uniformly in $\W_2$ to $\gamma$.
    To that end, consider the decomposition
    \begin{align*}
        \int_0^1 g_{\bar{\gamma}_n(t)}(\bar{w}_n(t),\bar{w}_n(t)) dt & = \tau \sum_{i=1}^n g_{\gamma_i}(w_i,w_i)\\
        &= \int_0^1 \int \left|\llangle \bar{w}_n(t,x),x'\rrangle + \llangle x,\bar{w}_n(t,x')\rrangle\right|^2 d\bar{\gamma}_n(t,x)d\bar{\gamma}_n(t,x') dt\\
        & = \int_0^1 2\int \bar{w}_n(t,x')^\intercal\bsigma_{\bar{\gamma}_n(t)}\bar{w}_n(t,x') d\bar{\gamma}_n(t,x)dt + \int_0^1 2\tr(\bar{\bL}_n(t)^2) dt.\numberthis\label{eq:target}
    \end{align*} 
    For the first term on the right-hand side, as $\bsigma_{\bar{\gamma}_{n_k}(t)}$ converges uniformly to $\bsigma_{\gamma_t}$, we have
    \begin{align*}
        R_k&\coloneqq\left|\int_0^1 2\int \bar{w}_{n_k}(t,x')^\intercal (\bsigma_{\bar{\gamma}_{n_k}(t)} - \bsigma_{\gamma_t}) \bar{w}_{n_k}(t,x') d\bar{\gamma}_{n_k}(t,x) dt \right|\\
        &\lesssim_{\sup_t M_2(\gamma_t)} \W_2(\bar{\gamma}_{n_k}(t),\gamma_t) \int_0^1 \|\bar{w}_{n_k}(t,x')\|_{L^2(\bar{\gamma}_{n_k}(t);\RR^d)}^2 dt\\
        &\to0
    \end{align*}
    as $k\to\infty$. Consequently
    \begin{align*}
        \liminf_{k\to\infty} \int_0^1 2\int \bar{w}_{n_k}(t,x')^\intercal&\bsigma_{\bar{\gamma}_{n_k}(t)}\bar{w}_{n_k}(t,x') d\bar{\gamma}_{n_k}(t,x) dt \\
        &\geq \liminf_{k\to\infty} \int_0^1 2\int \bar{w}_{n_k}(t,x')^\intercal\bsigma_{\gamma_t}\bar{w}_{n_k}(t,x') d\bar{\gamma}_{n_k}(t,x) dt-R_k\\
        & = \liminf_{k\to\infty} \int_0^1 2\int y^\intercal\bsigma_{\gamma_t} y\, d\nu_{n_k}(t,x,y) \\
        & \geq \int_0^1 2\int y^\intercal\bsigma_{\gamma_t} y\, d\nu(t,x,y) \\
        & \geq \int_0^1 2\int w_t(x)^\intercal\bsigma_{\gamma_t} w_t(x) d\gamma_t(x) dt, \numberthis\label{eq:part1}
    \end{align*}
    where the penultimate step uses the weak convergence of $\nu_{n_k}$, while the last one is Jensen's~inequality. 
    
    To deal with the second term, set $\bL_{i}\coloneqq\int x w_i(x)^\intercal d\gamma_{i}(x)\in\RR^{d\times d}$, $\bar{\bL}_n(t)\coloneqq  \bL_{i}$, for $t\in((i-1)\tau,i\tau]$ and $i=1,\ldots,n$, and $\bL(t)\coloneqq \int x w_t(x)^\intercal d\gamma_t(x)$, noticing that $\|\bL(t)\|_\F<\infty$ for a.e. $t$. Since $\bL_i$ is symmetric by construction, for any matrix-valued function of time $g \in C_c^\infty((0,1);\RR^{d\times d})$, we have
    \begin{align*}
        \lim_{k\to\infty} \int \tr\big(g(t)\bar{\bL}_{n_k}(t)\big) dt &= \lim_{k\to\infty} \int y^\intercal g(t) x\, d\nu_{n_k}(t,x,y)\\
        & = \int y^\intercal g(t) x\, d\nu(t,x,y)\\
        & = \int \tr\big(g(t)\bL(t)\big) dt,
    \end{align*}
    where the limit follows similarly to \eqref{eq:technical_eq_2}. Taking $g$ such that $g(t)^\intercal = -g(t)$, we obtain 
    \[0=\lim_{k\to\infty} \int \tr\big(g(t)\bar{\bL}_{n_k}(t)\big)=\int \tr\big(g(t)\bL(t)\big) dt,\]
    which implies that $\bL(t)$ is symmetric for a.e. $t$. Since $\|\bar{\bL}_{n_k}(t)\|_\F^2 \leq M_2(\bar{\gamma}_n(t))   L^2/c$, the measure $\eta_n\coloneqq (\id,\bar{\bL}_{n})_\sharp \upsilon_1$ is tight and has a weak limit $\eta(t,\bL)$ along a further subsequence (again not relabeled). Clearly the marginal over variable $t$ is $\upsilon_1$, and we write $\eta_t(\bL)$ for the disintegration of $\eta(t,\bL)$ w.r.t. $t$. Since
    \begin{align*}
        \lim_{k\to\infty} \int \tr\big(g(t)\bar{\bL}_{n_k}(t)\big) dt &= \lim_{k\to\infty} \int \tr\big(g(t) \bL \big) d\eta_{n_k}(t,\bL)\\
        & = \int \tr\big(g(t) \bL \big) d\eta(t,\bL)\\
        & = \int \tr\left(g(t) \int\bL d\eta_t(\bL)  \right) dt,
    \end{align*}
    where the limit again follows from the uniformly bounded second moment, compared to the earlier limit $\lim_{k\to\infty} \int \tr\big(g(t)\bar{\bL}_{n_k}(t)\big) dt  = \int \tr\big(g(t)\bL(t)\big) dt$, we have $\int\bL d\eta_t(\bL) = \bL(t)$ for a.e. $t$. Consequently, we obtain
    \begin{align*}
        \liminf_{k\to\infty} \int_0^1 2\tr\big(\bar{\bL}_{n_k}(t)^2\big) dt & = \liminf_{k\to\infty}\int_0^1 2  \|\bar{\bL}_{n_k}(t)\|_\F^2 dt\\
        & = \liminf_{k\to\infty}\int 2  \|\bL\|_\F^2 d\eta_{n_k}(t,\bL)\\
        & \geq \int 2  \|\bL\|_\F^2 d\eta(t,\bL)\\
        & \geq \int_0^1 2  \left\|\int \bL d\eta_t(\bL)\right\|_\F^2 dt\\
        & = \int_0^1 2  \|\bL(t)\|_\F^2 dt\\
        & = \int_0^1 2  \tr\big(\bL(t)^2\big) dt.\numberthis\label{eq:part2}
    \end{align*}
    Plugging \eqref{eq:part1} and \eqref{eq:part2} back into \eqref{eq:target}, we obtain the desired limit
    \begin{align*}
        \liminf_{k\to\infty} \int_0^1 g_{\bar{\gamma}_{n_k}(t)}\big(\bar{w}_{n_k}(t),\bar{w}_{n_k}(t)\big) dt & \geq \int_0^1 2\int w_t(x)^\intercal\bsigma_{\gamma_t} w_t(x) d\gamma_t(x) dt + \int_0^1 2  \tr(\bL(t)^2) dt\\
        & = \int_0^1 g_{\gamma_t}(w_t,w_t)dt,
    \end{align*}
    which, together with \eqref{eq:UB1}, concludes the proof of the lower bound as $\tau\to 0$ when $k\to\infty$.

    \medskip
    \noindent{\underline{Item (2) -- Upper bound:}} The upper bound essentially follows from \cref{lem:IGW_flow_upperbound}. Let $\curve{\tilde\rho_t,v_t}$ with $\sup_{t\in[0,1]}\IGW(\tilde\rho_t,\rho_t)=0$ satisfy the continuity equation $\partial_t\tilde{\rho}_t + \nabla\cdot \tilde{\rho}_t v_t =0$. For an interval $[r,r+h]\subset(0,1)$ with $h>0$, consider the reparametrized curve $\gamma_s = \tilde{\rho}_{r+ sh}$, where $\gamma_s$ connects $\tilde{\rho}_r, \tilde{\rho}_{r+h}$ for $s\in[0,1]$ and solves $\partial_s\gamma + \nabla\cdot \gamma w=0$ with $w_s(x)= h v_{r+sh}(x)$. Compute 
    \begin{align*}
        \IGW(\rho_r, \rho_{r+h})^2 &\leq \int_0^1 g_{\gamma_s}(w_s,w_s) ds\\
        & = h^2 \int_0^1 g_{\tilde{\rho}_{r+ sh}}(v_{r+ sh},v_{r+ sh}) ds\\
        & = h \int_{r}^{r+h} g_{\tilde{\rho}_t}(v_t,v_t) d t,
    \end{align*}
    where the inequality comes from \cref{lem:IGW_flow_upperbound}.

    By \cref{lem:metric_derivative}, $\lim_{h\to0}\frac{\IGW(\rho_r, \rho_{r+h})}{h} = |\rho'|(r)$ for a.e. $t$. Suppose $\int_{0}^{1} g_{\tilde{\rho}_t}(v_t,v_t) d t <\infty$, as otherwise the inequality trivializes, and define $G(r)\coloneqq\int_0^r g_{\tilde{\rho}_t}(v_t,v_t) d t$. Clearly, $G$ is absolutely continuous and $G'(r)$ exists for a.e. $r$ with $G'(r)=g_{\tilde{\rho}_t}(v_t,v_t)$ a.e. Furthermore, 
    \[
        |\rho'|(r)^2 = \lim_{h\to0} \frac{\IGW(\rho_r, \rho_{r+h})^2}{h^2} \leq \lim_{h\to0}\frac{G(r+h)-G(r)}{h}=G'(r)
        \]
    for a.e. $r$, and thus \[
        \ell_\IGW(\rho)^2 = \left(\int_0^1 |\rho'|(t) dt\right)^2 \leq \int_0^1 |\rho'|(t)^2 dt= \int_0^1 G'(t)dt= \int_0^1 g_{\tilde{\rho}_t}(v_t,v_t) d t,
    \]
    which concludes the proof.\qed

    \subsection{Proof of \cref{lem:approximating_curve_with_density}}\label{appen:lem:approximating_curve_with_density_proof}
    Suppose $\curve{\rho_t}$ is parametrized to be $L$-Lipschitz with $L=\ell_{\IGW}(\rho)$. Fixing $\epsilon>0$, we will construct a new curve $\curve{\gamma_t^\epsilon}$ with a corresponding velocity field $\curve{v_t^\epsilon}$ connecting
    \begin{equation}
        \rho_0\to \rho_0*\cN_\epsilon \to \rho_1*\cN_\epsilon\to \rho_1,\label{eq:convolved_curve}
    \end{equation}
    and control the action along each of the three pieces. The middle piece will be instantiated as the convolved curve $\curve{\rho_t*\cN_{\epsilon}}$, which satisfies the assumption of the first part of \cref{lem:IGW_curve_length_and_flow}, and therefore, has a velocity field associated with it (rather, an IGW-equivalent version thereof). Will show that the convolution only elongates the curve by a negligible amount, yielding an squared IGW-length of at most $\ell_\IGW(\rho)^2+O(\epsilon)$. The two remaining pieces, connecting $\rho_0,\rho_1$ and their Gaussian-smoothed versions, also have corresponding velocities and only contribute another $O(\epsilon)$ to the overall~length.

\medskip
    We start by analyzing the convolved curve $\curve{\rho_t*\cN_\epsilon}$, which accounts for the intermediate piece in \eqref{eq:convolved_curve}. By Item (1) of \cref{lem:IGW_curve_length_and_flow}, there exists an IGW equivalent curve $\tilde{\gamma}$ that is $\W_2$-Lipschitz, with $\tilde{v}_t$ that satisfies the continuity equation, such that
    \begin{align*}
        \int_0^1 g_{\tilde{\gamma}_t}(\tilde{v}_t,\tilde{v}_t) dt \leq \ell_\IGW(\tilde{\gamma})^2,
    \end{align*}
    and $\tilde{\gamma}_1\in\{\rho_1*\cN_\epsilon,(\bI^-_\sharp\rho_1)*\cN_\epsilon\}$. We provide the proof for when $\tilde{\gamma}_1 = \rho_1*\cN_\epsilon$; the derivation for the other case is similar. We next show that $\ell_\IGW(\tilde{\gamma})^2\leq \ell_\IGW(\rho)^2+O(\epsilon)$. For $s,t\in[0,1]$ and IGW plan $\pi\in\Pi(\rho_s,\bO_\sharp\rho_t)$ for $\bO\in\cO_{\rho_s,\rho_t}$, note that $\bO_\sharp (\rho_t*\cN_\epsilon) = (\bO_\sharp\rho_t)*\cN_\epsilon$ and compute
    \begin{align*}
        &\IGW(\rho_s*\cN_\epsilon, \rho_t*\cN_\epsilon)^2 \\
        &=\IGW(\rho_s*\cN_\epsilon,(\bO_\sharp\rho_t)*\cN_\epsilon)^2 \\
        & \stackrel{\mathrm{(a)}}{\leq}  \int  \big|\llangle x+ \epsilon z,x'+\epsilon z'\rrangle-\llangle y+\epsilon z,y'+\epsilon z'\rrangle \big|^2 d\pi\otimes\pi(x,y,x',y') d\cN_1\otimes\cN_1(z,z')\\
        & = \int  \big|\llangle x ,x' \rrangle-\llangle y ,y' \rrangle   + \epsilon \big(\llangle x-y , z'\rrangle +\llangle   z,x'-y' \rrangle\big) \big|^2 d\pi\otimes\pi(x,y,x',y') d\cN_1\otimes\cN_1(z,z')\\
        & = \int  \big|\llangle x ,x' \rrangle-\llangle y ,y' \rrangle\big|^2 d\pi\otimes\pi(x,y,x',y')\\
        &\qquad+ 2\epsilon\int \big(\llangle x ,x' \rrangle-\llangle y ,y' \rrangle\big)\big(\llangle x-y , z'\rrangle +\llangle   z,x'-y' \rrangle\big)  d\pi\otimes\pi(x,y,x',y') d\cN_1\otimes\cN_1(z,z')\\
        &\qquad +  \epsilon^2 \int \big(\llangle x-y , z'\rrangle +\llangle   z,x'-y' \rrangle\big)^2  d\pi\otimes\pi(x,y,x',y') d\cN_1\otimes\cN_1(z,z') \\   
        &\stackrel{\mathrm{(b)}}{\leq} \IGW(\rho_s,\rho_t)^2 \\
        &\qquad+ 2 \epsilon\IGW(\rho_s,\rho_t) \sqrt{\int \big(\llangle x-y , z'\rrangle +\llangle   z,x'-y' \rrangle\big)^2  d\pi\otimes\pi(x,y,x',y') d\cN_1\otimes\cN_1(z,z')}\\
        &\qquad +  \epsilon^2 \int \big(\llangle x-y , z'\rrangle +\llangle   z,x'-y' \rrangle\big)^2  d\pi\otimes\pi(x,y,x',y') d\cN_1\otimes\cN_1(z,z') \\   
        &\stackrel{\mathrm{(c)}}{\leq} \IGW(\rho_s,\rho_t)^2 \\
        &\qquad+ 2 \epsilon\IGW(\rho_s,\rho_t) \sqrt{\mspace{-4mu}\int\mspace{-4mu}  2\big(\|x-y\|^2\|z'\|^2 + \|   z\|^2\|x'-y'\|^2\big) d\pi\otimes\pi(x,y,x',y') d\cN_1\otimes\cN_1(z,z')}\\
        &\qquad +  \epsilon^2 \int  2\big(\|x-y\|^2\|z'\|^2 + \|   z\|^2\|x'-y'\|^2\big) d\pi\otimes\pi(x,y,x',y') d\cN_1\otimes\cN_1(z,z') \\   
        &= \IGW(\rho_s,\rho_t)^2 + 4\sqrt{d} \epsilon\IGW(\rho_s,\rho_t) \sqrt{\int \|x-y\|^2 d\pi\otimes\pi(x,y,x',y') } + 4d\epsilon^2 \int \|x-y\|^2 d\pi(x,y) \\
        &\stackrel{\mathrm{(d)}}{\leq} \left(1+4\sqrt{\frac{d}{c}} \epsilon + 4\frac{d}{c}\epsilon^2\right)\IGW(\rho_s,\rho_t)^2  \\
        &\leq \left(1+4\sqrt{\frac{d}{c}} \epsilon + 4\frac{d}{c}\epsilon^2\right)L^2|s-t|^2,  
    \end{align*}
    where (a) is by specifying the coupling of $\rho_s*\cN_\epsilon,(\bO_\sharp\rho_t)*\cN_\epsilon$ given by $(X+\epsilon Z, Y+\epsilon Z)$, where $(X,Y)\sim\pi$ is independent to $Z\sim\cN_1$ steps (b) and (c) use the used Cauchy–Schwarz inequality, and (d) comes from \eqref{eq:igw_w_comparison_intermediate_bound}. Conclude that 
    \begin{equation}
    \ell_\IGW(\tilde{\gamma})^2 = \ell_\IGW(\rho^\epsilon)^2 \leq \left(1+4\sqrt{\frac{d}{c}} \epsilon + 4\frac{d}{c}\epsilon^2\right)\ell_\IGW(\rho)^2\label{eq:flow01_UB}.
    \end{equation}

\medskip
    We next consider the curves connecting $\rho_i$ and their convolved versions $\rho_i*\cN_{\epsilon}$, $i=0,1$, accounting for the start and end segments in \eqref{eq:convolved_curve}. Starting from $\rho_0\to \rho_0*\cN_{\epsilon}$, consider the curve $\curve{\rho_0*\cN_{t\epsilon}}$ (later, we shall rescale time to ensure that the overall curve is parameterized by $t\in[0,1]$). Clearly, the curve is $\W_2$-Lipschitz with constant $\sqrt{d}\epsilon$: 
    \begin{align*}
\W_2(\rho_0*\cN_{s\epsilon},\rho_0*\cN_{t\epsilon})^2 \leq \int \|(x+s\epsilon z) -(x+t\epsilon z)\|^2 d\rho_0(x)d\cN_1(z) = d\epsilon^2|s-z|^2.
    \end{align*}
    By \cite[Theorem 8.3.1]{ambrosio2005gradient}, there is $\curve{v_{0,t}}$ such that $\curve{\rho_0*\cN_{t\epsilon},v_{0,t}}$ satisfies the continuity equation, and $\int_0^1 \|v_{0,t}\|_{L^2(\rho_0*\cN_{t\epsilon};\RR^d)}^2 dt \leq d\epsilon^2$. The corresponding action is small with $\epsilon$ via
    \begin{align*}
        \int_0^1 g_{\rho_0*\cN_{t\epsilon}}(v_{0,t},v_{0,t}) dt &= \int_0^1 \llangle v_{0,t},  \cL_{\bsigma_{\rho_0*\cN_{t\epsilon}},\rho_0*\cN_{t\epsilon}}[v_{0,t}]\rrangle_{L^2(\rho_0*\cN_{t\epsilon};\RR^d)} dt\\
        &\leq \frac{1}{2}\int_0^1 \Big(\|v_{0,t}\|_{L^2(\rho_0*\cN_{t\epsilon};\RR^d)}^2 + \|\cL_{\bsigma_{\rho_0*\cN_{t\epsilon}},\rho_0*\cN_{t\epsilon}}[v_{0,t}]\|_{L^2(\rho_0*\cN_{t\epsilon};\RR^d)}^2 \Big)dt\\
        &\leq \frac{1}{2}\int_0^1 \Big(1+8\big(\|\bsigma_{\rho_0*\cN_{t\epsilon}}\|_\op^2+M_2(\rho_0*\cN_{t\epsilon})^2\big)\Big)\|v_{0,t}\|_{L^2(\rho_0*\cN_{t\epsilon};\RR^d)}^2 dt\\
        &\leq \frac{1}{2}\big(1+16 (M_2(\rho_0)+\epsilon^2 d)^2\big) \int_0^1  \|v_{0,t}\|_{L^2(\rho_0*\cN_{t\epsilon};\RR^d)}^2 dt\\
        &\leq \frac{1}{2}\big(1+16 (M_2(\rho_0)+\epsilon^2 d)^2\big) d\epsilon^2,\numberthis\label{eq:flow0_UB}
    \end{align*}
    where the second inequality uses the bound \eqref{eq:operator_norm_of_cL} to control $\|\cL_{\bsigma_{\rho_0*\cN_{t\epsilon}},\rho_0*\cN_{t\epsilon}}[v_{0,t}]\|_{L^2(\rho_0*\cN_{t\epsilon};\RR^d)}^2$. By a similar argument, we have the pair $\curve{\rho_1*\cN_{t\epsilon}, v_{1,t}}$ satisfying the continuity equation and
    \begin{equation}
        \int_0^1 g_{\rho_1*\cN_{t\epsilon}}(v_{1,t},v_{1,t}) dt \leq \frac{1}{2}\big(1+16 (M_2(\rho_1)+\epsilon^2 d)^2\big) d\epsilon^2.\label{eq:flow1_UB}
    \end{equation}

    \medskip
    To conclude, we assemble the overall curve and its velocity field from the above pieces. For each $t\in[0,1]$, define
    \begin{align*}
        (\gamma_t^\epsilon,v_t^\epsilon) \coloneqq \begin{cases}
            \left(\rho_0*\cN_{t}, \frac{1}{\epsilon}v_{0,\frac{t}{\epsilon}}\right) &, t\in[0,\epsilon)\\
            \left(\tilde{\gamma}_{\frac{t-\epsilon}{1-2\epsilon}} , \frac{1}{1-2\epsilon}\tilde{v}_{\frac{t-\epsilon}{1-2\epsilon}}\right) &, t\in[\epsilon,1-\epsilon)\\
            \left(\rho_1*\cN_{1-t}, -\frac{1}{\epsilon}v_{1,\frac{1-t}{\epsilon}}\right) &, t\in[1-\epsilon,1]
        \end{cases}.
    \end{align*}
    It is straightforward to verify that the time rescaling preserves the continuity equation on the intervals $(0,\epsilon),(\epsilon,1-\epsilon),(1-\epsilon,1)$, respectively. Moreover, since the three curves are all $\W_2$-Lipschitz, we have that $\curve{\gamma_t^\epsilon}$ is also $\W_2$-Lipschitz. By \cite[Lemma 8.1.2]{ambrosio2005gradient}, we may extend the continuity equation to $(0,1)$, i.e. $\curve{\gamma_t^\epsilon,v_t^\epsilon}$ solves the continuity equation on $\RR^d\times(0,1)$. Employing the bounds from \eqref{eq:flow01_UB}-\eqref{eq:flow1_UB}, we lastly show that the action of the combined curve satisfies the desired~bound:
    \begin{align*}
        &\int_0^1 g_{\gamma_t^\epsilon}(v_t^\epsilon, v_t^\epsilon) dt\\
        &= \int_0^\epsilon g_{\gamma_t^\epsilon}(v_t^\epsilon,v_t^\epsilon) dt+ \int_\epsilon^{1-\epsilon} g_{\gamma_t^\epsilon}(v_t^\epsilon,v_t^\epsilon) dt+ \int_{1-\epsilon}^1 g_{\gamma_t^\epsilon}(v_t^\epsilon,v_t^\epsilon) dt \\
        & = \int_0^1 \epsilon g_{\rho_0*\cN_{t\epsilon}}\left(\frac{v_{0,t}}{\epsilon},\frac{v_{0,t}}{\epsilon}\right) dt \mspace{-4mu}+\mspace{-4mu} \int_0^1 (1-2\epsilon) g_{\tilde{\gamma}_t}\left(\frac{\tilde{v}_t}{1-2\epsilon},\frac{\tilde{v}_t}{1-2\epsilon}\right) dt + \int_0^1 \epsilon g_{\rho_1*\cN_{t\epsilon}}\left(\frac{v_{1,t}}{\epsilon},\frac{v_{1,t}}{\epsilon}\right) dt\\
        & \leq \frac{1}{2}\big(1+16 (M_2(\rho_0)+\epsilon^2 d)^2\big) d\epsilon + \frac{1}{1-2\epsilon} \ell_\IGW(\tilde{\gamma})^2+\frac{1}{2}\big(1+16 (M_2(\rho_1)+\epsilon^2 d)^2\big) d\epsilon   \\
        &\leq \frac{1}{1-2\epsilon} \left(1+4\sqrt{\frac{d}{c}} \epsilon + 4\frac{d}{c}\epsilon^2\right)\ell_\IGW(\rho)^2+ \big(1+16 (M_2(\rho_0)\vee M_2(\rho_1)+\epsilon^2 d)^2\big) d\epsilon\\
        &\leq \ell_\IGW(\rho)^2+O(\epsilon).\qed
    \end{align*}

\section{Additional Experiments for \cref{sec:experiment}}\label{appen:additional_experiments}

We provide additional numerical experiments, including the gradient flow of potential, interaction, and entropy starting from different shapes, as well as more shape matching examples. Due to space considerations, we provide these in \url{https://github.com/ZhengxinZh/IGW/blob/main/additional_experiments/additional_experiments.pdf}, with the numbering scheme continued therein. The gradient flows are illustrated in Fig. 9, with the initial distributions following the shape of an ellipse 9a, 9b, 9c; square 9d, 9e, 9f; two moons 9g, 9h, 9i; two circles 9j, 9k, 9l; and the infinity symbol 9m, 9n, 9o. The flow matchings results are given in Fig. 10, including cat to rotated cat 10a; heart to rotated heart 10b; ellipse to rotated ellipse 10c; cat to heart 10d; and cat to rotated heart 10e.

\section*{Acknowledgements} Z. Goldfeld is partially supported by NSF grants CAREER CCF-2046018, DMS-2210368, and CCF-2308446, and the IBM Academic Award. B. K. Sriperumbudur is partially supported by the NSF CAREER award DMS-1945396. This work was initiated during Z. Zhang's internship at the MIT-IBM Watson AI Lab.

\bibliographystyle{plain}
\bibliography{references}

\end{document}